\documentclass[12pt]{article}
\usepackage{amsthm}
\usepackage{varioref}
\usepackage{amsmath,amstext,amsthm,a4,amssymb,amscd}   
\usepackage[T1]{fontenc}
\usepackage[utf8]{inputenc}
\usepackage{lmodern}
\usepackage{array}
\usepackage{latexsym}
\usepackage{fancyhdr}
\usepackage{mathrsfs}\let\cal\mathscr
\usepackage[all]{xy}
\usepackage{makeidx}   
\usepackage{color}
\usepackage{pifont}
\usepackage{amsfonts,amssymb}
\usepackage{hyperref}
\usepackage{cleveref}
\usepackage[numbers,square]{natbib}
\usepackage[disable]{todonotes}
\usepackage{verbatim}
\usepackage{relsize}

\usepackage{avant}
\usepackage[T1]{fontenc}      

\usepackage{amsfonts,amssymb}  
 
\usepackage{graphicx}

\usepackage{natbib}
\newcommand \Om {\Omega}
\newcommand \om {\omega}

\renewcommand \leq {\leqslant}
\renewcommand \geq {\geqslant}

\DeclareMathOperator{\Vol}{Vol}
\DeclareMathOperator{\End}{End}

\DeclareMathOperator{\Tr}{Tr}

\DeclareMathOperator{\Ker}{Ker}

\DeclareMathOperator{\GL}{GL}
\DeclareMathOperator{\Spec}{Spec}

\DeclareMathOperator{\Lie}{Lie}

\DeclareMathOperator{\Ric}{Ric}

\DeclareMathOperator{\Aut}{Aut}
\DeclareMathOperator{\Herm}{Herm}

\DeclareMathOperator{\Prod}{Prod}
\DeclareMathOperator{\Met}{Met}

\DeclareMathOperator{\FS}{FS}

\DeclareMathOperator{\Fut}{Fut}

\def\id{{1\hskip-2.5pt{\rm l}}}

\newcommand \dbar {\overline{\partial}}

\newcommand \< {\mathcal{h}}
\renewcommand \> {\mathcal{i}}

\newcommand \cinf {\CC^\infty}

\newcommand \Id {{\rm Id}}

\renewcommand \epsilon {\varepsilon}

\newcommand \CC {{\cal C}}
\newcommand \BB {{\cal B}}

\newcommand \EE {{\cal E}}

\newcommand \HH {{\cal H}}

\newcommand \s {\textbf{s}}

\def\cL{\mathscr{L}}

\def\cB{\mathcal{B}}

\def\Re{{\rm Re}}
\def\Im{{\rm Im}}

\def\cL{\mathscr{L}}

\newcommand{\til}[1]{\widetilde{#1}}

\newcommand \dt {\frac{\partial}{\partial t}}

\newcommand \R {\mathbb R}

\newcommand \C {\mathbb C}

\newcommand \N {\mathbb N}

\newcommand \fl {\rightarrow}

\newcommand \ignore[1] {}

\theoremstyle{plain}
\newtheorem{theorem}{Theorem}[section]
\newtheorem{lem}[theorem]{Lemma}
\newtheorem{cor}[theorem]{Corollary}
\newtheorem{prop}[theorem]{Proposition}
\theoremstyle{definition}
\newtheorem*{ackn*}{Acknowledgments}
\newtheorem{defi}[theorem]{Definition}

\newtheorem{rmk}[theorem]{Remark}

\numberwithin{equation}{section}

\crefname{equation}{}{}
\crefname{lem}{Lemma}{Lemmas}
\crefname{theorem}{Theorem}{Theorems}
\crefname{cor}{Corollary}{Corollaries}
\crefname{ex}{Example}{Examples}
\crefname{defi}{Definition}{Definitions}
\crefname{prop}{Proposition}{Propositions}
\crefname{section}{Section}{Sections}
\crefname{subsection}{Section}{Sections}
\crefname{rmk}{Remark}{Remarks}
\crefname{nota}{Notation}{Notations}

\hypersetup{
    colorlinks=true,
    linkcolor=red,
    filecolor=magenta,      
    urlcolor=cyan,
}

\begin{document}

\title{\bf{Balanced metrics for Kähler-Ricci solitons and
quantized Futaki invariants}}
\author{Louis IOOS$^1$}
\date{}
\maketitle
\newcommand{\Addresses}{{
  \bigskip
  \footnotesize

  \textsc{Philipps-Universität Marburg, Hans-Meerwein-Strasse 6, 35043 Marburg,
Germany}\par\nopagebreak
  \textit{E-mail address}: \texttt{ioos@mathematik.uni-marburg.de}\par\nopagebreak
  \textit{Website}: \texttt{louisioos.github.io}
}}
\footnotetext[1]{Partially supported by the European Research Council Starting grant 757585}

\begin{abstract}
We show that a Kähler-Ricci soliton on a Fano manifold can always be smoothly approximated by a sequence of relative anticanonically balanced
metrics, also called quantized Kähler-Ricci solitons. The proof uses
a semiclassical estimate on the spectral gap of
an equivariant Berezin transform
to extend a strategy due to Donaldson, and can be seen as the quantization
of a method due to Tian and Zhu, using quantized Futaki invariants as obstructions for quantized Kähler-Ricci solitons. As corollaries,
we recover the uniqueness of
Kähler-Ricci solitons up to automorphisms, and show how our result
also applies to Kähler-Einstein Fano manifolds with general
automorphism group.
\end{abstract}

\section{Introduction}

The purpose of this paper is to use a relative extension of the notion of
balanced metrics, first introduced by Donaldson in \cite{Don01}, in order
to approximate Kähler-Ricci solitons on Fano manifolds.

Recall that a compact complex manifold $X$ is a \emph{Fano manifold}
if its anticanonical line bundle $L:=\det(T^{(1,0)}X)$ is
\emph{ample}. Thanks to a classical theorem of Kodaira,
this means that $L$ admits a
\emph{positive Hermitian metric}
$h\in\Met^+(L)$, so that its Chern curvature
$R_h\in\Om^2(X,\C)$ induces a Kähler form
on $X$ via the formula
\begin{equation}\label{preq}
\om_h:=\frac{\sqrt{-1}}{2\pi}R_h\,.
\end{equation}
On the other hand, the group $\Aut(X)$ of holomorphic diffeomorphisms of $X$
is a finite dimensional complex Lie group, inducing a complex
embedding of its Lie algebra $\Lie\Aut(X)$
into the Lie algebra $\cinf(X,TX)$ of real vector fields over $X$,
endowed with the complex structure $J\in\End(TX)$.
A Kähler form $\om_h\in\Om^2(X,\R)$
is called a \emph{Kähler-Ricci soliton} with respect to
$\xi\in\Lie\Aut(X)$ if it satisfies $L_{J\xi}\om_h=0$ and
\begin{equation}\label{KRdef}
\Ric(\om_h)-\om_h=L_{\xi}\,\om_h\,,
\end{equation}
where $\Ric(\om_h)\in\Om^2(X,\R)$ denotes the Ricci form of $\om_h$,
and $L_\eta$ denotes the Lie derivative along $\eta\in\Lie\Aut(X)$.
%and $\xi^{1,0}\in\cinf(X,T^{(1,0)}X)$ denotes the $(0,1)$-part of
%$\xi$.
This definition coincides with the definition of Tian and Zhu
in \cite{TZ00}.
%who showed that if such a Kähler-Ricci soliton exists,
%then it is unique up to the identity
%component $\Aut_0(X)$ of $\Aut(X)$.
In the case $\xi=0$, we recover
the notion of a \emph{Kähler-Einstein metric}.

Fix now $p\in\N^*$, and consider the space $H^0(X,L^p)$ of holomorphic
sections of the $p$-th tensor power $L^p:=L^{\otimes p}$.
Let $h^p\in\Met^+(L^p)$ be a positive Hermitian metric on $L^p$, and
consider the induced $L^2$-Hermitian inner product $L^2(h^p)$
defined on $s_1,\,s_2\in H^0(X,L^p)$ by
\begin{equation}\label{L2intro}
\<s_1,s_2\>_{L^2(h^p)}:=\int_X\,\<s_1(x),s_2(x)\>_{h^p}\,d\nu_{h}(x)\,,
\end{equation}
where $d\nu_h$ is the \emph{anticanonical volume form} associated with
the induced metric $h\in\Met^+(L)$ on $L$,
defined over any contractible open subset
$U\subset X$ by the formula
\begin{equation}\label{dnucandef}
d\nu_h:=\sqrt{-1}^{\,n^2}\frac{\theta\wedge\bar{\theta}}
{~~~|\theta|_{h^{-1}}^2}\,,
\end{equation}
independent of $\theta\in\cinf(U,\det(T^{(1,0)*}X))$ non-vanishing, where
$h^{-1}\in\Met(L^*)$ denotes the dual Hermitian metric on
$L^*=\det(T^{(1,0)*}X)$.
On the other hand, by definition of $L$ as an ample line bundle,
it induces an embedding of $X$ into the projective space
of hyperplanes in $H^0(X,L^p)$ for any $p\in\N^*$ large enough,
called Kodaira embedding. Given a Hermitian inner product
$H$ on $H^0(X,L^p)$, we can then endow $L^p$ with the
positive Hermitian metric
$\FS(H)\in\Met^+(L^p)$ induced by the Fubini-Study metric.
%Via this embedding, $L^p$ is identified
%with the restriction of the dual tautological
%line bundle, an
A positive Hermitian metric $h^p\in\Met^+(L^p)$
is called
\emph{anticanonically balanced relative to} $\xi\in\Lie\Aut(X)$
if it satisfies $L_{J\xi}\om_{h^p}=0$ and
\begin{equation}\label{relbaldef}
\om_{h^p}=\phi_{\xi/2p}^*\,\om_{\FS(L^2(h^p))}\,,
\end{equation}
where $\phi_{\eta}\in\Aut(X)$ exponentiates $\eta\in\Lie\Aut(X)$.
In the case $\xi=0$, we recover the usual notion of
an anticanonically balanced metrics, introduced by Donaldson in
\cite[\S\,2.2.2]{Don09}.

The main result of this paper is the following Theorem, which we prove in \cref{proofsec}.
For any $m\in\N$, let $|\cdot|_{\CC^m}$ be a fixed
$\CC^m$-norm on $\Om^2(X,\R)$, and write $\Aut_0(X)$ for the identity
component of $\Aut(X)$.
\begin{theorem}\label{mainth}
Let $\om_{h_\infty}\in\Om^2(X,\R)$ be a Kähler-Ricci soliton with respect to
$\xi_\infty\in\Lie\Aut(X)$.
Then for any $m\in\N$, there exists $C_m>0$ and anticanonically
balanced metrics $h^p\in\Met^+(L^p)$ relative to $\xi_p\in\Lie\Aut(X)$
for all $p\in\N^*$ big enough such that
\begin{equation}\label{mainthfla}
\xi_p\xrightarrow{p\to+\infty}\xi_\infty\quad~\text{and}\quad~
\left|\,\frac{1}{p}\om_{h^p}-\om_{h_\infty}\,\right|_{\CC^m}
\leq \frac{C_m}{p}\,.
\end{equation}
Furthermore, if $\til{h}^p\in\Met^+(L^p)$ is another anticanonically
balanced metric relative to $\til{\xi}_p\in\Lie\Aut(X)$ for some $p\in\N^*$
big enough, then there exists $\phi\in\Aut_0(X)$ such that
$\phi^*\til{\xi}_p=\xi_p$ and $\phi^*\om_{\til{h}^p}=\om_{h^p}$.
\end{theorem}
\cref{mainth} answers a question of Donaldson
in \cite[\S\,2.2.2]{Don09}. This question has also been studied in
previous works of Berman and Witt Nyström in \cite[Th.\,1.7]{BW14}
and Takahashi in \cite[Th.\,1.2]{Tak15}, where the convergence
in \cref{mainth} is
established in the weak sense of currents, and
under the assumption that
%$(X,\xi_\infty)$ is \emph{analytically strongly $K$-polystable},
%meaning that
a modified Ding functional over the infinite dimensional space
$\Met^+(L)$ is coercive modulo $\Aut_0(X)$.
%We note that this assumption might actually
%follow from the existence of a Kähler-Ricci soliton, but as far as we know,
%this has not been established so far. The fact that our
%\cref{mainth} does not need this assumption
%might give a hint towards this folklore conjecture,
%and more generally,
By contrast, our proof closely follows the finite dimensional method
of Donaldson in \cite{Don01}, and relies on methods
of \emph{Berezin-Toeplitz quantization}.
We hope that our approach can help to shed light on the different notions
of stability in this context, following the work of Saito and Takahashi
in \cite{ST19}.
Relative anticanonically balanced
metrics were introduced in \cite[\S\,4.2.2]{BW14} under the name
of \emph{quantized Kähler-Ricci solitons}.

As a straightforward consequence of \cref{mainth}, we get an alternative
proof of the following result of Tian and Zhu in \cite[Th.\,1.1]{TZ00} and
\cite[Th.\,3.2]{TZ02}, which does not rely on solving a Monge-Ampère
equation.

\begin{cor}\label{TZintro}
Let $\om_{h},\,\om_{\til{h}}\in\Om^2(X,\R)$ be Kähler-Ricci solitons with respect to $\xi,\,\til{\xi}\in\Lie\Aut(X)$ respectively.
Then there exists $\phi\in\Aut_0(X)$ such that
$\phi^*\til{\xi}=\xi$ and
$\phi^*\om_{\til{h}}=\om_{h}$.
\end{cor}

Let us point out that the coercivity assumption used in the proofs of
\cite[Th.\,1.7]{BW14} and \cite[Th.\,1.2]{Tak15} was shown to be a
consequence of the existence of a Kähler-Ricci soliton
by Darvas and Rubinstein in \cite[Th.\,8.1]{DR17}, but this last
result actually uses \cref{TZintro}, so that
this does not lead to an alternative proof.

The proof of \cref{mainth} 
%proof
%of Tian and Zhu
%in \cite{TZ00,TZ02} of the uniqueness of Kähler-Ricci
%solitons up to $\Aut_0(X)$,
follows the general strategy of Donaldson
in \cite{Don01}, who established an analogue of \cref{mainth}
in the case of $\Aut(X)$ discrete, showing that a polarized
Kähler metric of constant scalar curvature can always
be approximated by a sequence of \emph{balanced metrics},
defined as in \cref{relbaldef} for $\xi=0$
using the usual Liouville form instead of
the anticanonical volume form \cref{dnucandef}
in the $L^2$-Hermitian product \cref{L2intro}.
Specifically, our \cref{approxbal} uses Donaldson's method of
constructing approximately balanced metrics
via the asymptotic expansion of the \emph{Bergman kernel}
along the diagonal, which we recall in \cref{Bergdiagexp},
and our \cref{Donlem} is a straightforward adaptation of
a fundamental Lemma of Donaldson
on the convergence of the gradient flow of the norm squared
of the associated \emph{moment map} close to a zero.

The main difficulty lies instead in
the most delicate part of Donaldson's strategy, given in
\cite[\S\,3.2]{Don01} in order to establish the key estimate
\cite[Cor.\,22]{Don01}. This part of the proof gives an estimate
from below of the derivative of the associated moment map,
and has already been improved and
clarified by Phong and Sturm in \cite[Th.\,2]{PS04}.
In the situation of \cite{Don01,PS04}, the derivative of the moment
map has a natural geometric interpretation,
and this gives a natural approach to estimate the lower bound.
Unfortunately, this geometric interpretation does not carry directly
to the anticanonical case considered in \cref{mainth}.
This difficulty was overcome only recently by Takahashi
in \cite[Th.\,1.3]{Tak21}, who established \cref{mainth} in the case
of $\Aut(X)$ discrete. The further extension of
this geometric interpretation to the case of general
$\Aut(X)$ is an interesting problem, but this has not been
achieved yet.

The main novelty of our method is to replace
this geometric interpretation by the use of the asymptotics of the
\emph{spectral gap of the Berezin transform}, which where
first established in \cite[Th.\,3.1]{IKPS19} and which we extend
to the equivariant case in \cref{Bpgap}. These asymptotics are used
in a crucial way in \cref{existencesec} to obtain the necessary
estimate from below of the derivative of the appropriate moment map
in this context.
More precisely, we relate in \cref{dmu} the derivative of the moment map
with the equivariant \emph{Berezin-Toeplitz quantum channel} of \cref{quantchanrel}, and we then use our asymptotics of the spectral
gap to give an estimate of the Berezin-Toeplitz quantum channel in 
\cref{quantchanbnd}. As shown in \cref{dmuapprox}, this produces
the desired lower bound,
and shows as a by-product that our estimate is optimal.
This extends the strategy used in \cite{Ioo20} to establish \cref{mainth}
in the case of $\Aut(X)$ discrete, and allows to bypass the geometric
interpretation mentionned above, which has not yet been
worked out in the case of Kähler-Ricci solitons.
Note on the other hand that the asymptotics of the spectral gap
are based on the asymptotic expansion of the Berezin transform
recalled in \cref{KS}, which is in turn a consequence
of the asymptotic expansion of the Bergman kernel \emph{outside}
the diagonal. 
Our approach thus gives a unified way,
entirely based on Berezin-Toeplitz quantization, to treat
both the construction of approximately balanced
metrics
and the lower bound of the derivative of the moment map, instead of 
using an additional delicate geometric interpretation for the later.

The advantage of our method of proof of \cref{mainth} is
that it can be adapted in a systematic way to various
choices of a volume form in the Hilbert product
\cref{L2intro}, leading to various different notions
of balanced metrics. 
In \cite[\S\,2]{Ioo20}, we described a general set-up
in which our method can be applied, which includes the original
notion of balanced metrics of \cite{Don01}, but also
the \emph{$\nu$-balanced metrics} on
Calabi-Yau manifolds and the
\emph{canonically balanced metrics} on manifolds
with ample canonical line bundle,
introduced by Donaldson in \cite{Don09},
as well as the notion of twisted balanced metrics studied by Keller
in \cite{Kel09b} and Dervan in \cite[\S\,2.2]{Der16}.
It can also be extended to the case of balanced metrics over vector bundles,
following Wang in \cite{Wan05}, and to the case of coupled Kähler-Einstein
metrics as in \cite{Tak21}.
In all these cases, a new geometric interpretation
as in \cite{PS04} was needed to adapt the proof of
\cite{Don01} successfully. By contrast, our proof gives a general method
to deal with this key step, using the asymptotics of the spectral gap of the
corresponding Berezin transform,
as given in \cref{BTsecT} following the strategy
described in \cite[\S\,3]{IKPS19}.

In \cref{quantAutsec}, we study the quantization of the action of
the automorphism group to establish a quantized counterpart of the
method of Tian and Zhu in \cite{TZ02}. Namely, we show
that the holomorphic vector fields $\xi_p\in\Lie\Aut(X)$ of \cref{mainth}
are determined a priori for all $p\in\N^*$,
regardless of the existence of a relative
balanced metric. As a first step, we show in \cref{xipmin} that for any
$p\in\N^*$ large enough, there exists a unique vector field
$\xi\in\Lie\Aut(X)$
such that the associated \emph{quantized Futaki invariant}
$\Fut^\xi_p:\Lie\Aut(X)\to\C$ vanishes. Following Berman and Witt Nyström in
\cite[\S\,4.1.1]{BW14}, it is defined for any $\eta\in\Lie\Aut(X)$
by the formula
\begin{equation}\label{quantFutdefintro}
\Fut^\xi_p(\eta):=\Tr\left[L_\eta e^{L_\xi/p}\right]\,,
\end{equation}
where $L_\eta\in\End(H^0(X,L^p))$ denotes the natural action of
$\eta$ on the holomorphic sections of the $p$-th tensor power
of the anticanonical line bundle $\det(T^{(1,0)}X)$.
As a second step, we show in
\cref{Futpropintro} that if
there exists an anticanonically balanced metric $h^p\in\Met^+(L^p)$
relative to $\xi\in\Lie\Aut(X)$, then the associated quantized Futaki
invariant $\Fut^\xi_p:\Lie\Aut(X)\to\C$
vanishes identically.
%This also follows from \cite[Prop.\,4.7,\,4.9]{BW14},
%and
This can be seen as the quantization of
\cref{TZKR}, which is due to Tian and Zhu in \cite[Prop.\,3.1]{TZ02},
showing that the vector field
$\xi_\infty\in\Lie\Aut(X)$ of \cref{mainth}
is determined a priori as the unique
holomorphic vector field for which the associated
\emph{modified Futaki invariant} vanishes, regardless of the
existence of a Kähler-Ricci soliton.
This characterization of the vector fields
$\xi_p$ for all $p\in\N^*$
plays a crucial role in our proof of \cref{mainth}, and
has no analogue in Donaldson's approach in \cite{Don01}, since
it assumes $\Aut(X)$ discrete. In particular, we show
in \cref{xipmin} that the vector fields $\xi_p$ admit an asymptotic expansion
as $p\to+\infty$ with highest order coefficient equal to $\xi_\infty$,
which shows the first identity of \cref{mainthfla} and is used
in a crucial way for the construction in \cref{approxbal}
of approximately balanced metrics.

As remarked in \cite[Rmk.\,4.8]{BW14}, it is a consequence of
the equivariant Riemman-Roch formula that the vanishing of the
quantized Futaki invariants \cref{quantFutdefintro} relative to $\xi=0$
for all $p\in\N^*$ big enough is equivalent to the
vanishing of all \emph{higher order Futaki invariants} \cite{Fut04} of $X$.
We thus recover the fact, first established by Saito and Takahashi
in \cite[Lem.\,3.2]{ST19}, that the higher order Futaki invariants
are an obstruction for the existence of anticanonically balanced metrics
in the usual sense, for all $p\in\N^*$ big enough.
On the other hand, Ono, Sano and Yotsutani exhibit in \cite[Th.\,1.5]{OSY12}
an example of a toric Kähler-Einstein Fano manifold with non-vanishing higher order Futaki invariants.
This example thus implies the following corollary of
\cref{mainth}.

\begin{cor}\label{corintro}
There is a Fano manifold $X$ with Kähler-Einstein metric
$\om_{h_\infty}\in\Om^2(X,\R)$ such that the anticanonically
balanced metrics $h^p\in\Met^+(L^p)$
relative to $\xi_p\in\Lie\Aut(X)$ of \cref{mainth}
satisfy $\xi_p\neq 0$ for all $p\in\N^*$ big enough.
\end{cor}

In particular, we recover the fact from \cite[Ex.\,5.6]{ST19}
that the toric example of  \cite{OSY12}
does not admit any anticanonically balanced metric
in the usual sense, for all $p\in\N^*$ big enough.
This was also established in \cite[Cor.\,1.1]{Tak15}
in the sense of currents and
under a coercivity assumption on the Ding functional.
%If one could show that this example is strongly analytically
%$K$-polystable,
%\cref{corintro} would also follow from \cite[Cor.\,1.1]{Tak15}.
\cref{corintro} illustrates the fact that \cref{mainth}
is already interesting in the case of Kähler-Einstein metrics.
In fact, it is shown in \cite{BBGZ13,Ioo20}
that a Kähler-Einstein metric
on a Fano manifold $X$ with $\Aut(X)$ discrete can always be approximated
by anticanonically balanced metrics. \cref{corintro} then shows that
the assumption of $\Aut(X)$ discrete is necessary for such a result to
hold and that \cref{mainth} extends this result using relative
anticanonically balanced metrics.

In \cite{RTZ20}, Rubinstein, Tian
and Zhang introduce a notion of anticanonically balanced metrics
depending on a parameter $\delta>0$, which coincides with the
usual notion of an anticanonically balanced metric when $\delta=1$.
In \cite[Prop.\,5.10]{RTZ20}, they show that these balanced metrics
with $\delta<1$ can be used to approximate Kähler-Einstein metrics
on Fano manifolds with general $\Aut(X)$ as $\delta\to 1$,
with convergence in the sense of currents.
In the same way as in \cref{mainth},
our method readily extends to establish smooth convergence.
%and \cref{corintro} shows that the choice of $\delta<1$ and
%$\delta\to 1$ in \cite[Prop.\,5.10]{RTZ20} is necessary and optimal.
%Our method of proof of \cref{mainth} should also extend to this case.
%get rid of the extra assumption on the energy functional
%and get smooth convergence in that
%case, showing that these $\delta$-balanced metrics are also a convenient
%replacement for anticanonically balanced metrics in the case of non-discrete
%automorphism group.
On the other hand, our notion of relative anticanonically balanced metric
coincides with the
quantized Kähler-Ricci solitons of \cite{BW14}, where the more general
notion of a \emph{Kähler-Ricci $g$-soliton} is considered.
Although, we restrict to usual Kähler-Ricci solitons for simplicity,
our proof extends to this more general case without any difficulty.

A relative version of balanced metrics has first been introduced by
Mabuchi in \cite{Mab04} to study \emph{extremal Kähler metrics},
and can be seen as the relative version of constant scalar curvature
metrics in the case when $\Aut(X)$ is not discrete.
Instead, the notion
of relative anticanonically balanced metrics used in
this paper is an analogue of the \emph{relative balanced metrics}
introduced by Sano and Tipler in \cite{ST17}, defined as in
\cref{relbaldef} with
the anticanonical volume form \cref{dnucandef} replaced
by the usual Liouville form in the
$L^2$-Hermitian product \cref{L2intro}.
\cref{mainth} is then an anticanonical version of \cite[Th.\,1.1]{ST17},
where the extremal metric is replaced by a Kähler-Ricci soliton,
and \cref{xipexp,approxbal}
%on the asymptotic expansion of the vector fields $\xi_p$ as $p\to+\infty$,
%as well as the proof of \cref{approxbal} on the construction of
%approximately balanced metrics
were inspired by \cite[Lem.\,4.5, Th.\,5.5]{ST17}.
Closely related notions of relative balanced metrics as quantizations
of extremal Kähler metrics have also been introduced by
Hashimoto in \cite{Has21}, Mabuchi in \cite{Mab18} and Seyyedali in
\cite{Sey17}. All these works establish an analogue of
\cref{mainth} for extremal metrics, extending the geometric
interpretation of \cite{Don01,PS04} for the lower bound of
the moment map in their respective relative settings.
We refer to \cite[\S\,6]{Has21}
for a detailed comparison between these different notions.
We only point out here that the approach of \cite{Has21} is a
quantization of the fact that extremal Kähler metrics are critical
points of the \emph{Calabi functional}. In particular, this approach
does not extend to our case, since Kähler-Ricci solitons are not critical
points of the anticanonical analogue of the Calabi functional, which is
the \emph{Ding functional}.

%As pointed out by Apostolov and Huang
%in \cite{AH?}, this notion does not coincide with the notio
%We point out that, although their notion coincides with the notion introduced
%by, this does not remain true in the
%anticanonical setting. In fact, the approach of \cite{Sey17} is a
%quantization of the fact that extremal Kähler metrics are critical
%points of the \emph{Calabi functional}. In particular, this approach
%does not extend to our case, since Kähler-Ricci solitons are not critical
%points of the anticanonical analogue of the Calabi functional, which is
%the \emph{Ding functional}.

The theory of Berezin-Toeplitz
quantization was first developed by Bordemann,
Meinrenken and
Schlichenmaier in \cite{BMS94}, using the work of
Boutet de Monvel and Sjöstrand on the Szegö kernel in \citep{BdMS75} and the theory of
Toeplitz structures of Boutet de Monvel and Guillemin in
\cite{BdMG81}. This paper is based instead on the theory of
Ma and Marinescu in \cite{MM08b}, using the off-diagonal
asymptotic expansion of the Bergman kernel established
by Dai, Liu and Ma in \cite[Th.\,4.18']{DLM06}.
A comprehensive introduction for this theory can be found in
\cite{MM07}. The point of view of quantum measurement
theory on Berezin-Toeplitz quantization adopted
in this paper has been advocated by Polterovich
in \cite{Pol12,Pol14}.

\begin{ackn*}
The author wishes to thank Pr. Xiaonan Ma for his constant support and Pr. Leonid Polterovich for helpful discussions. The author also wishes to thank the anonymous referees for useful comments and suggestions.
This work was supported
by the European Research Council Starting grant 757585.
\end{ackn*}

\section{Setting}
\label{setting}

Let $X$ be a compact complex manifold with complex structure
$J\in\End(TX)$, and write
\begin{equation}\label{splitc}
TX_\C=T^{(1,0)}X\oplus T^{(0,1)}X
\end{equation}
for the splitting of the complexification of the tangent bundle $TX$ of $X$
into the eigenspaces of $J$ corresponding to the eigenvalues $\sqrt{-1}$ and $-\sqrt{-1}$ respectively. For any vector field $\xi\in\cinf(X,TX)$, we will
write $\xi^{1,0},\,\xi^{1,0}\in\cinf(X,T^{(1,0)}X)$ for its
components with respect to this splitting.

In this paper, we will always assume that $X$ is a Fano manifold,
so that the space
$\Met^+(L)$ of positive Hermitian metrics on $L:=\det(T^{(1,0)}X)$
is not empty.
For any $h\in\Met^+(L)$ and $p\in\N^*$, we write
$h^p\in\Met^+(L^p)$ for the induced positive Hermitian metric
on the $p$-th tensor power $L^p$. Conversely, any
$h^p\in\Met^+(L^p)$ uniquely determines
a positive Hermitian metric $h\in\Met^+(L)$.
We will also write $\Met(L)$ for the space of Hermitian metrics on $L$.
We write $\cinf(X,L^p)$
for the space of smooth sections of $L^p$
and $H^0(X,L^p)\subset\cinf(X,L^p)$
for the subspace of
holomorphic sections of $L^p$ over $X$.

%This means that $K_X^*$ admits a \emph{positive Hermitian metric}
%$h\in\Met^+(L)$, so that its \emph{Chern curvature}
%$R_h\in\Om^2(X,\C)$ induces a \emph{Kähler form}
%on $X$ via the formula
%\begin{equation}\label{preq}
%\om_h:=\frac{\sqrt{-1}}{2\pi}R_h\,.
%\end{equation}
Recall that for any $h\in\Met^+(L)$, the $2$-form
$\om_h\in\Om^2(X,\R)$ defined by formula \cref{preq}
is a \emph{Kähler form}, meaning that
the following formula
defines a Riemannian metric on $X$,
\begin{equation}\label{gTXintro}
g^{TX}_h:=\om_h(\cdot,J\cdot)\,.
\end{equation}
%The \emph{anticanonical volume form}
%induced by $h\in\Met^+(L)$ is the volume form $d\nu_h$
%defined over any contractible open subset
%$U\subset X$ by the formula
%\begin{equation}\label{dnucandef}
%d\nu_h:=\sqrt{-1}^{\,n^2}\frac{\theta\wedge\bar{\theta}}
%{~~~|\theta|_{h^{-1}}^2}\,,
%\end{equation}
%for any non-vanishing $\theta\in\cinf(U,K_X)$, where
%$h^{-1}$ denotes the Hermitian metric on $K_X$ induced by
%$h\in\Met^+(K_X^*)$.
Note by definition \cref{dnucandef} of the associated
anticanonical volume $d\nu_h$ that
for any $f\in\cinf(X,\R)$, we have
%such that $e^f h\in\Met^+(L)$, the definition \cref{dnucandef} implies that
\begin{equation}\label{dnuef/dnu}
d\nu_{e^fh}=e^{f}\,d\nu_h\,.
\end{equation}

\subsection{Action of the automorphism group}
\label{Autsec}

Recall that the group $\Aut(X)$ of holomorphic diffeomorphisms of
$X$ is a finite dimensional complex Lie group, and so that there is
a natural
embedding $\Lie\Aut(X)\subset\cinf(X,TX)$.
For any $\xi\in\Lie\Aut(X)$, we write
$\phi_{t\xi}\in\Aut(X),\,t\in\R$, 
for the flow generated by $\xi\in\Lie\Aut(X)$.
%, so that
%\begin{equation}
%\left\{
%\begin{array}{l}
%  \dt\phi_{t\xi}=\xi\,, \\
%  \\
% \phi_0 = \Id_X\,.
%\end{array}
%\right.
%\end{equation}
The holomorphic
action of $\Aut(X)$ on $X$ lifts naturally to $L:=\det(T^{(1,0)}X)$,
and for any $\xi\in\Lie\Aut(X)$, we
write $L_\xi$ for the induced differential operator acting on
a smooth section $s\in\cinf(X,L)$ by
\begin{equation}\label{Lxidef}
L_\xi s:=\dt\Big|_{t=0}\,\phi_{t\xi}^*\,s\,.
\end{equation}
Recall also that definition \cref{dnucandef} of the
anticanonical volume form $d\nu_h$ associated with $h\in\Met(L)$
does not depend on
$\theta\in\cinf(U,\det(T^{(1,0)*}X))$, so that for all $t\in\R$,
\begin{equation}\label{etxi*dnu}
\phi_{t\xi}^*\,d\nu_h
=\sqrt{-1}^{\,n^2}\frac{\phi_{t\xi}^*\,\theta\wedge
\phi_{t\xi}^*\,\bar{\theta}}
{|\phi_{t\xi}^*\,\theta|_{\phi_{t\xi}^*\,h^{-1}}^2}=d\nu_{\phi_{t\xi}^*h}\,.
\end{equation}
For any $h\in\Met^+(L)$, write $\nabla^h$
for the \emph{Chern connection} of $(L,h)$.
We then have the following complex
version of the \emph{Kostant formula}.

\begin{defi}\label{Kostantdef}
For any $h\in\Met^+(L)$, the associated \emph{holomorphy potential}
of $\xi\in\Lie\Aut(X)$ is the function $\theta_h(\xi)\in\cinf(X,\C)$
defined for any $s\in\cinf(X,\C)$ by the formula
\begin{equation}\label{Kostant}
\theta_h(\xi)\,s:=L_\xi\,s-\nabla^{h}_\xi\,s\,.
\end{equation}
\end{defi}

In the same way as the usual Kostant formula for moment maps,
formula \cref{Kostant} gives a well-defined
scalar function $\theta_h(\xi)\in\cinf(X,\C)$, which
via formula \cref{preq} for the Kähler form $\om_h\in\Om^2(X,\R)$
satisfies
\begin{equation}\label{holpot}
\iota_{\xi^{1,0}}\,\om_h=\frac{\sqrt{-1}}{2\pi}\dbar\,\theta_h(\xi)\,.
\end{equation}
Thanks to
the Kodaira vanishing theorem (see for instance
\cite[Prop.\,3.72,\,(1)]{BGV04} with $\cL:=K_X$),
recall that a Fano manifold $X$
satisfies $H^1(X,\C)=0$. Then as explained in \cite{FM95}, 
a fundamental result of Fujiki \cite{Fuj78} implies
in that case that $\Aut_0(X)$
is a linear complex algebraic group. In particular,
it includes the complexification
$K_\C\subset\Aut_0(X)$ of any connected compact subgroup $K\subset\Aut_0(X)$,
and we have a natural decomposition
\begin{equation}\label{LieKC}
\Lie K_\C=\Lie K\oplus\sqrt{-1}\Lie K\subset\Lie\Aut(X)\,.
\end{equation}
The following proposition gives some basic properties of the
holomorphy potential.

\begin{prop}\label{thetaxiprop}
For any $h\in\Met^+(K_X^*)$ and $\xi\in\Lie\Aut(X)$, we have
\begin{equation}\label{dtetxi*hfla}
\dt\Big|_{t=0}\,\phi_{t\xi}^*\,h=-2\,\Re\,\theta_h(\xi)\,h\,.
\end{equation}
Furthermore, we have
\begin{equation}\label{inttheta=0}
\int_X\,\theta_h(\xi)\,d\nu_h=0\,,
\end{equation}
and for any $f\in\cinf(X,\R)$, we have
\begin{equation}\label{thetaef}
\theta_{e^{f}h}(\xi)=\theta_h(\xi)-df.\xi^{1,0}\,.
\end{equation}
Finally, if $K\subset\Aut_0(X)$ is a compact subgroup
preserving $h\in\Met^+(L)$, then
$\theta_h:\Lie K_\C\rightarrow\cinf(X,\C)$ is
a $\C$-linear embedding, and for any $\xi\in\sqrt{-1}\Lie K$,
we have $\theta_h(\xi)\in\cinf(X,\R)$.
\end{prop}
\begin{proof}
By definition, the Chern connection
$\nabla^{h}$ of any $h\in\Met(K_X^*)^+$ induces the holomorphic structure
of $L$, while the lift of $\Aut(X)$ to $L:=\det(T^{(1,0)}X)$
is holomorphic.
This implies in particular that for all $h\in\Met(K_X^*)^+$ and
all $\xi\in\Lie\Aut(X)$, we have $L_{\xi^{0,1}}=\nabla^{h}_{\xi^{0,1}}$,
so that \cref{Kostantdef} implies
\begin{equation}\label{thetahol}
\theta_h(J\xi)=\sqrt{-1}\theta_h(\xi)\,.
\end{equation}
By the unitarity of the Chern
connection, for any $s\in\cinf(X,K_X^*)$, \cref{Kostantdef} gives
\begin{equation}\label{dtetxi*h}
\begin{split}
\dt\Big|_{t=0}\,
|s|_{\phi_{t\xi}^*\,h}^2
&=L_\xi\<s,s\>_h-\<L_\xi s,s\>_h-\<s,L_\xi s\>_h\\
&=\<(\nabla_\xi^{h}-L_\xi) s,s\>_h+\<s,(\nabla_\xi^{h}-L_\xi) s\>_h\\
&=-2\,\Re\,\theta_h(\xi)\,|s|_h^2\,,
\end{split}
\end{equation}
This shows the identity \cref{dtetxi*hfla}.
Combining formulas \cref{dnuef/dnu,etxi*dnu,dtetxi*h}, we then get
\begin{equation}
-2\int_X\,\Re\,\theta_h(\xi)\,d\nu_h=\dt\Big|_{t=0}\int_X\,\phi_{t\xi}^*\,d\nu_h
=0\,.
\end{equation}
Using the fact from formula
\cref{thetahol} that $\Im\,\theta_h(\xi)=-\Re\,\theta_h(J\xi)$,
this implies the identity \cref{inttheta=0}.

On the other hand, the identity \cref{thetaef} is an immediate consequence
of \cref{Kostantdef}
and the fact that for any $f\in\cinf(X,\R)$, we have
\begin{equation}
\nabla^{e^fh}=\nabla^h+df.\xi^{1,0}\,.
\end{equation}
Finally, if $K\subset\Aut_0(X)$ preserves $h\in\Met^+(L)$, then
formula \cref{dtetxi*h} implies that for any $\xi\in\Lie K$, we have
$\frac{\sqrt{-1}}{2\pi}\theta_h(\xi)\in\cinf(X,\R)$, while formula \cref{thetahol}
implies that $\theta_h:\Lie K_\C\rightarrow\cinf(X,\C)$
is $\C$-linear for the standard
complex structure on both spaces, and formula \cref{holpot} shows that
$\theta_h(\xi)\equiv 0$ if and only if $\xi=0$ by non-degeneracy of $\om_h$.
This concludes the proof. 
\end{proof}

\begin{rmk}\label{momentrmk}
Note that if $K\subset\Aut_0(X)$ preserves
$h\in\Met^+(L)$, \cref{Kostantdef} for
the map $\frac{\sqrt{-1}}{2\pi}\theta_h:\Lie K\to\cinf(X,\R)$
reduces to the definition of a \emph{moment map}
for the action of $K$ on the symplectic manifold
$(X,\om_h)$ via the usual Kostant formula.
The usual condition for the moment map
is recovered from the real part of formula \cref{holpot} by
$\C$-linearity of $\theta_h:\Lie K_\C\to\cinf(X,\C)$.
\end{rmk}
%
%Note that if $T\subset\Aut(X)$ is a compact torus and
%$h\in\Met^+(L)$ is a $T$-invariant metric, then the Kostant formula
%\cref{Kostant} implies that the map
%\begin{equation}\label{thetamoment}
%\frac{\sqrt{-1}}{2\pi}\theta_h:\Lie T\longrightarrow\cinf(X,\R)
%\end{equation}
%is a \emph{moment map}
%for the action of $T$ on $X$ with respect to the symplectic form $\om_h$,
%and the real part of formula \cref{holpot} reduces to the usual condition
%for a moment map.
%
%Note that the integral identity \cref{inttheta=0} fixes the
%ambiguity of an additive constant in the definition of
%$\theta_h(\xi)\in\cinf(X,\R)$ as an holomorphic potential \cref{holpot}.
%In particular, it differs from the more common convention asking
%for the integral against the \emph{Liouville form} to vanish instead.

We will established the quantized
counterpart of the following result
of Tian and Zhu in \cref{quantFutsec}.

\begin{prop}\label{TZinv}
{\emph{\cite[Proof of Lemma\,2.2, Prop.\,2.1]{TZ02}}}
Let $K\subset\Aut_0(X)$ be a given compact subgroup. Then
there exists a strictly
convex and proper functional $F:\sqrt{-1}\Lie K\to\R$,
such that for any $\xi\in\sqrt{-1}\Lie K$ and $h\in\Met^+(L)$, we have
\begin{equation}\label{FTZ}
F(\xi):=\int_X\,e^{\theta_{h}(\xi)}\,\frac{\om^n_{h}}{n!}\,.
\end{equation}
Furthermore, for any $\xi\in\sqrt{-1}\Lie K$,
the following formula for the associated
\emph{modified Futaki invariant}
$\Fut_{\xi}:\sqrt{-1}\Lie K\to\C$ at $\eta\in\sqrt{-1}\Lie K$
does not depend on $h\in\Met^+(L)$
\begin{equation}\label{FutTZ}
\Fut_\xi(\eta):=\int_X\,\theta_h(\eta)\,e^{\theta_h(\xi)}\,\frac{\om_h^n}{n!}\,.
\end{equation}
Finally, there exists a unique
$\xi_\infty\in\sqrt{-1}\Lie K$
such that $\Fut_{\xi_\infty}:\sqrt{-1}\Lie K\to\C$
vanishes identically, which is given
by the unique minimizer of the functional $F:\sqrt{-1}\Lie K\to\R$ and
satisfies $[\xi_\infty,\eta]=0$ for all $\eta\in\Lie K_\C$.
\end{prop}
%The holomorphic invariant \cref{FutTZ} coincides with the
%\emph{modified Futaki invariant}.
%\todo{ref}
%Note that the statement about \cref{FutTZ} directly follows from the first,
%as the functional \cref{FTZ} is an energy functional for the modified
%Futaki invariant, so that for any $\xi\in\sqrt{-1}\Lie K$ and $\eta\in\Lie K_\C$,
%we have
%\begin{equation}
%\dt\big|_{t=0}\,F(\xi+t\eta)=\Fut_\xi(\eta)\,.
%\end{equation}
%\todo{vraiment utile cette remarque?}

\subsection{Kähler-Ricci solitons}
\label{KRsec}

A Kähler form $\om\in\Om^2(X,\R)$ on a
compact complex manifold $X$ induces a natural
Hermitian metric $h_\om\in\Met(L)$ on $L=\det(T^{(1,0)}X)$,
defined using the anticanonical volume form
\cref{dnucandef} by the formula
\begin{equation}\label{dnuh=omh}
\frac{\om^n}{n!}=d\nu_{h_\om}\,.
\end{equation}
The associated \emph{Ricci form} is defined by the formula
$\Ric(\om):=\frac{\sqrt{-1}}{2\pi}R_{h_\om}$
where $R_{h_\om}\in\Om^2(X,\C)$ is the Chern curvature of the Hermitian
metric $h_\om\in\Met(L)$.
Using Cartan's formula and by definition \cref{holpot} of
the holomorphy potential, we see that
$h\in\Met^+(L)$ induces a
Kähler-Ricci soliton $\om_h\in\Om^2(X,\R)$
with respect to $\xi\in\Lie\Aut(X)$ in the sense of formula
\cref{KRdef} if and only if
\begin{equation}
\Ric(\om_h)=\om_h+\frac{\sqrt{-1}}{2\pi}\partial \dbar\,\theta_h(\xi)\,.
\end{equation}
By definition
of the Chern curvature, this means that
there exists a constant $c>0$ such that $h_{\om_h}=ce^{\theta_h(\xi)}h$,
so that formula \cref{dnuh=omh} shows that
$\om_h\in\Om^2(X,\R)$
is a Kähler-Ricci soliton with respect to $\xi\in\Lie\Aut(X)$
if and only if there is a constant $c>0$ such that
\begin{equation}\label{KRvolfla}
d\nu_h=ce^{\theta_h(\xi)}\frac{\om^n_h}{n!}\,.
\end{equation}
Let now $\om_h\in\Om^2(X,\R)$
be a Kähler-Ricci soliton with respect to $\xi\in\Lie\Aut(X)$,
and recalling formula \cref{gTXintro} for the associated Riemannian metric,
let $K\subset\Aut_0(X)$ be a connected subgroup
of holomorphic isometries of $(X,g_h^{TX})$. As $L_{J\xi}\om_h=0$
by definition, we get that $\xi\in\sqrt{-1}\Lie K$ in the decomposition
\cref{LieKC}.
The following result of Tian and Zhu shows that the modified
Futaki invariant of \cref{TZinv} is an obstruction for the
existence of Kähler-Ricci solitons. It will play a key role in the
construction of approximately balanced metrics in \cref{approxbalsec}.

\begin{prop}\label{TZKR}
{\emph{\cite[Prop.\,1.3]{TZ02}}}
Let $\om_h\in\Om^2(X,\R)$ be a Kähler-Ricci soliton
with respect to $\xi\in\Lie\Aut(X)$,
and let $K\subset\Aut_0(X)$ be a connected subgroup
of holomorphic isometries of $(X,g_h^{TX})$.
Then the associated modified
Futaki invariant $\Fut_{\xi}:\sqrt{-1}\Lie K\to\C$
of \cref{TZinv}
vanishes identically.
\end{prop}

For any compact subgroup $K\subset\Aut_0(X)$,
we write $\Met^+(L)^K$ for the space of $K$-invariant positive
Hermitian metrics, and $\cinf(X,\C)^K$
for the space of $K$-invariant functions
over $X$. For any $h\in\Met^+(L)$, write $\Delta_h$ for the
Riemannian Laplacian
of $(X,g_h^{TX})$ acting on $\cinf(X,\R)$.

\begin{lem}\label{TZopdef}
Let $K\subset\Aut_0(X)$ be a connected compact subgroup and let
$T\subset K$ be the identity component of its center.
For any $h\in\Met^+(L)^K$ and $\xi\in\sqrt{-1}\Lie T$, the operator
$\Delta_h^{(\xi)}$ acting on $f\in\cinf(X,\R)^K$ by the formula
\begin{equation}\label{TZopdeffla}
\Delta_h^{(\xi)}f
:=\frac{1}{4\pi}\Delta_{h}\,f-df.\xi^{1,0}\,,
\end{equation}
is positive and essentially self-adjoint with respect to the scalar product
$\<\cdot,\cdot\>_{L^2(h,\xi)}$
defined on $f,\,g\in\cinf(X,\R)^K$ by the formula
\begin{equation}\label{L2theta}
\<f,g\>_{L^2(h,\xi)}=\int_X\,f\,g\,e^{\theta_h(\xi)}\,\frac{\om_h^n}{n!}\,.
\end{equation}
Furthermore, we have $\Ker\Delta_h^{(\xi)}=\C$.
\end{lem}
\begin{proof}
Fix $h\in\Met^+(L)^K$ and $\xi\in\sqrt{-1}\Lie T$,
and recall that by definition, we have $J\xi\in\Lie T$.
Then all $f\in\cinf(X,\R)^K$ satisfy $df.\xi=2df.\xi^{1,0}$,
and for all
$\eta\in\Lie K$, we have $[\xi,\eta]=0$ and $[\Delta_h,\eta]=0$,
so that $\Delta_h^{(\xi)}$ preserves $\cinf(X,\R)^K$
inside $\cinf(X,\C)$.
Using \cref{thetaxiprop},
the imaginary part of the holomorphy potential equation
\cref{holpot} gives $2\pi\iota_{J\xi}\om_h=-d\theta_h(\xi)$.
Then writing $\<\cdot,\cdot\>_{g_h^{TX}}$ for the pointwise scalar product
on $T^*X$ induced by the Riemannian metric \cref{gTXintro}
and using an integration by part from formulas \cref{TZopdeffla}
and \cref{L2theta}, for any $f,\,g\in\cinf(X,\C)^K$ we get
\begin{equation}
\<\Delta_h^{(\xi)}f,g\>_{L^2(h,\xi)}=
%\int_X\,\left(\frac{\Delta_h}{4\pi}f-df.\xi^{1,0}\right)\,g\,e^{\theta_h(\xi)}\,\frac{\om^n}{n!}\\
%=\frac{1}{4\pi}\int_X\,\<df,dg\>_{g_h^{TX}}
%\,e^{\theta_h(\xi)}\,\frac{\om^n_h}{n!}
%+\frac{1}{4\pi}\int_X\,\<df,d\theta_h(\xi)\>_{g_h^{TX}}\,g\,e^{\theta_h(\xi)}\,\frac{\om^n}{n!}\\
%-\frac{1}{2}\int_X\,df.\xi~g\,e^{\theta_h(\xi)}\,\frac{\om^n_h}{n!}\\
%=
\int_X\,\<df,dg\>_{g_h^{TX}}\,e^{\theta_h(\xi)}\,\frac{\om^n_h}{n!}\,.
\end{equation}
This shows that the operator
$\Delta_h^{(\xi)}$ given by formula \cref{TZopdeffla}
is essentially self-adjoint and positive with respect
to the scalar product $L^2(h,\xi)$ on $\cinf(X,\R)^K$, and that its
kernel is reduced to the constant functions.
\end{proof}

Let $\om_{h_\infty}\in\Om^2(X,\R)$ be a Kähler-Ricci soliton with
respect to $\xi_\infty\in\Lie\Aut(X)$, let $K\subset\Aut_0(X)$ be a
connected compact subgroup of isometries of $(X,g_{h_\infty}^{TX})$ and
let $T\subset K$ be the identity component of its center.
Via \cref{Kostantdef}, \cref{TZKR} implies in particular
that for any $\eta\in\Lie K$
and any $\xi\in\sqrt{-1}\Lie T$, we have
\begin{equation}\label{Letatheta=0}
\dt\Big|_{t=0}\phi_{t\eta}^*\,\theta_{h_\infty}(\xi)=\theta_{h_\infty}([\eta,\xi])=0\,,
\end{equation}
so that \cref{thetaxiprop} implies that $\eta\in\sqrt{-1}\Lie K$
belongs to $\sqrt{-1}\Lie T$ if and only if
$\theta_{h_\infty}(\eta)\in\cinf(X,\R)^K$.
As explained by Futaki in \cite[\S\,4]{Fut87} and
using formula \cref{KRvolfla} for Kähler-Ricci solitons,
this implies the following result via a straightforward
generalization of a theorem of Lichnerowicz
and Matsushima in \cite{Lic59,Mat57} restricted to the subspace
$\cinf(X,\R)^K\subset\cinf(X,\C)$.
We also refer to \cite[Lem.\,2.2]{TZ00}
for a self-contained proof.

\begin{prop}\label{TZ}
{\emph{\cite[Prop.\,4.1]{Fut87}}}
Let $h_\infty\in\Met^+(L)$ be a Kähler-Ricci soliton with respect to
$\xi_\infty\in\Lie\Aut(X)$, let
$K\subset\Aut_0(X)$ be the identity component of the group of holomorphic
isometries of $(X,g_{h_\infty}^{TX})$ and
let $T\subset K$ be the identity component of its center.
Then the first positive eigenvalue $\lambda_1(h_\infty,\xi_\infty)>0$
of $\Delta_{h_\infty}^{(\xi_\infty)}$ acting on
$\cinf(X,\R)^K$ as in \cref{TZopdef}
satisfies $\lambda_1(h_\infty,\xi_\infty)=1$,
and the associated eigenspace satisfies
\begin{equation}
\Ker\,\left(\Delta_{h_\infty}^{(\xi_\infty)}-\Id\right)=
\<\theta_{h_\infty}(\xi)~|~\xi\in\sqrt{-1}\Lie T\>\,.
\end{equation}
\end{prop}
%
%\begin{rmk}
%To compare this result with \cite[Lem.\,2.2]{TZ00}, note that
%%formula \cref{TZopdef} is the positive version of Tian-Zhu's operator, so that
%the operator defined in \cite[(2.2)]{TZ00} for $\phi=0$ and $t=0$
%coincides with $-(\Delta^{(\xi_\infty)}_{h_\infty}-\Id)$, and
%\cref{TZ} is a consequence of \cite[Lem.\,2.2\,(iii)]{TZ00}
%restricted to the subspace $\cinf(X,\R)^K$ thanks to
%formula \cref{Letatheta=0}. Note also
%that the coefficient $1/4\pi$ in formula \cref{TZopdef} comes
%from our convention \cref{preq} for the Kähler metric and the fact that
%we are considering the Riemannian Laplacian instead of the holomorphic
%Laplacian.
%\end{rmk}

%Then the operator
%\begin{equation}
%\Delta_h^{(\xi)}:=e^{-\theta(\xi)/2}\left(\Delta+\xi\right)e^{\theta(\xi)/2}\,,
%\end{equation}
%is a generalized Laplacian in the sense of \cite[Def.\,2.2]{BGV94},
%self-adjoint with respect to the usual $L^2$-scalar product
%on $\cinf(X,\R)$,
%and thus admits an associated heat operator $e^{t\Delta_h^{(\xi)}}$ for all $t>0$.
%\todo{introduire les notations}

\subsection{Berezin-Toeplitz quantization}
\label{BTsec}

In this Section, we fix a positive Hermitian metric $h\in\Met^+(L)$ on
the line bundle
$L:=\det(T^{(1,0)}X)$, and use the associated volume form $d\nu_h$
given by formula \cref{dnucandef}. For any $p\in\N^*$, we consider the
Hermitian product $L^2(h^p)$ on $H^0(X,L^p)$ defined for any
$s_1,\,s_2\in\cinf(X,L^p)$ by
\begin{equation}\label{L2}
\<s_1,s_2\>_{L^2(h^p)}:=
\int_X\,\<s_1(x),s_2(x)\>_{h^p}\,d\nu_h(x)\,.
\end{equation}
We write
\begin{equation}
\HH_p:=\left(H^0(X,L^p),\<\cdot,\cdot\>_{L^2(h^p)}\right)\,,
\end{equation}
for the associated Hilbert space of holomorphic sections, and set
$n_p:=\dim\HH_p$. We
write $\cL(\HH_p)$ for the space of Hermitian endomorphisms
of $\HH_p$.

By definition of $L$ ample and for all $p\in\N^*$ big enough,
the \emph{Kodaira map}
\begin{equation}\label{Kod}
\begin{split}
\text{Kod}_{p}:X&\longrightarrow\mathbb{P}(H^0(X,L^p)^*)\,,\\
x\,&\,\longmapsto \{\,s\in H^0(X,L^p)~|~s(x)=0\,\}
\end{split}
\end{equation}
is well-defined and an embedding. In this section, we will always implicitly
assume $p\in\N^*$ big enough so that this is verified.

\begin{defi}\label{cohstateprojdef}
The \emph{coherent state
projector} is the map
\begin{equation}
\Pi_{h^p}:X\longrightarrow \cL(\HH_p)
\end{equation}
sending $x\in X$ to the orthogonal projector satisfying
\begin{equation}\label{cohstateprojfla}
\Ker\Pi_{h^p}(x)=\{\,s\in\HH_p~|~s(x)=0\,\}\,.
\end{equation}
The \emph{Rawnsley function} is the unique positive function
$\rho_{h^p}\in\cinf(X,\R)$ defined
for any $s_1,\,s_2\in\HH_p$ and $x\in X$ by
\begin{equation}\label{Rawndeffla}
\rho_{h^p}(x)\,\<\Pi_{h^p}(x)s_1,s_2\>_{L^2(h^p)}
=\<s_1(x),s_2(x)\>_{h^p}\,.
\end{equation}
\end{defi}

From the well-definition of the Kodaira map \cref{Kod}, formula
\cref{cohstateprojfla} implies that $\Pi_{h^p}(x)$ is a rank-$1$ projector
for all $x\in X$, so that
for any orthonormal basis $\{s_j\}_{j=1}^{n_p}$ of $\HH_p$,
formula \cref{Rawndeffla} implies
\begin{equation}\label{rhodensofstate}
\rho_{h^p}=\rho_{h^p}\,\Tr[\Pi_{h^p}]=
\sum_{j=1}^{n_p}\rho_{h^p}\,\<\Pi_{h^p}s_j,s_j\>_{L^2(h^p)}
=\sum_{j=1}^{n_p}\,|s_j|^2_{h^p}\,.
\end{equation}
This gives the characterization of the Rawnsley function as
a \emph{density of states}, which in turn
coincides with the \emph{Bergman kernel} with respect to $d\nu_h$
along the diagonal, as described in \cite[\S\,4.1.9]{MM07}.
%We will study its description as a \emph{density of states}
%in \cref{relbalsec}.
%, since then the
%bilinear form $\<\Pi_{h^p}(x)\cdot,\cdot\>_{L^2(h^p)}$ does descend to a
%Hermitian metric on $L^p$.
The following Theorem describes the semi-classical behaviour of
the Rawnsley function as $p\to+\infty$, extending \cite{Cat99,Zel98}.

\begin{theorem}\label{Bergdiagexp}
{\emph{\cite[Th.\,1.3]{DLM06}}}
There exist functions
$b^{(r)}_h\in\cinf(X,\R)$ for all $r\in\N$
such that for any $m,\,k\in\N$, there exists $C_{m,k}>0$
such that for all $p\in\N^*$ big enough,
\begin{equation}\label{Bergdiagexpfla}
\left|\,\rho_{h^p}-p^n\sum_{r=0}^{k-1}\frac{1}{p^r}b^{(r)}_h\,
\right|_{\CC^m}\leq C_{m,k}p^{n-k}\,,
\end{equation}
where $b^{(0)}_h\in\cinf(X,\R)$
is given by the identity $b^{(0)}_h\,d\nu_h=\om^n_h/n!$.

Furthermore, the functions $b^{(r)}_h\in\cinf(X,\R)$ for all $r\in\N$
depend
smoothly on $h\in\Met^+(L)$ and its successive derivatives,
and for each $m,\,k\in\N$, there exists $l\in\N$ 
such that the constant $C_{m,k}>0$ can be chosen uniformly for
$h\in\Met^+(L)$ in a bounded subset
in $\CC^l$-norm.
\end{theorem}

The concepts introduced in
\cref{cohstateprojdef} induce a \emph{coherent state quantization}, 
described via the following fundamental tools.

\begin{defi}\label{BTquantdef} For any $p\in\N^*$,
the \emph{Berezin-Toeplitz quantization map} is the linear map
$T_{h^p}:\cinf(X,\R)\to\cL(\HH_p)$ defined for any
$f\in\cinf(X,\R)$ by the formula
\begin{equation}\label{BTmapfla}
T_{h^p}(f):=\int_X f(x)\,\Pi_{h^p}(x)\,\rho_{h^p}(x)\,d\nu_h(x)\,.
\end{equation}
The \emph{Berezin symbol} is the linear map
$\sigma_{h^p}:\cL(\HH_p)\to\cinf(X,\R)$ defined for any
$A\in\cL(\HH_p)$ and $x\in X$ by the formula
\begin{equation}\label{BTdequantfla}
\sigma_{h^p}(A)(x):=\Tr[\Pi_{h^p}(x)A]\,.
\end{equation}
\end{defi}

Using formula \cref{Rawndeffla}, we get for any
$f\in\cinf(X,\R)$ and any $s_1,\,s_2\in\HH_p$,
\begin{equation}\label{Tpf}
\<T_{h^p}(f)s_1,s_2\>_{L^2(h^p)}
=\int_X\,f(x)\,\<s_1(x),s_2(x)\>_{h^p}\,d\nu_h(x)\,,
\end{equation}
recovering from \cref{BTquantdef} the usual
definition of Berezin-Toeplitz quantization
associated with the volume form
$d\nu_h$, as described in \cite[Chap.\,7]{MM07}.
The following fundamental Theorem describes the semi-classical behaviour
of the Berezin-Toeplitz quantization as $p\to+\infty$.
We write $\|\cdot\|_{op}$ for the operator norm on endomorphisms of $\HH_p$.

\begin{theorem}\label{Toepexp}
{\emph{\cite[Th.\,1.1]{MM08b}}}
For any $f,\,g\in\cinf(X,\R)$,
there exist bi-differential operators $C^{(r)}_h$ for all $r\in\N$
such that for any $k\in\N$, there exists $C_{k}>0$
such that for all $p\in\N^*$ big enough,
\begin{equation}\label{Toepexpfla}
\left\|\,T_{h^p}(f)T_{h^p}(g)-
\sum_{r=0}^{k-1}\frac{1}{p^r}\,T_{h^p}(C^{(r)}_h(f,g))\,
\right\|_{op}\leq \frac{C_{k}}{p^{k}}\,,
\end{equation}
where $C^{(0)}_h(f,g)\in\cinf(X,\R)$
is given by $C^{(0)}_h(f,g)=f g$.

Furthermore, the bi-differential operators $C_h^{(r)}$
depend smoothly on $h\in\Met^+(L)$ and its successive derivatives
for all $r\in\N$, and for each $k\in\N$, there exists $l\in\N$ 
such that the constant $C_{k}>0$ can be chosen uniformly for
$f,\,g\in\cinf(X,\R)$ and $h\in\Met^+(L)$ in a bounded subset
in $\CC^l$-norm.
\end{theorem}

In the context of quantization,
the Berezin symbol \cref{BTdequantfla}
of a quantum observable $A\in\cL(\HH_p)$ is interpreted
as the classical observable given by
the expectation value of $A$
at coherent states.
The following result shows that this operator is dual to the
Berezin-Toeplitz quantization
with respect to the trace norm on $\cL(\HH_p)$ and
the $L^2$-norm on $\cinf(X,\R)$ induced by the density $\rho_{h^p}\,d\nu_h$.

\begin{prop}\label{dualprop}
For any $A\in\cL(\HH_p)$ and $f\in\cinf(X,\R)$, we have
\begin{equation}\label{dualfla}
\Tr[T_{h^p}(f)\,A]=\int_X\,f\,\sigma_{h^p}(A)\,\rho_{h^p}\,d\nu_h\,.
\end{equation}
Furthermore, we have $T_{h^p}(1)=\Id_{\HH_p}$ and
$\sigma_{h^p}(\Id_{\HH_p})=1$.
\end{prop}
\begin{proof}
Formula \cref{dualfla} is an immediate consequence of \cref{BTquantdef}.
On the other hand, by \cref{cohstateprojdef} we have
\begin{equation}\label{intPirho=Id}
\int_X\,\Pi_{h^p}(x)\,\rho_{h^p}(x)\,d\nu_h(x)=\Id_{\HH_p}\,.
\end{equation}
This implies the identity $T_{h^p}(1)=\Id_{\HH_p}$, while the second identity
is a consequence of the fact that $\Pi_{h^p}$ is a
rank-$1$ projector, so that $\Tr[\Pi_{h^p}]=1$.
\end{proof}

This gives rise to the following concept, which will be the
main technical tool of this paper.

\begin{defi}\label{Bertransdef}
The \emph{Berezin transform} is the
linear operator
\begin{equation}
\begin{split}
\cB_{h^p}:\cinf(X,\R)&\longrightarrow\cinf(X,\R)\,,\\
f~&\longmapsto~\sigma_{h^p}
\left(T_{h^p}\left(f\right)\right)\,.
\end{split}
\end{equation}
\end{defi}

As explained in details in \cite[\S\,2]{IKPS19}, the Berezin transform
is a \emph{Markov operator} with stationary measure $\rho_{h^p}\,d\nu_h$,
and measures the delocalisation of a classical observable after quantization.
From this point of view, the following semi-classical result can be thought
as a quantitative refinement of the celebrated \emph{Heisenberg's uncertainty
principle}, and refines a semi-classical expansion due to
Karabegov and Schlichenmaier \cite{KS01}.

\begin{theorem}\label{KS}
{\emph{\cite[Lem.\,7.2.4]{MM07},\,\cite[Prop.\,4.8]{IKPS19}}}
For any $f\in\cinf(X,\R)$,
there exist differential operators $D_h^{(r)}\in\cinf(X,\R)$ for all $r\in\N$
such that for any $k,\,m\in\N$, there exists a constant $C_{m,\,k}>0$
such that for all $p\in\N^*$ big enough,
\begin{equation}\label{KSfla}
\left|\,\cB_{h^p}(f)-\sum_{r=0}^{k-1}\,p^{-r}\,D_h^{(r)}(f)\,\right|_{\CC^m}\leq
\frac{C_{m,k}}{p^k}\;.
\end{equation}
with $D_h^{(0)}(f)=f$ and $D_h^{(1)}(f)=\frac{1}{4\pi}\Delta_h\,f$,
where $\Delta_h$ is the Riemannian Laplacian of $(X,g_h^{TX})$.

Furthermore, the differential operators $D_h^{(r)}$
depend smoothly on $h\in\Met^+(L)$ and its successive derivatives
for all $r\in\N$, and for every $m,\,k\in\N$, there exists $l\in\N$ 
such that the constant $C_{m,k}>0$ can be chosen uniformly
for $f\in\cinf(X,\R)$ and $h\in\Met^+(L)$ in a bounded subset
in $\CC^l$-norm.
\end{theorem}

\section{Quantum action of the automorphism group}
\label{quantAutsec}

One basic property of any reasonable quantization is its compatibility
with symmetries, represented here by the action of
the automorphism group $\Aut(X)$.
In this Section, we will establish this fact through the use of
Berezin-Toeplitz quantization, and establish
fundamental properties of the quantized Futaki invariant
\cref{quantFutdefintro} as a holomorphic
invariant for this action. In the whole Section, we fix $h\in\Met^+(L)$
and consider the setting of \cref{BTsec}.

\subsection{Quantization of the holomorphy potentials}
\label{quantholpotsec}

Let $K\subset\Aut_0(X)$ be the identity component of the
group of holomorphic isometries of
$(X,g_h^{TX})$. In the notations of \cref{BTsec}, the action of $K$ on $X$
lifts to a unitary action on $\HH_p$, and any $\xi\in\sqrt{-1}\Lie T$
induces a Hermitian operator $L_\xi\in\cL(\HH_p)$ defined via
formula \cref{Lxidef}. On the other hand, \cref{momentrmk} implies that
$\theta_h(\xi)\in\cinf(X,\R)$ generates a \emph{Hamiltonian flow}
$\phi_{tJ\xi}\in K$, for all $t\in\R$. From general principles,
the Hermitian operator $L_\xi\in\cL(\HH_p)$
can thus be seen as the appropriate
quantization of the classical observable $\theta_h(\xi)\in\cinf(X,\R)$,
and the following result illustrates this principle via
Berezin-Toeplitz quantization.
Recall that we consider Berezin-Toeplitz
quantization with respect to the anticanonical volume form \cref{dnucandef}
instead of the usual Liouville form.

\begin{prop}\label{Tuy}
For any $\xi\in\Lie\Aut(X)$, the induced operator
$L_\xi\in\End(\HH_p)$ satisfies
\begin{equation}\label{Tuyfla}
L_\xi=\left(p+1\right)T_{h^p}\left(\theta_h(\xi)\right)\,.
\end{equation}
In particular, there exists a constant $C>0$,
independent of $h^p\in\Met^+(L^p)$ and $\xi\in\Lie\Aut(X)$,
such that for any
$p\in\N^*$, we have
\begin{equation}\label{Lxi<Cp}
\|L_\xi\|_{op}<C\,p\,|\xi|\,.
\end{equation}
\end{prop}
\begin{proof}
This is version of an argument due to Tuynman \cite{Tuy87},
adapted to the anticanonical
volume form \cref{dnucandef}.
Recall that the Chern connection $\nabla^h$ induces the holomorphic
structure of $L$, and recall the variation formulas
\cref{dnuef/dnu,etxi*dnu}
for the anticanonical volume form. Then by \cref{thetaxiprop},
for any holomorphic sections $s_1,\,s_2\in H^0(X,L^p)$ we get
\begin{equation}
\begin{split}
\int_X\,\<\nabla^{h^p}_\xi s_1,s_2\>_{h^p}\,d\nu_h
&=\int_X\,\<\nabla^{h^p}_{\xi^{1,0}} s_1,s_2\>_{h^p}\,d\nu_h
=\int_X\,L_{\xi^{1,0}}\< s_1,s_2\>_{h^p}\,d\nu_h\\
&=-\int_X\,\< s_1,s_2\>_{h^p}\,L_{\xi^{1,0}}d\nu_h
=\int_X\,\theta_h(\xi)\,\< s_1,s_2\>_{h^p}\,d\nu_h
\end{split}
\end{equation}
On the other hand, \cref{Kostantdef} implies that for any
$s\in\HH_p$, we have
\begin{equation}
p\,\theta_h(\xi)\,s=L_{\xi}\,s-\nabla^{h^p}_{\xi}s\,.
\end{equation}
Then using formula \cref{Tpf} for the Berezin-Toeplitz quantization map,
for any $s_1,\,s_2\in\HH_p$ we get
\begin{equation}
\begin{split}
\<L_\xi\,s_1,s_2\>_{L^2(h^p)}
&=p\int_X\theta_h(\xi)\<s_1,s_2\>_{h^p}\,d\nu_h+
\int_X\,\<\nabla^{h^p}_\xi s_1,s_2\>_{h^p}\,d\nu_h\\
&=\left(p+1\right)\,\<T_p(\theta_h(\xi))\,s_1,s_2\>_{L^2(h^p)}\,.
\end{split}
\end{equation}
This gives the identity \cref{Tuyfla}, which implies in turn
the inequality \cref{Lxi<Cp} via formula \cref{Tpf}.
\end{proof}

Using \cref{Toepexp}, we can establish the following result, which is
an exponentiation of \cref{Tuy}.

\begin{prop}\label{Tuycor}
There exist functions $\theta^{(j)}_{\xi}\in\cinf(X,\C)$
for all $j\in\N$, depending smoothly on $\xi\in\Lie\Aut(X)$,
such that for any $k\in\N$, the exponential
$e^{L_\xi/p}\in\End(\HH_p)$
satisfies the following expansion in the sense of the operator norm
as $p\to+\infty$,
\begin{equation}\label{eLxiexp}
e^{L_\xi/p}
=T_{h^p}\left(e^{\theta_h(\xi)}\right)+
\sum_{j=1}^{k-1}\,p^{-j}\,T_{h^p}(\theta^{(j)}_{\xi})+O(p^{-k})\,.
\end{equation}
In particular,
there exists a constant $C>0$ such that for all $p\in\N^*$ big enough,
\begin{equation}\label{eLxibded}
C^{-1}p^n\leq\Tr\left[e^{L_\xi/p}\right]\leq Cp^n\,.
\end{equation} 
These estimates are uniform for $\xi\in\Lie\Aut(X)$
and $h\in\Met^+(L)$ in any bounded set in $\CC^l$-norm.
\end{prop}
\begin{proof}
In this proof, the notation $O(p^{-k})$ for some $k\in\N$ is
taken in the sense of the operator norm as $p\to+\infty$,
uniformly in the $\CC^l$-norm of
$\theta_h(\xi)\in\cinf(X,\R)$ and $h\in\Met^+(L)$, for some $l\in\N$
only depending on $k\in\N$ and for all $t\in[0,1]$.

First note that \cref{Tuy} implies the following ordinary differential equation
in $t\in[0,1]$,
\begin{equation}\label{odeexp}
\left\{
\begin{array}{l}
  \dt\,e^{tL_\xi/p}=\left(1+\frac{1}{p}\right)
  T_{h^p}(\theta_{h}(\xi))e^{tL_\xi/p} \\
  \\
  e^{tL_\xi/p}\big|_{t=0}=\Id_{\HH_p}.
\end{array}
\right.
\end{equation}
On the other hand, \cref{Toepexp} implies that for all $t\in[0,1]$,
we have
\begin{equation}\label{odeToep}
\dt\,T_{h^p}(e^{t\theta_{h}(\xi)})=T_{h^p}(\theta_{h}(\xi))
  T_{h^p}(e^{t\theta_{h}(\xi)})+O(p^{-1})\,.
\end{equation}
Using the fact from \cref{BTquantdef}
and formula \cref{intPirho=Id} that $T_{h^p}(1)=\Id_{\HH_p}$,
we can then apply Grönwall's lemma to the difference of \cref{odeexp}
and \cref{odeToep} to get the expansion \cref{eLxiexp} for $k=1$.

Assume now by induction on $k\geq 2$ that there exist
a function $g_{k,t}\in\cinf(X,\C)$ and
functions $f_{j,t}\in\cinf(X,\C)$ for all $1\leq j\leq k-1$,
all smooth in $t\in[0,1]$, such that
\begin{multline}\label{odeToepk}
\left(1+\frac{1}{p}\right)T_{h^p}(\theta_{h}(\xi))
T_{h^p}\left(\sum_{j=0}^{k-1}p^{-j}f_{j,t}\right)\\
=\dt\,T_{h^p}\left(\sum_{j=0}^{k-1}p^{-j}f_{j,t}\right)
+p^{-k}T_{h^p}(g_{k,t})+O(p^{-k+1})\,.
\end{multline}
Using \cref{Toepexp}, we then
see that \cref{odeToepk} holds for $k$ replaced
by $k+1$, with the function $f_{k,t}\in\cinf(X,\R)$ defined
as the solution of the
ordinary differential equation
\begin{equation}\label{ODEtheta}
\left\{
\begin{array}{l}
  \dt\,f_{k,t}=f_{k,t}\,\theta_{h}(\xi)+g_{k,t} \\
  \\
 f_{k,0}=0\,.
\end{array}
\right.
\end{equation}
This shows by induction that \cref{odeToepk} holds for all $k\in\N$.
Applying Grönwall's lemma to the difference of \cref{odeexp} and
\cref{odeToepk} as above, this gives the result taking
$\theta^{(j)}_\xi:=f_{j,1}$ for all $j\in\N$. The smooth dependance on
$\xi\in\Lie\Aut(X)$ is then clear from the ordinary differential equation
\cref{ODEtheta}.

Recall on the other hand
that the coherent state projector of \cref{cohstateprojdef}
is a rank-$1$ projector, so that $\Tr[\Pi_p(x)]=1$ for all $x\in X$.
Using \cref{Bergdiagexp} and formula \cref{intPirho=Id}, we then
get a constant $C>0$ such that the dimension of $\HH_p$ satisfies
$n_p<C p^n$ for all $p\in\N^*$, which implies
$\Tr\left[A\right]\leq Cp^n\left\|A\right\|_{op}$
for all $A\in\End(\HH_p)$ and $p\in\N^*$ by Cauchy-Schwartz inequality.
The inequality \cref{eLxibded} then follows by
applying this estimate to
$A=O(p^{-1})$ in the expansion \cref{eLxiexp} for $k=1$,
and by \cref{BTquantdef} of the
Berezin-Toeplitz quantization map.
%of $f\in\cinf(X,\R)$, which
%readily implies $\|T_{h^p}(f)\|_{op}\leq\max_{x\in X}\,|f(x)|$.
%On the other hand, using the fact that the coherent states of
%\cref{cohstateprojdef} are rank-$1$ projectors and \cref{dualprop},
%\cref{Bergdiagexp} implies
%\begin{equation}\label{npest}
%
%\end{equation}
%The second identity in \cref{eLxibded} is then
%a consequence of the first, and the general fact that 
\end{proof}

\subsection{Quantized Futaki invariants}
\label{quantFutsec}

In this Section, we study the quantized Futaki invariants
\cref{quantFutsec}, which
play a central role in this paper. Recall
that for any $\xi\in\Lie\Aut(X)$ and any $p\in\N^*$, we have
an operator $L_\xi\in\End(H^0(X,L^p))$
defined by formula \cref{Lxidef}.
%
%\begin{defi}\label{quantFut}
%The \emph{quantized Futaki invariant}
%relative to $\xi\in\Lie\Aut(X)$ at level $p\in\N^*$
%is the linear map
%\begin{equation}
%\begin{split}
%\Fut_p^{\xi}:\Lie\Aut(X)&\longrightarrow\C\\
%\xi&\longmapsto\Tr_{\HH_p}\left[L_\eta e^{p^{-1}L_\xi}\right]\,.
%\end{split}
%\end{equation}
%\end{defi}

Recall the decomposition \cref{LieKC} for the Lie algebra
of the complexification $K_\C\subset\Aut_0(X)$ of
any compact subgroup $K\subset\Aut_0(X)$.
We will establish a quantized version of \cref{TZinv}
of Tian and Zhu,
whose first part is given by the following result.
It can also be found in \cite[\S\,3.1]{Tak15}, but we give here
a short proof using Berezin-Toeplitz quantization.
%to show the existence and uniqueness of $\xi_p\in\sqrt{-1}\Lie K$ such that
%$\Fut_\xi$ vanishes identically, for all $p\in\N^*$. This will follow
%from the following convexity of the following energy functional.

\begin{lem}\label{Fpcvx}
Let $T\subset\Aut(X)$ be a compact torus. Then
for any $\in\N$, the functional $F_p:\sqrt{-1}\Lie T\to\R$
defined by the formula
\begin{equation}\label{Fpcvxfla}
F_p(\xi):=p\Tr[e^{L_\xi/p}]\,,
\end{equation}
is strictly convex and proper.
\end{lem}
\begin{proof}
Let $T\subset\Aut(X)$ be a compact torus, and consider the setting of
\cref{BTsec} with a $T$-invariant positive Hermitian
metric $h\in\Met^+(L)^T$,
which can always be constructed
by average over $T$.
Then for any $\xi\in\sqrt{-1}\Lie T$,
\cref{thetaxiprop}
implies that there exists $x\in X$ such that
$\theta_h(\xi)(x)>0$.
Let now $p\in\N^*$ be large enough so that the Kodaira map \cref{Kod}
is well-defined, and let $s_p\in\HH_p$ with $\|s_p\|_{L^2(h^p)}=1$ be 
in the image of the coherent state projector $\Pi_{h^p}(x)\in\cL(\HH_p)$
of \cref{cohstateprojdef}.
Then from \cref{KS,Tuy}, as $p\to+\infty$ we get
\begin{equation}
\<L_\xi s_x,s_x\>_{L^2(h^p)}=\Tr[L_\xi\Pi_{h^p}(x)]=\cB_{h^p}(\theta_h(\xi))(x)
=\theta_h(\xi)(x)+O(p^{-1})\,,
\end{equation}
uniformly in the $\CC^l$-norm of $\theta_h(\xi)$ for some $l\in\N$.
Thus there exists
$p_0\in\N$ independent of $\xi\in\Lie\Aut(X)$ such that $L_\xi\in\cL(\HH_p)$
is non-negative for all $p\geq p_0$.
For any $p\in\N^*$, write $\Spec_p(T)\subset(\Lie T)^*$
for the joint spectrum of $L_\xi\in\cL(\HH_p)$
for all $\xi\in\sqrt{-1}\Lie T$.
Taking $p\geq p_0$, we then get that
for any $\xi\in\sqrt{-1}\Lie T$,
there exists $\chi_0\in\Spec_p(T)$ such that
$(\chi_0,\xi)>0$.
Thus for any $\eta\in\sqrt{-1}\Lie T$, we have
\begin{equation}
\frac{d^2}{dt^2}\,F_p(\eta+t\xi)=
p^{-1}\sum_{\chi\in\Spec_p(T)}(\chi,\xi)^2 e^{(\chi,\eta+t\xi)/p}>0
\end{equation}
and
\begin{equation}
F_p(\eta+t\xi)=p\sum_{\chi\in\Spec_p(T)} e^{(\chi,\eta+t\xi)/p}\geq
p\,e^{(\chi_0,\eta)/p}e^{t(\chi_0,\xi)/p}
\xrightarrow{t\to+\infty}+\infty\,.
\end{equation}
This proves the result.
\end{proof}

We then have the following corollary, which is a central property of
the quantized Futaki invariants \cref{quantFutdefintro} and
which is the quantization
of the second part of \cref{TZinv}.
It can also be found in \cite[\S\,3.2]{Tak15}.

\begin{cor}\label{xipmin}
Let $K\subset\Aut_0(X)$ be a connected compact subgroup, and let $T\subset K$
be the identity component of its center. Then
for any $p\in\N^*$ big enough,
there exists a unique $\xi_p\in\sqrt{-1}\Lie K$ such that the associated
quantized Futaki invariant $\Fut^{\xi_p}_p:\sqrt{-1}\Lie K\to\C$
vanishes identically, which
is given by the unique minimizer of the functional $F_p:\sqrt{-1}\Lie T\to\R$
of \cref{Fpcvx}.
\end{cor}
\begin{proof}
Fix $p\in\N^*$, and note that for any compact torus $T\subset\Aut_0(X)$ and any
$\xi,\,\eta\in\sqrt{-1}\,\Lie T$, we have
\begin{equation}
\Fut_p^{\xi}(\eta)=\dt\big|_{t=0}\,F_p(\xi+t\eta)\,.
\end{equation}
Now \cref{Fpcvx} implies that $F_p:\sqrt{-1}\Lie T\to\R$
admits a unique minimizer, which is the
unique $\xi_p\in\sqrt{-1}\Lie T$
such that $\Fut_p^{\xi}(\eta)=0$ for all $\eta\in\sqrt{-1}\Lie T$.

Let now $K\subset\Aut_0(X)$ be a connected compact subgroup, and let
$T\subset K$ be the identity component of its center.
Then for any $\xi_1,\,\xi_2\in\Lie K_\C$, we have
\begin{equation}
\Tr[L_{[\xi_1,\xi_2]}e^{L_{\xi_p}/p}]
=\Tr\left[[L_{\xi_1},L_{\xi_2}e^{L_{\xi_p}/p}]\right]=0\,,
\end{equation}
so that $\Fut^{\xi_p}_p(\eta)=0$ for all
$\eta\in[\Lie K_\C,\Lie K_\C]\subset \Lie K_\C$.
Using the classical decomposition
$\Lie K_\C=\Lie T_\C\oplus[\Lie K_\C,\Lie K_\C]$, we then get that
$\Fut^{\xi_p}_p(\eta)=0$ for all $\eta\in\Lie K_\C$.

Assume now that $\til{\xi}_p\in\sqrt{-1}\Lie K$ is such that
$\Fut^{\til\xi_p}_p(\eta)=0$
for all $\eta\in\Lie K_\C$, and let now $\til{T}\subset K$ be a maximal
compact torus
such that $\til{\xi}_p\in\sqrt{-1}\Lie\til{T}$.
Then we have $\Lie T_\C\subset\Lie\til{T}_\C$ by
definition of a maximal torus, and we have
$\til{\xi}_p=\xi_p$ by uniqueness of the minimizer of
$F_p:\sqrt{-1}\Lie\til{T}\to\R$ given by \cref{Fpcvx}. This gives the result.
\end{proof}

\subsection{Asymptotic expansion of the quantized Futaki invariant}
\label{quantFutexpsec}

In this Section, we fix a connected compact subgroup $K\subset\Aut_0(X)$
and write $T\subset K$ for the identity component of its center.
The following result describes the semi-classical behavior of the
quantized Futaki invariant \cref{quantFutdefintro} and the functional
\cref{Fpcvxfla} as $p\to+\infty$, recovering the modified
Futaki invariant \cref{FutTZ} and the functional \cref{FTZ}
as the highest order coefficient respectively.
As explained for instance in \cite[\S\,4.4]{BW14},
this is essentially a consequence of the equivariant Riemann-Roch formula,
but we give
here a short proof using Berezin-Toeplitz quantization.

\begin{prop}\label{Futexp}
There exist smooth maps $F_p^{(j)}:\sqrt{-1}\Lie K\to\C$
for all $j\in\N$
such that for any $k\in\N^*$, we have the following asymptotic
expansion as $p\to+\infty$,
\begin{equation}\label{Futexp2}
\frac{F_p(\eta)}{p^{n+1}}=F(\eta)+\sum_{j=1}^{k-1}\,p^{-j}\,F_p^{(j)}(\eta)
+O(p^{-k})\,.
\end{equation}
Furthermore, there exist linear maps
$\Fut^{(j)}_{\xi}:\Lie\Aut(X)\to\C$ for all $j\in\N$, depending smoothly
on $\xi\in\Lie\Aut(X)$, such that for any $k\in\N^*$,
we have the following asymptotic expansion as $p\to+\infty$,
\begin{equation}\label{Futexp1}
\frac{\Fut_p^{\xi}(\eta)}{p^{n+1}}=\Fut_\xi(\eta)+
\sum_{j=1}^{k-1}\,p^{-j}\,\Fut^{(j)}_{\xi}(\eta)+O(p^{-k})\,.
\end{equation}
These estimates are uniform for $\xi,\,\eta\in\Lie\Aut(X)$
in any compact set.
\end{prop}
\begin{proof}
Fix $h\in\Met^+(L)$. Using \cref{dualprop,Bertransdef} and by \cref{Bergdiagexp,KS}, we get differential operators
$D_j$ for any $j\in\N$ such that for any
$f,\,g\in\cinf(X,\C)$, any $k\in\N$
and as $p\to+\infty$,
\begin{equation}\label{TrTfTgest}
\begin{split}
\Tr[T_{h^p}(f)T_{h^p}(g)]&=
\int_X\,f\,\cB_{h^p}(g)\,\rho_p\,d\nu_h\\
&=p^n\int_X\,f\,g\,\frac{\om^n_h}{n!}+\sum_{j=1}^{k-1}\,p^{n-j}\,\int_X\,f\,D_j(g)\,d\nu_h+O(p^{n-k})\,,
\end{split}
\end{equation}
uniformly in the derivatives of $f,\,g$ up to order $l\in\N$ only depending
on $k\in\N$.

Now using \cref{Tuy,Tuycor}, for any $k\in\N$ and as $p\to+\infty$,
the quantized Futaki invariant \cref{quantFutdefintro} and the functional \cref{Fpcvxfla} satisfy
\begin{multline}\label{FutexpT}
\frac{\Fut^{\xi}_p(\eta)}{p+1}
=\Tr\left[T_{h^p}(\theta_h(\eta))T_{h^p}(e^{\theta_h(\xi)})\right]\\
+\sum_{j=1}^{k-1}\,p^{-j}\,
\Tr\left[T_{h^p}(\theta_h(\eta))T_{h^p}(\theta^{(j)}_\xi)\right]+O(p^{-k})
\quad\text{and}\\
\frac{F_p(\xi)}{p}=\Tr\left[T_{h^p}(e^{\theta_h(\xi)})\right]
+\sum_{j=1}^{k-1}\,p^{-j}\,
\Tr\left[T_{h^p}(\theta^{(j)}_\xi)\right]+O(p^{-k})\,,~~~~~~~~
\end{multline}
uniformly for $\xi,\,\eta\in\Lie\Aut(X)$ in any compact set.
Using the fact from \cref{dualprop} that $\Id_{\HH_p}=T_{h^p}(1)$,
we can then apply the identity \cref{TrTfTgest}
to the expansions \cref{FutexpT} and compare with \cref{TZinv}
for the first coefficients to get the result.

%PREUVE ALTERNATIVE\\
%
%Fix $h\in\Met^+(K_X^*)^T$, and let $\s\in\BB(H^0(X,L^p))$ be an associated
%$L^2$-orthonormal basis.
%Under the hypotheses of the theorem, we have
%$e^{L_\xi}\in\cL(H^0(X,L^p))$,
%%and $L_\eta=L_{\eta_1}+\sqrt{-1} L_{\eta_2}
%so that using \cref{quantFut} and [?], we get
%\begin{equation}
%\begin{split}
%\Fut^{\xi}_p(\eta)=\Tr_{H^0(X,L^p)}[e^{L_\xi/2p}L_\eta e^{L_\xi/2p}]
%&=\int_X\,\sigma_\s(e^{L_\xi/2p}L_\eta e^{L_\xi/2p})\,\rho_p\,d\nu_h\\
%&=\int_X\,\sigma_{e^{L_\xi/2}\s}(L_\eta)\sigma_\s(e^{L_\xi}/p)\,\rho_p\,d\nu_h\,.
%\end{split}
%\end{equation}
%Now using formula \cref{thetaef} and \cref{FSvar2}, we get
%\begin{multline}
%\sigma_\s(e^{p^{-1}L_\xi})=e^{4\pi\theta_{h}(J\xi)}e^{4\pi(d\log\rho_p.\xi^{1,0})/p}
%\quad\text{and}\\
%\sigma_{e^{L_\xi/2p}\s}(L_\eta)=4\pi p\,\theta_{h}(J\eta)+
%4\pi\sqrt{-1} d\log\rho_p.\eta^{1,0}+
%4\pi d\theta_{h_\s}(J\xi).\eta^{1,0}
%-\frac{4\pi}{p} d(d\log\rho_p.\xi^{1,0}).\eta^{1,0}\,.
%\end{multline}
%As all these terms admit an asymptotic expansion, this proves the first part.
%\todo{FAIRE DES LEMMES}
%
%For the second part, using also \cref{Bergdiagexp} we get
%\begin{equation}
%\Fut^{\xi}_p(\eta)=4\pi p^{n+1}\int_X\,\theta_h(J\eta)
%\,e^{4\pi\,\theta_h(J\xi)}\,\frac{\om_h^n}{n!}+O(p^{-1})\,.
%\end{equation}
\end{proof}
%
%\begin{rmk}\label{higherFut}
%This result gives a precise meaning to the statement, informally
%mentioned in the Introduction, that \cref{xipmin} is a quantization
%of the result \cref{TZinv} of Tian and Zhu.
%Note that we also recover from \cref{Futexp}
%the fact that the energy functional \cref{FTZ}
%and modified Futaki invariant \cref{FutTZ}
%do not depend on the choice of a positive
%metric $h\in\Met^+(L)$.
%\end{rmk}

The following result
describes the semi-classical behavior of the sequence of
vector fields produced by \cref{xipmin} as $p\to+\infty$, recovering
the vector field of \cref{TZinv} as the highest coefficient.
This is an anticanonical analogue of the result of Sano and
Tipler in \cite[Lem.\,4.5]{ST17}, and
the proof closely follows their strategy.

\begin{prop}\label{xipexp}
There exist $\xi^{(j)}\in\sqrt{-1}\Lie T$ for all $j\in\N$ such that
for any $k\in\N$,
the sequence $\{\xi_p\in\sqrt{-1}\Lie T\}_{p\in\N^*}$
of \cref{xipmin}
satisfies the following expansion as $p\to+\infty$,
\begin{equation}
\xi_p=\xi_\infty+\sum_{j=1}^{k-1}\,p^{-j}\,\xi^{(j)}+O(p^{-k})\,,
\end{equation}
where $\xi_\infty\in\sqrt{-1}\Lie T$ is the unique minimizer of
the functional \cref{FTZ}.
\end{prop}
\begin{proof}
Let $h\in\Met^+(L)^T$ be a $T$-invariant metric, which always exists
by average over $T$. 
Using \cref{Futexp}, we know that for all $\eta\in\sqrt{-1}\Lie T$,
the functionals \cref{FTZ} and \cref{Fpcvxfla}
satisfy $F_p(\eta)\xrightarrow{p\to+\infty} F(\eta)$.
As these functionals are strictly convex and proper,
this implies that their unique minimizers satisfy
\begin{equation}\label{xiptoxiinf}
\xi_p\xrightarrow{p\to+\infty}\xi_\infty\,.
\end{equation}
%Using now \cref{Toepexp,Tuy,Tuycor}, we get by definition of the
%exponential of an endomorphism that for any
%$\xi\in\sqrt{-1}\LieT$ and any $k\in\N$, we have
%\begin{multline}\label{expexpexp}
%e^{p^{-1}L_{\xi_{\infty}}+p^{-k}L_{\xi}}=e^{p^{-1}L_{\xi_{\infty}}}e^{p^{-k}L_{\xi}}
%=e^{p^{-1}L_{\xi_{\infty}}}
%\sum_{j=0}^{+\infty}p^{-kj}\,L_\xi^j\\
%=T_{h^p}\left(e^{\theta_h(\xi_\infty)}\right)+
%\sum_{j=1}^{k-2}\,p^{-j}\,T_{h^p}(\theta^{(j)}_{\xi_\infty})+
%p^{-k+1}T_{h^p}\left(e^{\theta_h(\xi_\infty)}\theta_h(\xi)
%+\theta^{(k-1)}_{\xi_\infty}\right)+O(p^{-k})\,.
%\end{multline}
From \cref{TZinv}, we know that $\Fut_{\xi_\infty}(\eta)=0$ for all
$\eta\in\Lie K_\C$. Using \cref{Futexp} and formula
\cref{FTZ} for the first coefficient, we can take the Taylor expansion
as $p\to+\infty$ of $\Fut^{\xi_\infty+p^{-1}\xi^{(1)}}_p(\eta)$
for any $\xi^{(1)}\in\sqrt{-1}\Lie T$ to get
\begin{multline}\label{Futord2}
\frac{\Fut^{\xi_\infty+p^{-1}\xi^{(1)}}_p(\eta)}{p^{n+1}}\\
=p^{-1}\left(\Fut^{(1)}_{\xi_\infty}(\eta)
+\int_X\,\theta_h(\eta)\,\theta_h(\xi^{(1)})\,
e^{\theta_h(\xi_\infty)}\,\frac{\om_h^n}{n!}
\right)+O(p^{-2})\,.
\end{multline}
Recall now the linear embedding $\theta_h:\Lie T_\C
\to\cinf(X,\C)$ induced by \cref{thetaxiprop}, and restrict
the scalar product $L^2(h_\infty,\xi_\infty)$
defined in \cref{TZopdef} for $K=T$
to the subspace
$\<\theta_h(\xi)\,|\,\xi\in\sqrt{-1}\Lie T\>\subset\cinf(X,\R)^T$.
Then by non-degeneracy, for any linear form $G:\Lie T_\C\to\C$,
there exists a unique $\xi_G\in\sqrt{-1}\Lie T$
such that for all $\eta\in\Lie T_\C$,
\begin{equation}\label{G=0}
G(\eta)+\int_X\,\theta_h(\eta)\,\theta_h(\xi_G)\,
e^{\theta_h(\xi_\infty)}\,\frac{\om_h^n}{n!}=0\,.
\end{equation}
Setting $\xi^{(1)}:=\xi_{G}$ for $G=\Fut^{(1)}_{\xi_\infty}$,
the first term of  the right hand side of \cref{Futord2} vanishes.
Now for any $\xi^{(2)}\in\sqrt{-1}\Lie T$, this together with
\cref{Futexp} and formula \cref{Futord2}
gives a linear form $G^{(2)}:\Lie T_\C\to\C$ such that
for any $\eta\in\Lie T_\C$ we have as $p\to+\infty$,
\begin{multline}\label{Futord3}
\frac{\Fut^{\xi_\infty+p^{-1}\xi^{(1)}+p^{-2}\xi^{(2)}}_p(\eta)}{p^{n+1}}\\
=p^{-2}\left(G^{(2)}(\eta)
+\int_X\,\theta_h(\eta)\,\theta_h(\xi^{(2)})\,
e^{\theta_h(\xi_\infty)}\,\frac{\om_h^n}{n!}
\right)+O(p^{-3})\,.
\end{multline}
Taking $\xi^{(2)}:=\xi_{G^{(2)}}$ as in \cref{G=0}, the first term
in the right hand side of the expansion \cref{Futord3} vanishes.
%\begin{equation}
%p^{-(n+1)}\Fut_{\xi_\infty+p^{-1}\xi^{(1)}+p^{-2}\xi^{(2)}}(\eta)=O(p^{-3})\,.
%\end{equation}
Repeating this reasoning, we then construct by induction
a sequence $\xi^{(j)}\in\sqrt{-1}\Lie T,\,j\in\N$, such that for all
$k\in\N$ and $\eta\in\Lie T_\C$, we have as $p\to+\infty$,
\begin{equation}\label{Futexpexp}
\Fut^{\xi_\infty+\sum_{j=1}^{k}p^{-j}\xi^{(j)}}_p(\eta)=O(p^{n-k})\,.
\end{equation}
For all $p,\,k\in\N$, set
$\xi^{(k)}_p:=\xi_\infty+\sum_{j=1}^{k}p^{-j}\xi^{(j)}$. Then using
\cref{xipmin} and formula \cref{Futexpexp}, we know that
for all $k\in\N$
and as $p\to+\infty$,
\begin{equation}\label{intFutest}
\int_0^1\,\dt\,\Fut^{t\xi_p^{(k)}+(1-t)\xi_p}_p(\eta)\,dt=
\Fut^{\xi_p^{(k)}}_p(\eta)-\Fut^{\xi_p}_p(\eta)=O(p^{n-k})\,,
\end{equation}
uniformly for $\eta\in\Lie T_\C$ in any compact set.
On the other hand, recall from \cref{xiptoxiinf} that
$t\xi_p^{(k)}+(1-t)\xi_p\to\xi_\infty$ uniformly
in $t\in[0,1]$ as $p\to+\infty$.
%this implies that for all $p\in\N^*$ big enough, there is
%$\epsilon>0$ such that every eigenvalue of
%$\exp\big(p^{-1}\big(tL_{\xi_p^{(k)}}+(1-t)L_{\xi_p}\big)\big)$
%is bigger than $\epsilon$,
%for all $t\in[0,1]$.
Together with \cref{Toepexp,Tuy}, this implies
the existence of constants $l\in\N$, $C,\,c,\,\epsilon>0$
such that for all $p\in\N^*$ big enough,
\begin{equation}\label{intFutest2}
\begin{split}
\int_0^1\,\dt\,&\Fut_{t\xi_p^{(k)}+(1-t)\xi_p}(\xi_p^{(k)}-\xi_p)\,dt\\
&=p^{-1}\int_0^1\,\Tr\left[\left(L_{\xi_p^{(k)}-\xi_p}\right)^2
\exp\left(p^{-1}\left(tL_{\xi_p^{(k)}}+(1-t)L_{\xi_p}\right)\right)\right]\,dt\\
&\geq\epsilon
p^{-1}(p+1)^2\Tr\left[T_{h^p}\left(\theta_h(\xi_p^{(k)}-\xi_p)\right)^2\right]\\
&\geq p^{n+1}\left(c-Cp^{-1}\right)
\left|\theta_h(\xi_p^{(k)}-\xi_p)\right|_{\CC^l}^2\,.
\end{split}
\end{equation}
Combining the estimates \cref{intFutest} and
\cref{intFutest2}, we get as $p\to+\infty$,
\begin{equation}
\left|\theta_h(\xi_p^{(k)}-\xi_p)\right|_{\CC^m}^2=O(p^{-k-1})\,,
\end{equation}
which in turn implies that $|\xi_p^{(k)}-\xi_p|^2=O(p^{-k-1})$ by equivalence
of norms on the finite dimensional space $\sqrt{-1}\Lie T$,
as $\theta_h:\sqrt{-1}\Lie T\to\cinf(X,\R)$ is an embedding
by \cref{thetaxiprop}. This proves the result.
\end{proof}

\section{Balanced metrics}
\label{balsec}

In this Section, we study the notion \cref{relbaldef}
of an anticanonically
balanced metric relative to $\xi\in\Lie\Aut(X)$, and using
the quantized Futaki invariant \cref{quantFutdefintro} as an obstruction
for their existence, we show that the vector field $\xi$ is determined by the
complex geometry of $X$.

In the whole Section, we consider
$p\in\N^*$ big enough so that the Kodaira map
\cref{Kod} is well-defined and an embedding, and
fix a compact torus $T\subset\Aut_0(X)$. We will use freely
the decomposition \cref{LieKC} for $K=T$.

\subsection{Fubini-Study metrics}
\label{findimsec}

Let $H$ be a $T$-invariant Hermitian inner product on $H^0(X,L^p)$,
which always exists by average over $T$.
For all $\xi\in\sqrt{-1}\Lie T$,
the operators $L_\xi\in\End(H^0(X,L^p))$ induced by formula \cref{Lxidef}
are then Hermitian with respect to $H$,
so that they admit a joint spectrum $\Spec_p(T)\subset(\Lie T)^*$
not depending on $H$. For any $\chi\in\Spec_p(T)$, write
\begin{equation}\label{H0chi}
H^0(X,L^p)_{\chi}:=\{s\in H^0(X,L^p)~|~L_\xi\,s=(\chi,\xi)\,s
~~\text{for all}~~\xi\in\Lie T\}\,.
\end{equation}
Write $\BB(H^0(X,L^p)_\chi)$ for the space of bases of $H^0(X,L^p)_{\chi}$,
and set
\begin{equation}\label{BBTdef}
\BB(H^0(X,L^p))^T:=
\prod\limits_{\chi\in\Spec_p(T)}\BB(H^0(X,L^p)_{\chi})\,.
\end{equation}
For any $\chi\in\Spec_p(T)$,
%$\s\in\BB(H^0(X,L^p))^T$ and $\chi\in\Spec_p(T)$, we write $\s_\xi$
%for its component in $\BB(H^0(X,L^p)_{\chi})$.
write $n_p(\chi):=\dim H^0(X,L^p)_{\chi}$. The space $\BB(H^0(X,L^p))^T$
admits a free and transitive action of the group
\begin{equation}\label{GLT}
\GL(\C^{n_p})^T:=\bigotimes_{\chi\in\Spec_p(T)}\GL(\C^{n_p(\chi)})\,,
\end{equation}
acting component by component by the canonical action of
$\GL(\C^{n_p(\chi)})$ on bases of $H^0(X,L^p)_\chi$, for all
$\chi\in\Spec_p(T)$.
%by the following formula, for any $G\in\GL(\C^{n_p(\chi)})$
%and $\s=\{s_j\}_{j=1}^{n_p(\chi)}\in\BB(H^0(X,L^p)_\chi)$,
%\begin{equation}\label{GLact}
%G.\s:=\left\{\sum_{k=1}^{n_p(\chi)} G_{jk}s_k\right\}_{j=1}^{n_p(\chi)}\,.
%\end{equation}
Note that for any $\xi\in\Lie T_\C$,
we have $e^{L_\xi}\in\GL(\C^{n_p})^T$ acting in a canonical way as
a scalar on each component. 
For any $n\in\N$, we write
$U(n)\subset\GL(\C^n)$ for the subgroup of unitary matrices
acting on $\C^n$, and we set
\begin{equation}\label{UT}
U(n_p)^T:=\bigotimes_{\chi\in\Spec_p(T)}U(n_p(\chi))
\subset\GL(\C^{n_p})^T\,.
\end{equation}

To any $\s\in\BB(H^0(X,L^p))^T$, we can associate a basis
$\{s_j\}_{j=1}^{n_p}$ of $H^0(X,L^p)$, uniquely determined up to reordering
by the condition that it restricts to the
corresponding basis of $H^0(X,L^p)_\chi$ for each $\chi\in\Spec_p(T)$.
We write $H_\s$ for the unique
$T$-invariant inner product on
$H^0(X,L^p)$ for which $\{s_j\}_{j=1}^{n_p}$ is orthonormal.
Conversely,
we say that $\s\in\BB(H^0(X,L^p))^T$ is
\emph{orthonormal} with respect to a Hermitian inner product
$H$ on $H^0(X,L^p)$ if $\{s_j\}_{j=1}^{n_p}$ is, so that $H_\s=H$.

We are now ready to introduce the main definition of the Section.

\begin{defi}\label{FSdef}
For any $\s\in\BB(H^0(X,L^p))^T$,
the associated \emph{Fubini-Study metric} $h_\s^p\in\Met^+(L^p)$
is characterized for any $s_1,\,s_2\in H^0(X,L^p)$ and $x\in X$
by the formula
\begin{equation}\label{hFSdef}
\<s_1(x),s_2(x)\>_{h_\s^p}:=\<\Pi_\s(x)\,s_1,\,s_2\>_{H_\s}\,,
\end{equation}
where $\Pi_\s(x)$ is the unique orthogonal
projector with respect to $H_\s$ satisfying
\begin{equation}\label{Pisfla}
\Ker\Pi_\s(x)=\{\,s\in H^0(X,L^p)~|~s(x)=0\,\}\,.
\end{equation}
\end{defi}
That formula \cref{hFSdef} defines a positive Hermitian metric
is a consequence of the fact that the Kodaira map \cref{Kod} is an
embedding.

For any $T$-invariant Hermitian product $H$, write $\cL(H^0(X,L^p),H)^T$
for the space of Hermitian operators with respect to $H$
commuting with the action of $T$.
Via the action of $\GL(\C^{n_p})^T$ on $\BB(H^0(X,L^p))^T$,
any given $\s\in\BB(H^0(X,L^p))^T$ induces an identification
\begin{equation}\label{isomHermT}
\cL(H^0(X,L^p),H_\s)^T\simeq\Herm(\C^{n_p})^T:=
\bigoplus_{\chi\in\Spec_p(T)}\Herm(\C^{n_p(\chi)})\,.
\end{equation}
In particular, for any $\xi\in\sqrt{-1}\Lie T$,
we have $L_\xi\in\Herm(\C^{n_p})^T$ not depending on
$\s\in\BB(H^0(X,L^p))^T$.
We then have the following basic variation formula
for Fubini-Study metrics.

\begin{prop}\label{FSvar}
For any $\textbf{s}\in\BB(H^0(X,L^p))^T$ and
$A\in\cL(H^0(X,L^p),H_\s)^T$, set
\begin{equation}\label{sigmasdef}
\sigma_\s(A):=\Tr[A\Pi_\s]\in\cinf(X,\R)\,.
\end{equation}
Then for any $B\in\Herm(\C^{n_p})^T$, we have
\begin{equation}\label{hpeBsfla}
\sigma_{\s}(e^{2B})\,h^p_{e^B\s}=h^p_{\s}\,,
\end{equation}
in the identification \cref{isomHermT} induced by $\s\in\BB(H^0(X,L^p))^T$.
\end{prop}
\begin{proof}
First note that for any
$\s\in\BB(H^0(X,L^p))^T$, writing $\{s_j\}_{j=1}^{n_p}$ for
an induced basis of $H^0(X,L^p)$,
\cref{FSdef} implies
\begin{equation}\label{sumprop}
\sum_{j=1}^{n_p}|s_j|^2_{h^p_\s}=
\sum_{j=1}^{n_p}\<\Pi_{H_\s}s_j,s_j\>_{H_\s}
=\Tr[\Pi_{H_\s}]=1\,,
\end{equation}
and this formula characterizes $h^p_\s\in\Met^+(L^p)^T$.
On the other hand, for any $B\in\cL(H^0(X,L^p),H_\s)$, we have
\begin{equation}\label{G*GBersymb}
\begin{split}
\sigma_{H_\s}(e^{2B})
&=\Tr[e^B\Pi_{H_\s}e^B]\\
&=\sum_{j=1}^{n_p}\<\Pi_{H_\s}e^Bs_j,e^Bs_j\>_{H_\s}
=\sum_{j=1}^{n_p}\left|e^{B}s_j\,\right|_{h^p_\s}^2\,,
\end{split}
\end{equation}
which gives the result by the characterization \cref{sumprop}
applied to both $h_\s$ and $h_{e^B\s}$.
\end{proof}

\begin{rmk}\label{h=rhohsrmk}
Let $K$ be a compact Lie group containing $T$ in its center,
and let $h^p\in\Met^+(L^p)^K$ be a $K$-invariant positive Hermitian metric.
Then the associated $L^2$-Hermitian product $L^2(h^p)$ given by
formula \cref{L2}
is also $K$-invariant, and there exists
$\s_{p}\in\BB(H^0(X,L^p))^T$ orthonormal with respect to $L^2(h^p)$.
Furthermore, the orthogonal projector
$\Pi_{\s_p}(x)$ of \cref{FSdef} coincides
with the coherent state projector of \cref{cohstateprojdef} at $x\in X$,
so that
\begin{equation}\label{h=rhohs}
h^p=\rho_{h^p}\,h_{\s_p}^p\,,
\end{equation}
and the function $\sigma_{\s_p}(A)\in\cinf(X,\R)$
defined by formula \cref{sigmasdef} for any $A\in\cL(\HH_p)$ commuting
with the action of $T$, coincides with its
Berezin symbol as in \cref{BTquantdef}.
\end{rmk}
%
%and through the identification \cref{isomEnd} for any
%$\s\in\BB(H^0(X,L^p))^T$,
%we get an identification
%\begin{equation}\label{isomEndT}
%T_\s\BB(H^0(X,L^p))^T\simeq\bigoplus_{\chi\in\Spec_p(T)}\End(\C^{n_p(\chi)})\,.
%\end{equation}
Let us end this section by the following variation formula for the Berezin
symbol with respect to a change of basis

\begin{prop}\label{FSvar2}
For any $\s\in\BB(H^0(X,L^p))^T$,
any $B\in\Herm(\C^{n_p})^T$ and any $A\in\cL(H^0(X,L^p),H_{e^B\s})^T$,
in the identification \cref{isomHermT} induced by $\s$,
we have
\begin{equation}\label{FSvarfla}
\sigma_{e^B\s}(A)\sigma_{\s}(e^{2B})=\sigma_{\s}(e^BAe^B)\,.
\end{equation}
In particular, for any $\xi\in\sqrt{-1}\Lie T$, we have
$\sigma_\s(e^{2L_\xi}A)=\sigma_{e^{L_\xi}\s}(A)\sigma_\s(e^{2L_\xi})$.
\end{prop}
\begin{proof}
Let $\s\in\BB(H^0(X,L^p))^T$ be given, and let $\{s_j\}_{j=1}^{n_p}$
be an induced basis of $H^0(X,L^p)$.
Then by \cref{FSdef}, the projector $\Pi_{\s}(x)$
at $x\in X$ of can be written in this basis as
\begin{equation}\label{cohstatecoord}
\Pi_{\s}(x)=\Big(\<s_j(x),s_k(x)\>_{h_\s^p}
\Big)_{j,\,k=1}^{n_p}\,.
\end{equation}
Take now $B\in\Herm(\C^{n_p})^T$,
%and write
%$(G_{jk})_{j,k=1}^{n_p}\in\GL(\C^{n_p})$ and
%$(A_{jk})_{j,k=1}^{n_p}\in\Herm(\C^{n_p})$ for the matrices
%of $e^B$ and $A$,
seen as a Hermitian operator with respect to $H_\s$ via the identification
\cref{isomHermT}, and take $A\in\cL(H^0(X,L^p),H_{e^B\s})^T$.
Writing them in the basis $\{s_j\}_{j=1}^{n_p}$,
using \cref{FSdef} and \cref{FSvar}, we then get
\begin{equation}
\begin{split}
\sigma_{e^B\s}(A)&=\sum_{j,k=1}^{n_p}
\<Ae^B s_k,e^B s_j\>_{h^p_{e^B\s}}\\
&=\sigma_{\s}(e^{2B})^{-1}
\sum_{l,m=1}^{n_p}\left(e^BAe^B\right)_{ml}\<s_l,s_m\>_{h^p_{\s}}
=\sigma_{\s}(e^{2B})^{-1}\,\sigma_{\s}(e^BAe^B)\,.
\end{split}
\end{equation}
This gives the result.
\end{proof}

\subsection{Relative balanced metrics}
\label{relbalsec}

In this Section, we introduce the notion of
\emph{relatively balanced metrics},
and we exhibit their role as a quantized version of
Kähler-Ricci solitons.
In particular, we will establish a quantized version
of \cref{TZKR} of Tian and Zhu for the quantized Futaki invariant.

Recall that we write
$h_\s^p\in\Met^+(L^p)^T$ for the Fubini-Study metric of
\cref{FSdef} associated with any $\s\in\BB(H^0(X,L^p))^T$,
and recall that for any $\xi\in\Lie\Aut(X)$, we write
$\phi_\xi\in\Aut(X)$ for its exponentiation.
We will need the following equivariance property.

\begin{prop}\label{pullbackprop}
For any $\s\in\BB(H^0(X,L^p))^T$ and any $\eta\in\Lie\Aut(X)$ such that
$L_\eta\in\cL(H^0(X,L^p),H_\s)^T$, we have
\begin{equation}\label{pullback}
\phi_{\eta}^*\,h_{\s}=h_{e^{L_\eta}\s}\quad\text{and}\quad
\phi_\eta^*\,\sigma_\s(A)=\sigma_{e^{L_\eta}\s}(e^{-L_\eta}Ae^{L_\eta})\,,
\end{equation}
for any $A\in\cL(H^0(X,L^p),H_\s)^T$.

Furthermore, if $H_\s$ is preserved by a connected
subgroup $K\subset\Aut_0(X)$,
the Fubini-Study metric $h_\s^p\in\Met^+(L^p)$
is $K$-invariant.
\end{prop}
\begin{proof}
Let $\s\in\BB(H^0(X,L^p))^T$ be given,
and let $\eta\in\Lie\Aut(X)$ be such that
$L_\eta\in\cL(H^0(X,L^p),H_\s)^T$.
For any $s_1,\,s_2\in H^0(X,L^p)$, we have by definition that
$\<s_1,s_2\>_{H_{e^{L_\eta}\s}}=\<e^{-L_\eta}s_1,e^{-L_\eta}s_2\>_{H_\s}$,
and for any $x\in X$, we have the identity
\begin{equation}\label{Pixi=xiPixi}
\Pi_{e^{L_\eta}\s}(x)=e^{L_\eta}\Pi_{\s}(\phi_\eta(x))e^{-L_\eta}\,,
\end{equation}
which follows from the fact that
both sides are orthogonal projectors
with respect to $H_{e^{L_\eta}\s}$, and have common kernel
by formula \cref{Pisfla}. Plugging these two identities in
formula \cref{hFSdef} for $h_{e^{L_\eta}\s}^p$
and using formula \cref{Lxidef} for
$L_\eta$, we get
\begin{equation}
\begin{split}
\<s_1(x),s_2(x)\>_{h_{e^{L_\eta}\s}^p}
&=\<\Pi_{\s}(\phi_\eta(x))e^{-L_\eta}s_1,e^{-L_\eta}s_2\>_{H_\s}\\
&=\<e^{-L_\eta}s_1(\phi_\eta(x)),e^{-L_\eta}s_2(\phi_\eta(x))\>_{h_\s^p}\,,
\end{split}
\end{equation}
which gives the first identity of \cref{pullback}
by definition of the pullback
of an Hermitian metric.
On the other hand, from formula
\cref{Pixi=xiPixi} we get
\begin{equation}
\sigma_\s(A)(\phi_{\eta}(x))=\Tr[e^{-L_\eta}\Pi_{e^{L_\eta}\s}(x)e^{L_\eta}A]=
\sigma_{e^{L_\eta}\s}(e^{-L_\eta}Ae^{L_\eta})(x)\,.
\end{equation}
This gives the second identity of \cref{pullback} and concludes the proof.
Finally, the fact that $h_{\s}^p$ is $K$-invariant
when $K\subset\Aut_0(X)$ preserves $H_\s$
then follows from \cref{FSdef}, as $\Pi_\s$ only depends on $H_\s$.
\end{proof}

To simplify notations, let us write $\om_\s:=\om_{h_\s}$ for the
Kähler form form associated with the Fubini-Study metric induced by
$\s\in\BB(H^0(X,L^p))^T$.
From \cref{pullbackprop}, a
positive Hermitian metric $h^p\in\Met^+(L^p)^T$ is
anticanonically balanced relative to $\xi\in\Lie\Aut(X)$ in the sense
of formula \cref{relbaldef}
if $\xi\in\sqrt{-1}\Lie T$
and if for any $\s_p\in\BB(H^0(X,L^p))^T$ orthonormal with respect to $L^2(h^p)$, we have
\begin{equation}\label{relbaldef2}
\om_{h^p}=\om_{e^{L_\xi/2p}\s_p}\,.
\end{equation}
%Note that this implies in particular that $\xi\in\sqrt{-1}\Lie T$
%in the decomposition \cref{LieKC}. In the sequel, we will
%fix such a compact torus $T\subset\Aut(X)$.
We then have the following useful alternative characterization of
relatively balanced metrics.

\begin{prop}\label{relbalrho}
A positive Hermitian metric $h^p\in\Met^+(L^p)^T$ is anticanonically balanced
relative to $\xi\in\sqrt{-1}\Lie T$ if and only if the associated
Rawnsley function satisfies
\begin{equation}
\sigma_{h^p}(e^{L_\xi/p})\rho_{h^p}=
\frac{\Tr[e^{L_\xi/p}]}{\Vol(d\nu_h)}\,.
\end{equation}
\end{prop}
\begin{proof}
Consider $h^p\in\Met^+(L^p)^T$, and let $s_p\in\BB(H^0(X,L^p))^T$ be
orthonormal with respect to $L^2(h^p)$ as in \cref{h=rhohsrmk}.
Using also \cref{FSvar}, we have
\begin{equation}\label{hp=rhosighs}
h^p=\rho_{h^p}\,h^p_{\s_p}=\rho_{h^p}\,
\sigma_{h^p}(e^{L_\xi/p})\,h_{e^{L_\xi/2p}\s_p}^p\,.
\end{equation}
Using \cref{relbaldef}, we then get that $h^p\in\Met^+(L^p)^T$
is anticanonically
balanced relative to $\xi\in\sqrt{-1}\Lie T$ if and only if the function
$\rho_{h^p}\,\sigma_{h^p}(e^{L_\xi/p})\in\cinf(X,\R)$
is constant over $X$. To compute
this constant, it suffices to note that \cref{dualprop} implies
\begin{equation}\label{rhosig=tre}
\int_X\,\rho_{h^p}\,\sigma_{h^p}(e^{L_\xi/p})\,d\nu_h=\Tr[e^{L_\xi/p}]\,.
\end{equation}
This gives the result.
\end{proof}

Using this characterization of relative anticanonically balanced
metrics together with the tools of \cref{BTsec,quantholpotsec},
we can now give a short proof of the following key fact.
%,
%which can also be established as a consequence of
%\cite[Prop.\,4.7,\,4.9]{BW14}.

\begin{prop}\label{Futpropintro}
If
there exists an anticanonically balanced metric $h^p\in\Met^+(L^p)$
relative to $\xi\in\Lie\Aut(X)$, then the quantized Futaki
invariant $\Fut^\xi_p:\Lie\Aut(X)\to\C$
vanishes identically.
\end{prop}
%\begin{prop}\label{acimpliesFut=0}
%Let $p\in\N^*$ be such that $X$ admits an anticanonically balanced metric
%relative to $\xi\in\sqrt{-1}\Lie T$, for some compact torus
%\todo{bonne formulation?}
%$T\subset\Aut(X)$. Then for all $\eta\in\Lie\Aut(X)$,
%we have
%\begin{equation}
%\Fut^\xi_p(\eta)=0\,.
%\end{equation}
%\end{prop}
%\begin{proof}
\begin{proof}
Let $h^p\in\Met^+(L^p)^T$ be anticanonically balanced metric relative
to $\xi\in\sqrt{-1}\Lie T$.
Then by definition \cref{quantFutdefintro} of the 
quantized Futaki invariant relative to $\xi$, using
\cref{thetaxiprop,dualprop,Tuy,relbalrho}, for any $\eta\in\Lie\Aut(X)$
we get
\begin{equation}
\begin{split}
\frac{\Fut^\xi_p(\eta)}{p+1}=\Tr[T_{h^p}(\theta_h(\eta)) e^{L_\xi/p}]
&=\int_X\,\theta_h(\eta)\,\sigma_{h^p}(e^{L_\xi/p})\,\rho_{h^p}\,d\nu_h\\
&=\frac{\Tr[e^{L_\xi/p}]}{\Vol(d\nu_h)}\int_X\,\theta_h(\eta)\,d\nu_h=0\,.
\end{split}
\end{equation}
This shows the result.
\end{proof}

%\begin{rmk}
%\todo{remarque nécessaire?}
%As explained in \cite[Rmk.\,2.8]{BW14},
%the vanishing of the quantized Futaki invariants
%\cref{quantFutdefintro} for $\xi=0$ and all $p\in\N^*$ big enough implies the
%vanishing of the \emph{higher order Futaki invariants} \cite{Fut04} of $X$.
%On the other
%hand, the definition \cref{relbalsec} of on anticanonically balanced metric relative to $\xi=0$ reduces to the
%usual definition of an anticanonically balanced metric.
%Thus \cref{Futpropintro} implies that the higher Futaki invariants
%are an obstruction for the existence of anticanonically balanced metrics,
%for all $p\in\N^*$ big enough. This fact is illustrated by \cref{corintro}
%in the Introduction.
%\end{rmk}

\subsection{Relative moment maps}
\label{relmomsec}

In this section, we give the finite dimensional characterization
of relative balanced metrics, using a relative version of
Donaldson's moment map picture in \cite{Don01}.
For any $\s\in\BB(H^0(X,L^p))^T$, let us write $d\nu_\s:=d\nu_{h_\s}$
for the anticanonical volume form \cref{dnucandef} induced by the
associated Fubini-Study metric.

\begin{defi}\label{momentdef}
The \emph{anticanonical moment map} relative to $\xi\in\sqrt{-1}\Lie T$
is the map $\mu_\xi:\BB(H^0(X,L^p))^T\to\Herm(\C^{n_p})^T$
defined for all $\textbf{s}\in\BB(H^0(X,L^p))^T$
by the formula
\begin{equation}\label{muxifla}
\mu_\xi(\textbf{s})
:=\left(\int_X\,\<s_j(x),s_k(x)\>_{e^{L_\xi/2p}\s}\,
d\nu_{e^{L_\xi/2p}\textbf{s}}(x)\right)_{j,\,k=1}^{n_p}
-\frac{\Vol(d\nu_{\textbf{s}})}{\Tr[e^{L_\xi/p}]}\Id\,,
\end{equation}
where $\{s_j\}_{j=1}^{n_p}$ is an induced basis of $H^0(X,L^p)$.
\end{defi}

Note that we wrote formula \cref{muxifla} as an element of
$\Herm(\C^{n_p})$ instead of $\Herm(\C^{n_p})^T$.
However,
\cref{pullbackprop} shows that $L^2(h_{e^{L_\xi/2p}\s})$ is $T$-invariant,
so that
the right hand side of \cref{muxifla} splits into blocks corresponding
to the eigenspaces \cref{H0chi} of the action of $T$ on $H^0(X,L^p)$,
giving an element of $\Herm(\C^{n_p})^T$ as in formula \cref{isomHermT}
depending only on
$\textbf{s}\in\BB(H^0(X,L^p))^T$. This identification will always be
implicitly understood in the sequel.

Note that we do not claim that \cref{momentdef} defines an actual
relative moment map in the usual sense, and we will consequently not
use any moment map property as such anywhere in this paper.
We will however stick to
this name, as it is the anticanonical analogue of the relative
moment map considered by Sano and Tipler in \cite[\S\,3.3]{ST17}.
%, and is
%the moment map naturally associated with the
%Donaldson's iterations considered by Berman and Witt Nyström
%in \cite[\S\,4.3]{BW14}.
Its relevance in the context of balanced metrics comes
from the following basic result, which follows immediately from the
definition.

\begin{prop}\label{momentbal}
For any $\s\in\BB(H^0(X,L^p))^T$ and $\xi\in\sqrt{-1}\Lie T$,
we have
\begin{equation}
\mu_\xi(\s)=0
\end{equation}
if and only if there exists an anticanonically balanced metric
$h^p\in\Met^+(L^p)^T$ relative to $\xi$ for which $\s$ is orthonormal
with respect to $L^2(h^p)$.
\end{prop}
\begin{proof}
Writing
\begin{equation}\label{momentbalfla1}
h^p:=\frac{\Vol(d\nu_{\textbf{s}})}{\Tr[e^{L_\xi/p}]}\,h_{e^{L_\xi/p}\s}^p\,,
\end{equation}
\cref{momentdef} shows that
$\s$ is orthonormal with respect to $L^2(h^p)$
if and only if $\mu_\xi(\s)=0$.
Hence formula \cref{momentbalfla1} for
$h^p\in\Met^+(L^p)^T$ implies formula
\cref{relbaldef2} for an anticanonically
balanced metric with respect to $\xi$.
This gives the result.
\end{proof}

For any $\xi\in\sqrt{-1}\Lie T$, consider the
scalar product $\<\cdot,\cdot\>_{\xi}$
defined on any $A,\,B\in\Herm(\C^{n_p})^T$ by
\begin{equation}\label{Trxi}
\<A,B\>_{\xi}:=\Tr[e^{L_\xi/p}AB]\,.
\end{equation}
We then have the following important obstruction result, which is
compatible with \cref{Futpropintro} via \cref{momentbal}.
%This can also be deduced by looking at the time derivative
%of the quantized energy functional introduced in \cite[\S\,4.4.2]{BW14}.

\begin{prop}\label{Trmu}
For any $\xi\in\sqrt{-1}\Lie T$ and $\s\in\BB(H^0(X,L^p))^T$, we have
$\<\Id,\mu_\xi(\textbf{s})\>_{\xi}=0$, and
for any $\eta\in\sqrt{-1}\Lie T$, we have
\begin{equation}\label{mu=Fut}
p\<L_\eta,\mu_\xi(\textbf{s})\>_{\xi}
=-\frac{\Vol(d\nu_{\textbf{s}})}{\Tr[e^{L_\xi/p}]}\Fut^{\xi}_p(\eta)\,.
\end{equation}
Furthermore, for any connected compact subgroup $K\subset\Aut_0(X)$
preserving $H_\s$ and containing $T\subset K$ in its center,
we have that $\mu_\xi(\textbf{s})\in\Herm(\C^{n_p})^T$
commutes with the action of $K$ on $H^0(X,L^p)$
via the identification \cref{isomHermT}.
\end{prop}
\begin{proof}
Let $\xi\in\sqrt{-1}\Lie T$ and
$\textbf{s}\in\BB(H^0(X,L^p))^T$ be given, and let
$\{s_j\}_{j=1}^{n_p}$ be an induced basis of $H^0(X,L^p)$.
Using formula
\cref{cohstatecoord} for the basis $\{e^{L_\xi/2p}s_j\}_{j=1}^{n_p}$,
the coherent state projector $\Pi_{e^{L_\xi/2p}\s}(x)$ at $x\in X$
in the basis $\{s_j\}_{j=1}^{n_p}$ reads
\begin{equation}\label{cohstatecoordxi}
e^{-L_\xi/2p}\Pi_{e^{L_\xi/2p}\s}(x)e^{-L_\xi/2p}
=\Big(\<s_j(x),s_k(x)\>_{h_{e^{L_\xi/2p}\s}^p}
\Big)_{j,\,k=1}^{n_p}\,.
\end{equation}
Using \cref{pullbackprop} together with formula \cref{etxi*dnu},
the fact that $\Pi_{e^{L_\xi/2p}\s}(x)$ is a rank-$1$ projector implies
\begin{equation}
\Tr[e^{L_\xi/p}\mu_\xi(\textbf{s})]=\int_X\,\Tr[\Pi_{e^{L_\xi/2p}\s}(x)]\,
d\nu_{e^{L_\xi/2p}\s}(x)-\Vol(d\nu_\s)=0\,.
\end{equation}
This proves the first assertion.

On the other hand, using \cref{FSvar,pullbackprop}, for
all $\eta\in\sqrt{-1}\Lie T$,
formula \cref{etxi*dnu} implies
\begin{equation}
\int_X\,\sigma_\s(L_\eta)\,d\nu_{\s}
=p\,\dt\Big|_{t=0}\int_X\,\phi_{t\eta}^*\,d\nu_h=0\,.
\end{equation}
From formulas \cref{quantFutdefintro} and \cref{cohstatecoordxi}, 
this gives
\begin{equation}
\begin{split}
\Tr[e^{L_\xi/p}L_\eta\,\mu_\xi(\textbf{s})]
&=\int_X\,\sigma_{e^{L_\xi/2p}\s}(L_\eta)\,d\nu_{e^{L_\xi/2p}\s}-\frac{\Vol(d\nu_{\textbf{s}})}{\Tr[e^{L_\xi/p}]}\Fut^{\xi}_p(\eta)\\
&=-\frac{\Vol(d\nu_{\textbf{s}})}{\Tr[e^{L_\xi/p}]}\Fut^{\xi}_p(\eta)\,.
\end{split}
\end{equation}
Finally, let $K\subset\Aut_0(X)$ be a connected
compact subgroup preserving $H_\s$
containing $T\subset K$ in its center, and recall from
\cref{pullbackprop} that $h_{\s}\in\Met^+(L)^K$, and
for any $\eta\in\Lie K$, we have
$\Pi_{e^{L_\eta}\s}=\Pi_\s$. Then
using formulas \cref{etxi*dnu},
\cref{Pixi=xiPixi} and \cref{cohstatecoordxi},
in the identification \cref{isomHermT} we get
\begin{equation}
\begin{split}
e^{-L_\eta}\mu_\xi(\s)e^{L_\eta}&=\int_X\,
e^{-L_\xi/2p}e^{-L_\eta}\Pi_{e^{L_\xi/2p}\s}(x)e^{L_\eta}e^{-L_\xi/2p}\,
d\nu_{\s}-\frac{\Vol(d\nu_{\textbf{s}})}{\Tr[e^{L_\xi/p}]}\Id\\
&=\int_X\,
e^{-L_\xi/2p}\Pi_{e^{L_\xi/2p}\s}(\phi_{\eta}(x))e^{-L_\xi/2p}\,
d\nu_{\s}-\frac{\Vol(d\nu_{\textbf{s}})}{\Tr[e^{L_\xi/p}]}\Id\\
&=\mu_\xi(\s)\,,
\end{split}
\end{equation}
where we used a change of variable with respect to $\phi_\eta\in\Aut_0(X)$
to get the last line. This concludes the proof.
\end{proof}

\section{Equivariant Berezin-Toeplitz quantization}
\label{BTsecT}

In this Section, we fix a connected
compact subgroup $K\subset\Aut_0(X)$, and write
$T\subset K$ for the identity component of its center.
We will study the properties of the
Berezin-Toeplitz quantization of \cref{BTsec} with respect to
the action of $K$.
Let $p\in\N^*$ be large enough so
that the Kodaira map \cref{Kod} is an embedding, and consider the setting or
\cref{BTsec}
for a $K$-invariant positive Hermitian metric $h^p\in\Met^+(L^p)^K$.

\subsection{Quantum channel}
\label{quantchannsec}

Recall that we write $\cinf(X,\R)^K$ for the space of
$K$-invariant functions on $X$, and
$\cL(\HH_p)^K$ for the space of Hermitian operators
commuting with the action of $K$ on $\HH_p$.

\begin{lem}\label{sigmainv}
The symbol and quantization maps introduced in
\cref{BTquantdef} restrict to linear maps
$\sigma_{h^p}:\cL(\HH_p)^K\to\cinf(X,\R)^K$
and $T_{h^p}:\cinf(X,\R)^K\to\cL(\HH_p)^K$.
\end{lem}
\begin{proof}
Following \cref{h=rhohsrmk}, let
$\s_p\in\BB(H^0(X,L^p))^T$ be orthonormal with respect to $L^2(h^p)$.
As $K$ preserves $L^2(h^p)$ and as the coherent state projector
of \cref{FSdef} only depends on $H_{\s_p}=L^2(h^p)$, we have
$\Pi_{e^{L_\eta}\s_p}=\Pi_{\s_p}=\Pi_{h^p}$ for all $\eta\in\Lie K$.
\cref{pullbackprop} then implies that $\sigma_{h^p}(A)\in\cinf(X,\R)^K$
for all $A\in\cL(\HH_p)^K$. This shows that the Berezin symbol
restricts to a map $\sigma_{h^p}:\cL(\HH_p)^K\to\cinf(X,\R)^K$ .

On the other hand, \cref{pullbackprop} shows that
$h^p_{\s_p}\in\Met^+(L^p)$ is $K$-invariant,
and formula \cref{h=rhohs}
then implies that $\rho_{h^p}\in\cinf(X,\R)^K$.
Thus for any $\eta\in\Lie K$ and $f\in\cinf(X,\R)^K$,
we can use formula \cref{Pixi=xiPixi} and
a change of variable with respect to
$\phi_{\xi}\in\Aut_0(X)$ to get
\begin{equation}
\begin{split}
T_{h^p}(f)&=\int_X f(\phi_{\eta}(x))
\,\Pi_{h^p}(\phi_{\eta}(x))\,
\rho_{h^p}(\phi_{\eta}(x))\,\phi_{\eta}^*\,d\nu_h(x)\\
&=\int_X f(x)\,e^{L_\eta}\Pi_{h^p}(x)e^{-L_\eta}\,
\rho_{h^p}(x)\,d\nu_h(x)=e^{L_\eta}T_{h^p}(f)e^{-L_\eta}\,,
\end{split}
\end{equation}
so that $T_{h^p}(f)\in\cL(\HH_p)^K$ for all $f\in\cinf(X,\R)^K$.
This concludes the proof.
\end{proof}

From now on, we fix $\xi\in\sqrt{-1}\Lie T$.
Recall the scalar product \cref{Trxi} on
$\cL(\HH_p)^K\simeq\Herm(\C^{n_p})^K$, and consider
the scalar product $L^2(h,\xi,p)$
defined on any $f,\,g\in\cinf(X,\R)^K$ by
\begin{equation}\label{L2hpxi}
\<f,g\>_{L^2(h,\xi,p)}:=
\int_X\,f\,g\,
\frac{\sigma_{h^p}(e^{L_\xi/p})\,\rho_{h^p}}{\Tr[e^{L_\xi/p}]}\,d\nu_h\,.
\end{equation}
The following result is the equivariant version of the duality \cref{dualprop} between Berezin symbol and Berezin-Toeplitz quantization.

\begin{prop}\label{dualpropT}
For any $A\in\cL(\HH_p)^K$ and $f\in\cinf(X,\R)^K$, we have
\begin{equation}\label{dualflaT}
\frac{\<T_{h^p}(f),A\>_{\xi}}{\Tr[e^{L_\xi/p}]}
=\<f,\phi_{\xi/2p}^*\sigma_{h^p}(A)\>_{L^2(h,\xi,p)}\,.
\end{equation}
%are dual to each other with respect to the scalar product $L^2(\alpha_{\xi,p})$
%on $\cinf(X,\R)^T$ induced by \cref{alphaxip} and the scalar product
%\cref{Trxi} on $\cL(H^0(X,L^p),H_\s)^T$ respectively.
\end{prop}
\begin{proof}
Following \cref{h=rhohsrmk}, let
$\s_p\in\BB(H^0(X,L^p))^T$ be orthonormal with respect to $L^2(h^p)$.
Then using \cref{dualprop,FSvar2,pullbackprop},
for any $A\in\cL(\HH_p)^K$ and $f\in\cinf(X,\R)^K$,
we get
\begin{equation}\label{duaxilfla}
\begin{split}
\Tr[e^{L_\xi/p}AT_{h^p}(f)]&=\int_X\,\sigma_{h^p}(e^{L_\xi/p}A)\,\rho_{h^p}
\,d\nu_{h}\\
&=\int_X\,\sigma_{e^{L_\xi/p}\s_p}(A)\,\sigma_{h^p}(e^{L_\xi/p})
\,\rho_{h^p}\,d\nu_{h}\\
&=\int_X\,\phi_{\xi/2p}^*\,\sigma_{h^p}(A)\,\sigma_{h^p}(e^{L_\xi/p})
\,\rho_{h^p}\,d\nu_{h}\,.
\end{split}
\end{equation}
This gives the result.
\end{proof}
%,
%\cref{dualpropT} states that the map
%$\phi_{\xi/2p}^*\,\sigma_{h^p}:\cL(\HH_p)^T\to\cinf(X,\R)^T$
%\todo{vraiment nécessaire cette remarque?}
%is dual to the
%Berezin-Toeplitz quantization
%with respect to the respective scalar products \cref{Trxi} and
%\cref{L2hpxi}.
We can now introduce the main tool of this Section.

\begin{defi}\label{quantchanrel}
For any $\xi\in\sqrt{-1}\Lie T$,
the \emph{Berezin-Toeplitz quantum channel relative to} $\xi$
is the linear map defined by
\begin{equation}
\begin{split}
\EE_{h^p}^\xi:\cL(\HH_p)^K&\longrightarrow\cL(\HH_p)^K\\
A~&\longmapsto~T_{h^p}\left(\phi_{\xi/2p}^*\,\sigma_{h^p}
\left(A\right)\right)\,.
\end{split}
\end{equation}
\end{defi}

\cref{dualpropT} the shows that the quantum channel relative to
$\xi\in\sqrt{-1}\Lie T$ is a positive and self-adjoint operator
acting on the real Hilbert space $\cL(\HH_p)^K$
endowed with the scalar product
$\<\cdot,\cdot\>_{\xi}$ defined by formula \cref{Trxi}.

\subsection{Berezin transform}
\label{Bertranssec}

%Fix a compact torus $T\subset\Aut(X)$, and consider the
%setting of \cref{BTsec} for a $T$-invariant
%positive Hermitian metric $h\in\Met^+(L)^T$.
%\todo{fixer $h$ au lieu de $h^p$}
The goal of this section is to extend the
results of \cite{IKPS19} on the Berezin transform of \cref{Bertransdef}
to the equivariant setting of \cref{BTsecT}.
For any $\xi\in\sqrt{-1}\Lie T$, we consider the linear isomorphisms
$\phi^*_\xi:\cinf(X,\R)^K\to\cinf(X,\R)^K$
by pullback with respect to $\phi_\xi\in T_\C$.
The following basic result draws
a link with the quantum channel of \cref{quantchanrel}.

\begin{prop}\label{specidprop}
For any $\xi\in\sqrt{-1}\Lie T$,
the linear map
\begin{equation}\label{BertransT}
\phi_{\xi/2p}^*\,\cB_{h^p}:\cinf(X,\R)^K\longrightarrow\cinf(X,\R)^K
\end{equation}
is a positive and self-adjoint operator with respect to
the scalar product $L^2(h,\xi,p)$ given by formula \cref{L2hpxi}.
\todo{c'est aussi un opérateur de Markov! Le dire, et écrire
$\Spec\subset[0,1]$}

Furthermore, the positive spectrums of $\phi_{\xi/2p}^*\,\cB_{h^p}$
and $\EE^{\xi}_{h^p}$ coincide.
% and the
%equivariant Berezin-Toeplitz quantization map
%$T_{h^p}:\cinf(X,\R)^K\to\cL(\HH_p)^K$ establishes a bijection between the
%associated eigenspaces.
\end{prop}
\begin{proof}
%Fix $h^p\in\Met^+(L^p)^T$, and let $\s_p\in\BB(H^0(X,L^p))^T$
%be orthonormal with respect to $L^2(h^p)$.
%For any $f\in\cinf(X,\R)^T$, using \cref{sigmainv} and
%$\sigma_{h^p}=\sigma_{\s_p}$ by definition,
%we have
%\todo{c'est en particulier \cref{sigmainvfla} qu'on utilise}
%\begin{equation}
%(\phi_{\xi/2p}^*\,\cB_{h^p}\,f)(x)
%=\sigma_{\s_p}\left(T_{h^p}(f)\right)(\phi_{\xi/2p}(x))
%=\sigma_{e^{L_\xi/p}\s_p}\left(T_{h^p}(f)\right)\,,
%\end{equation}
%so that
%$\phi_{\xi/2p}^*\,\cB_{h^p}=\sigma_{e^{L_\xi}\s_p}\circ\sigma_{e^{L_\xi}\s_p}^*$,
%where the dual is taken in the sense of \cref{dualpropT}.
%\todo{reformuler: dual c'est moche}
The fact that $\phi_{\xi/2p}^*\,\cB_{h^p}$ is a self-adjoint and
positive operator on $\cinf(X,\R)^K$ with
respect to the scalar product $L^2(h,\xi,p)$
is a straightforward consequence of \cref{dualpropT}.
Furthermore, this operator factorizes through the finite dimensional space
$\cL(\HH_p)^K$, so that in particular, it is a compact operator with
smooth kernel.
This implies that $\Spec(\phi_{\xi/2p}^*\,\cB_{h^p})$ is discrete
and contains a finite number of non-vanishing eigenvalues counted with
multiplicity.

Let now $f\in\cinf(X,\R)^K$ be an eigenfunction of
$\phi_{\xi/2p}^*\,\cB_{h^p}$
with eigenvalue $\lambda\neq 0$. Then from \cref{Bertransdef} and
\cref{quantchanrel}, we have
\begin{equation}
\EE^{\xi}_{h^p}(T_{h^p}(f))=T_{h^p}\left(\phi_{\xi/2p}^*\,\cB_{h^p}(f)\right)=
\lambda\,T_{h^p}(f)\,,
\end{equation}
so that $T_{h^p}(f)\in\cL(\HH_p)^K$ is a non-vanishing
eigenvector of $\EE^{\xi}_{h^p}$
associated with the eigenvalue $\lambda$,
since by definition $\sigma_{e^{L_\xi}\s_p}\left(T_{h^p}(f)\right)=\lambda f\neq 0$.
This shows that the positive spectrums of $\phi_{\xi/2p}^*\,\cB_{h^p}$
and $\EE^{\xi}_{h^p}$ coincide. This concludes the proof.
\end{proof}

The following result is the analogue of \cref{KS}
for the equivariant Berezin transform
\cref{BertransT}, where the role of the Riemannian Laplacian
is played by the operator $\Delta_h^{(\xi)}$
of \cref{TZopdef}.

\begin{prop}\label{KSxi}
For any $\xi\in\sqrt{-1}\Lie T$ and $m\in\N$,
there exists $C_m>0$ such that for any
$f\in\cinf(X,\C)^K$ and all $p\in\N^*$, we have
\begin{equation}\label{KSxiexp}
\left|\phi_{\xi/2p}^*\,\cB_{h^p}f-f+p^{-1}\Delta_h^{(\xi)} f\right|_{\CC^m}\leq
\frac{C_m}{p^2}|f|_{\CC^{m+6}}\;.
\end{equation}
Furthermore, there exists $l\in\N$ such that
the constant $C_{m}>0$ can be chosen uniformly
for $\xi\in\sqrt{-1}\Lie T$ in a compact set
and $h\in\Met^+(L)$ in a bounded subset
in $\CC^l$-norm.
\end{prop}
\begin{proof}
Using \cref{KS}, we get for any $m\in\N$ a constant
$C_m>0$ such that for any
$f\in\cinf(X,\C)$ and all $p\in\N^*$, we have
\begin{equation}\label{KSexp1}
\left|\cB_{h^p}f-f+\frac{\Delta_h}{4\pi p} f\right|_{\CC^m}\leq
\frac{C_m}{p^2}|f|_{\CC^{m+4}}\;.
\end{equation}
On the other hand, considering the Taylor expansion of
$\phi_{\xi/2p}$ in $p^{-1}$, we get for any $m\in\N$ a constant
$C_m>0$ such that for any $f\in\cinf(X,\R)$ and all $p\in\N^*$,
\begin{equation}\label{KSexp2}
\left|\phi_{\xi/2p}^*\,f-f-\frac{df.\xi}{2p}\right|_{\CC^m}\leq
\frac{C_m}{p^2}|f|_{\CC^{m+2}}\;.
\end{equation}
By definition \cref{TZopdeffla} of the operator $\Delta_h^{(\xi)}$ and the
fact that $df.\xi=2df.\xi^{1,0}$ for $f\in\cinf(X,\C)^K$
and $\xi\in\sqrt{-1}\Lie T$, this gives the result.
\end{proof}

Let $\epsilon_0>0$ be smaller than the injectivity radius of $(X,g_h^{TX})$,
fix $x_0 \in X$, and let
$Z=(Z_1,...,Z_{2d})\in\R^{2d}$ with $|Z|<\epsilon_0$ be geodesic
normal coordinates around $x_0$, where $|\cdot|$ is the Euclidean
norm of $\R^{2d}$.
%In these coordinates, the canonical volume form is given by
%$dv_X(Z)=\kappa_{x_0}(Z)dZ$, where $\kappa_{x_0}\in\cinf(\R^n,\R)$
%satisfies $\kappa_{x_0}(0)=1$.
For any $K(\cdot,\cdot)\in\cinf(X\times X,\C)$, we write
$K_{x_0}(\cdot,\cdot)$ for its image in these coordinates, and
we write $|K_x|_{\CC^m(X)}$ for the local
$\CC^m$-norm with respect to $x \in X$.

Let $d^X$ be the Riemannian distance on $(X,g_h^{TX})$,
and recall as in the proof of \cref{specidprop}
that $\cB_{h^p}$
admits a smooth Schwartz kernel, for all $p\in\N^*$.
The main tool of this Section is the
following asymptotic expansion as $p\to+\infty$ of this
Schwartz kernel, which
follows from \cite[Th.\,4.18']{DLM06}.
Let $|\cdot|_{\CC^m}$ denote the local $\CC^m$ norm on local
sections of $L^p\boxtimes \left(L^p\right)^*$.
In the following statement, the estimate $O(p^{-\infty})$ means
$O(p^{-k})$ in the usual sense as $p\to+\infty$, for all $k\in\N$.

\begin{theorem}{\emph{\cite[Th.\,3.7]{IKPS19}}}
\label{BTasy}
For any $m\,,k\in\N$, $\epsilon>0$, there is
$C>0$ such that for all $p\in\N^*$ and
$x,y\in X$ satisfying $d^X(x,y)>\epsilon$, we have
\begin{equation}\label{thetafla}
|\cB_{h^p}(x,y)|_{\CC^m}\leq Cp^{-k}\;.
\end{equation}
For any $m, k\in\N$, there is
$N\in\N$, $C>0$ such that for any
$x_0\in X,\,|Z|,|Z'|<\epsilon_0$
and for all $p\in\N^*$, we have
%\todo{$n$ replaced by $d$, here and below}
\begin{multline}\label{BTexp}
\Big|p^{-d}\cB_{h^p,x_0}(Z,Z')
-\sum_{r=0}^{k-1} p^{-r/2}J_{r,x_0}(\sqrt{p}Z,\sqrt{p}Z')
\exp(-\pi p|Z-Z'|^2)\Big|_{\CC^m(X)}\\
\leq Cp^{-\frac{k}{2}}(1+\sqrt{p}|Z|+\sqrt{p}|Z'|)^N
\exp(-\sqrt{p}|Z-Z'|/C)+O(p^{-\infty})\;,
\end{multline}
where $J_{r,x_0}(Z,Z')$ are a family of polynomials
in $Z,Z'\in\R^{2n}$ of the same parity as $r\in\N$,
depending smoothly on $x_0\in X$. Furthermore,
for any $Z,Z'\in\R^{2n}$ we have
\begin{equation}\label{|J|0}
J_{0,x_0}(Z,Z')=1\quad\text{and}\quad J_{1,x_0}(Z,Z')=0\;.
\end{equation}
Finally, for any $m\in\N$, there exists $l\in\N$ such that
the estimate \cref{BTexp} is uniform for $h\in\Met^+(L)^K$ in a
bounded subset in $\CC^l$-norm.
\end{theorem}

\subsection{Spectral asymptotics}
\label{Specsec}

Fix $h\in\Met^+(L)^K$, and
write $\<\cdot,\cdot\>_{L^2}$ for the associated
$L^2$-Hermitian product on $\cinf(X,\C)^K$, defined
by formula \cref{L2theta} for $\xi=0$. We write
$\|\cdot\|_{L_2}$ for the associated norm, and
$L^2(X,\C)^K$ for the induced completion of $\cinf(X,\C)^K$.
In the notations of \cref{KRsec},
the Riemannian Laplacian $\Delta_h$
is then an elliptic self-adjoint operator
acting on $L^2(X,\C)^K$, and we write
\begin{equation}\label{evLB}
0=\lambda_0 < \lambda_1 \leq\lambda_2\leq \dots\leq\lambda_k\leq \dots\;,
\end{equation}
for the increasing sequence of its eigenvalues.
For all $j\in\N$, let $e_j\in\cinf(X,\C)^K$ be the normalized
eigenfunction associated with $\lambda_j$,
so that $\|e_j\|_{L_2}=1$ and
$\Delta e_j=\lambda_j e_j$.
%Then for any $f\in\cinf(X,\C)$, we have
%the following equality in $L_2$,
%\begin{equation}
%f=\sum_{j=0}^{+\infty}\<f,e_j\>_{L_2}e_j.
%\end{equation}
For any $F:\R\fl\R$ bounded, we define the
bounded operator $F(\Delta)$
acting on $L_2(X,\C)^K$ by the formula
\begin{equation}\label{calculfct}
F(\Delta)f=\sum_{i=0}^{+\infty}
F(\lambda_j)\<f,e_j\>_{L_2}e_j\;.
\end{equation}
In particular we can consider its associated \emph{heat operator}
$e^{-t \Delta_h}$ acting on $\cinf(X,\C)^K$, for all $t\geq 0$.
For any $m\in 2\N$,
write $\|\cdot\|_{H^m}$ for the norm defined for all
$f\in\cinf(X,\C)^K$ by
\begin{equation}\label{ellest}
\|f\|_{H^m} := \|\Delta^{m/2}_hf\|_{L_2} + \|f\|_{L_2}\;.
\end{equation}
By the classical Sobolev
embedding theorem and the elliptic estimates for $\Delta_h$,
there exists $k\in\N$ such that for any $m\in 2\N$,
there is $C_m>0$ such that
\begin{equation}\label{Sobth}
|f|_{\CC^m}\leq C_m\|f\|_{H^{m+k}}\,,
\end{equation}
for all $f\in\cinf(X,\C)^K$. Furthermore, there exists $l\in\N$
such that the constant
$C_m>0$ can be chosen uniformly for $h\in\Met^+(L)^K$ in
a bounded subset in $\CC^l$-norm. By convention, we set
$\|f\|_{H^0}:=\|f\|_{L_2}$.

We then have the following result, which is the analogue of
\cite[Prop.\,3.9]{IKPS19} for the equivariant Berezin transform
\cref{BertransT}.

\begin{prop}\label{boundexp}
For any $m\in 2\N$ and any $\xi\in\sqrt{-1}\Lie T$,
there exists $C_m>0$ such that for any
$f\in\cinf(X,\C)^K$ and all $p\in\N^*$, we have
\begin{equation}\label{boundexpfla}
\left\|\left(e^{-\frac{\Delta_h}{4\pi p}}
-e^{\theta_h(\xi)/2}\phi_{\xi/2p}^*\,\cB_{h^p} e^{-\theta_h(\xi)/2}\right)f\right\|_{H^m}\leq
\frac{C_m}{p}\|f\|_{H^m}\;.
\end{equation}
Furthermore, there exists $l\in\N$ such that the constant $C_m>0$
can be chosen uniformly for $\xi\in\sqrt{-1}\Lie T$ in a compact set
and $h\in\Met^+(L)^K$ in a bounded subset
in $\CC^l$-norm.
\end{prop}
\begin{proof}
In this proof, we use the notation $O(|W|^k)\in\R^{2n}$ or $\R$
in the usual sense as $W\in\R^{2n}$ goes to $0$,
for any $k\in\N$, and write $\<\cdot,\cdot\>$ for the Euclidean
product of $\R^{2n}$.

First note that the operator
\begin{equation}\label{tilBp}
\til{\cB}_p:=
e^{\theta_h(\xi)/2}\phi_{\xi/2p}^*\,\cB_{h^p}e^{-\theta_h(\xi)/2}
\end{equation}
has a smooth kernel given for all $x,\,y\in X$ by
\begin{equation}\label{Btilkern}
\til{\cB}_p(x,y)=e^{\theta_h(\xi)/2}(x)e^{-\theta_h(\xi)/2}(y)\,
\cB_{h^p}(\phi_{\xi/2p}(x),y)\,.
\end{equation}
Recall formula \cref{BTexp} for the asymptotic
expansion as $p\to+\infty$ of the Berezin transform $\cB_{h^p}$ in
geodesic coordinates $Z,\,Z'\in\R^{2n}$ with $|Z|,\,|Z'|<\epsilon_0$
around $x_0\in X$. For all $\xi\in\sqrt{-1}\Lie T$
in a compact set, considering the Taylor expansion of
$\phi_{\xi/2p}(Z)$ as
$p^{-1}\to 0$ and $|Z|\to 0$, we get
\begin{equation}\label{expphiZ}
\begin{split}
&\exp(-\pi p\,|\phi_{\xi/2p}(Z)-Z'|^2)\\
&=
\exp(-\pi p|Z-Z'|^2-\pi\<\xi_{x_0},Z-Z'\>+\<O(|Z|)+p^{-1}O(|Z|),Z-Z'\>)\\
&=\exp(-\pi p|Z-Z'|^2)\Big(1-\pi p^{-1/2}\<\xi_{x_0},\sqrt{p}(Z-Z')\>\\
&
\quad\quad\quad\quad\quad\quad\quad\quad\quad\quad+p^{-1}\<O(|\sqrt{p}Z|)+p^{-1}O(|\sqrt{p}Z|),\sqrt{p}(Z-Z')\>\Big)\,.
\end{split}
\end{equation}
On the other hand, recall from \cref{thetaxiprop}
that the imaginary part of the holomorphy potential equation
\cref{holpot} gives $2\pi\iota_{J\xi}\om_h=-d\theta_h(\xi)$.
Using the definition \cref{gTXintro} for $g_{h,x_0}^{TX}=:\<\cdot,\cdot\>$,
we get the following Taylor expansions as $Z\to 0$,
\begin{equation}\label{ethetaZZ'}
\begin{split}
&e^{\theta_h(\xi)/2}(Z)e^{-\theta_h(\xi)/2}(Z')\\
&=1+d\theta_h(\xi).(Z-Z')/2
+(O(|Z|)+O(|Z'|))^2\\
&=1+\pi\<\xi,Z-Z'\>+(O(|Z|)+O(|Z'|))^2\\
&=1+p^{-1/2}\pi\<\xi,\sqrt{p}(Z-Z')\>+p^{-1}
(O(|\sqrt{p}Z|)+O(|\sqrt{p}Z'|))^2\,.
\end{split}
\end{equation}
Multiplying the estimates \cref{expphiZ} and
\cref{ethetaZZ'} gives
\begin{multline}
e^{\theta_h(\xi)/2}(Z)e^{-\theta_h(\xi)/2}(Z')
\exp(-\pi p|\phi_{\xi/2p}(Z)-Z'|^2)\\
=
\exp(-\pi p|Z-Z'|^2)(1+p^{-1}O(|\sqrt{p}Z|)+p^{-1}O(|\sqrt{p}Z'|))\,,
\end{multline}
so that the coefficient of order $p^{-1/2}$ vanishes.
Plugging the expansion \cref{BTexp} in formula \cref{Btilkern}
for the Schwartz kernel of $\til{\cB}_p$, we see that
it also satisfies \cref{BTasy}, with first coefficients
satisfying \cref{|J|0}.

Setting now $R_p:=e^{\frac{\Delta}{4\pi p}}-\til{\cB}_p$,
and using the classical small-time asymptotic expansion
of the heat kernel,
as given for example in \cite[Th.\,2.29]{BGV04},
and by \cref{BTexp}, we see that
its Schwartz kernel $R_p(\cdot,\cdot)$ with respect to $dv_X$
satisfies $R_p(x,y)=O(p^{-\infty})$
for all $x,y\in X$ satisfying $d^X(x,y)>\epsilon_0$,
and we get for any $m\in\N$ a constant
$C>0$ and $N\in\N$ such that
\begin{multline}\label{B-eexp}
\left|R_{p,x_0}(Z,Z')\right|_{\CC^m(X)}\\
\leq Cp^{-1}(1+\sqrt{p}|Z|+\sqrt{p}|Z'|)^N
\exp(-\sqrt{p}|Z-Z'|/C)+O(p^{-\infty})\;.
\end{multline}
Following the proof of \cite[Prop.\,3.9]{IKPS19}, this readily
implies that for any $m\in 2\N$, there is a constant
$C_m>0$ such that for all $f\in\cinf(X,\R)$,
\begin{equation}
\left\|R_p(f)\right\|_{H^m}\leq
\frac{C_m}{p}\|f\|_{H^m}\;.
\end{equation}
The uniformity of $C_m>0$ with respect to $h\in\Met^+(L)^K$
comes from the uniformity of
the small-time asymptotic expansion of the heat kernel with respect to
the Riemannian metric together with the uniformity in \cref{BTasy}.
This gives the result.
\end{proof}

Let us now write $\|\cdot\|_{L^2(h,\xi)}$ and
$\|\cdot\|_{L^2(h,\xi,p)}$ for the norms associated with
the $L^2$-Heermitian products \cref{L2theta} and \cref{L2hpxi}
respectively.
Using \cref{Bergdiagexp}, \cref{KS} and \cref{Tuycor},
we get a constant $C>0$, uniform for $\xi\in\sqrt{-1}\Lie T$
in a compact set
and $h\in\Met^+(L)^K$ in a bounded subset
in $\CC^l$-norm for some $l\in\N$, such that
\begin{equation}\label{WpL2eq}
\left(1-\frac{C}{p}\right)\|\cdot\|_{L^2(h,\xi)}\leq
\|\cdot\|_{L^2(h,\xi,p)}
\leq \left(1+\frac{C}{p}\right)\|\cdot\|_{L^2(h,\xi)}\,.
\end{equation}
\cref{boundexp} implies the following key lemma, which is an analogue
of \cite[Lem.\,3.10]{IKPS19} for the equivariant Berezin transform
\cref{BertransT}.
%The result is then a consequence of \cref{boundexp} in the same
%way as \cite[Prop.\,3.9]{IKPS19} is a consequence of
%\cite[Prop.\,3.10]{IKPS19},

\begin{lem}\label{hjrefprop}
For any fixed $L>0$,
consider sequences $\{f_p\in\cinf(X,\C)^K\}_{p\in\N^*}$ and
$\{\mu_p\in\Spec(\phi_{\xi/2p}^*\,\cB_{h^p})\}_{p\in\N^*}$,
such that $p|1-\mu_p|<L$ for all $p\in\N^*$ and
\begin{equation}\label{efBp}
\|f_p\|_{L^2(h,\xi,p)}=1\quad\text{and}\quad
\phi_{\xi/2p}^*\,\cB_{h^p}(f_p)=\mu_p f_p\;,
\end{equation}
Then
for all $m\in 2\N$, there exists $C_{L,m}>0$
such that for all $p\in\N^*$, we have
\begin{equation}\label{hjrefined}
\|f_p\|_{H^{m}}\leq C_{L,m}\;,
\end{equation}
not depending on $\xi\in\sqrt{-1}\Lie T$
in a compact set and $h\in\Met^+(L)^K$ in a bounded subset in $\CC^l$-norm
for some $l\in\N$.
\end{lem}
\begin{proof}

%Then the sequence
%$\{e^{\theta(\xi)/2}f_p\}_{p\in\N^*}$ also
%satisfies formula \cref{efBp} with $\phi_{\xi/2p}^*\,\cB_{h^p}$
%replaced by
% and the norms $\|\cdot\|_{L^2(h,\xi,p)}$ replaced by
%,
%for all $p\in\N$. Using \cref{boundexp},
%we can then apply the proof of 
%
% up to a trivial conjugation.
%The uniformity comes from the uniformity in \cref{boundexp} and the
%manifest uniformity in
%the proof of \cite[Lemma\,3.10]{IKPS19}.
%$h\in\Met^+(L)^K$ in a subset bounded in $\CC^l$-norm
%for some $l\in\N$, to get the estimate \cref{hjrefined} for
%$e^{\theta(\xi)/2}f_p$, hence for $f_p$, for all $p\in\N$.
%This gives the result.
Let
$\{f_p\in\cinf(X,\C)^K\}_{p\in\N^*}$ be a sequence satisfying
\cref{efBp} as above, and for all $p\in\N^*$, set
\begin{equation}\label{tildef}
\|\cdot\|_p:=\|e^{-\theta_h(\xi)/2}\cdot\|_{L^2(h,\xi,p)}\quad
\text{and}\quad\til{f}_p:=e^{\theta_h(\xi)/2}f_p\,.
\end{equation}
In particular, we have $\|\til{f}_p\|_{p}=1$ and
$\til\cB_{p}(\til{f_p})=\mu_p \til{f}_p$ for all $p\in\N^*$,
for the operator $\til{\cB_p}$ defined by formula \cref{tilBp},
and the estimate \cref{WpL2eq} implies the estimate \cref{hjrefined} for $m=0$.

By induction on $m\in 2\N$,
assume now that \cref{hjrefined}
is satisfied for $m-2$.
Write
\begin{equation}\label{deltFest}
\begin{split}
p(e^{-\frac{\Delta_h}{4\pi p}}-\til\cB_p)\til{f}_p&=p(1-\mu_p)
\til{f}_p-
p(\id-e^{-\frac{\Delta_h}{4\pi p}})\til{f}_p\\
&=p(1-\mu_p)\til{f}_p-\Delta_h F(\Delta_h/p)\til{f}_p\;,
\end{split}
\end{equation}
where the bounded operator $F(\Delta_h/p)$ acting on $L_2(X,\C)^K$
is defined as in \cref{calculfct} for the continuous function
$F:\R\fl\R$ given for any $s\in\R^*$ by
$F(s)=4\pi(1-e^{-s/4\pi})/s$.
As $|p(1-\mu_p)|<L$ for all $p\in\N^*$, by \cref{boundexp} and
formula \cref{ellest} for $\|\cdot\|_{H^{m}}$,
this gives a constant $C_m>0$, uniform
$h\in\Met^+(L)^K$ in a bounded subset in $\CC^l$-norm
for some $l\in\N$, such that
\begin{equation}\label{sobfest}
\|F(\Delta_h/p)\til{f}_p\|_{H^{m}}\leq C_m\|\til{f}_p\|_{H^{m-2}}\;.
\end{equation}
On the other hand, note that by hypothesis, we have $\mu_p\fl 1$
as $p\fl+\infty$. Using \cref{boundexp} again, we then get
$\epsilon_m>0$ and $p_m\in\N^*$, uniform
$h\in\Met^+(L)^K$ in a bounded subset in $\CC^l$-norm,
such that for all $p>p_m$,
\begin{equation}\label{hjdeltFest}
\begin{split}
\|F(\Delta_h/p)\til{f}_p\|_{H^{2m}}&\geq
\|F(\Delta_h/p)\til{f}_p+(\til\cB_p-e^{-\frac{\Delta_h}{4\pi p}})\til{f}_p\|_{H^{m}}
-\|(\til\cB_p-e^{-\frac{\Delta_h}{4\pi p}})\til{f}_p\|_{H^{m}}\\
&\geq\inf_{s>0}\,\{F(s)+\mu_p-e^{-s/4\pi}\}\,\|\til{f}_p\|_{H^{m}}
-C_m p^{-1}\|\til{f}_p\|_{H^{m}}\\
&\geq\epsilon_m\|\til{f}_p\|_{H^{m}}\;.
\end{split}
\end{equation}
Hence by \cref{sobfest}, we get a constant $C_{L,m}>0$,
uniform in $h\in\Met^+(L)^K$ in a bounded subset in $\CC^l$-norm,
such that for all $p\in\N^*$, we have
$\|\til{f}_p\|_{H^{m}}\leq C_{L,m}$, which gives
\cref{hjrefined} by \cref{tildef}.
\end{proof}

Using \cref{TZopdef}, write
\begin{equation}\label{lambdajhxi}
0=\lambda_0(h,\xi)<\lambda_1(h,\xi)\leq\dots\leq\lambda_k(h,\xi)
\leq \dots
\end{equation}
for the increasing sequence of eigenvalues of $\Delta_h^{(\xi)}$,
and using \cref{specidprop}, write
\begin{equation}\label{gammajhpxip}
\gamma_0(h^p,\xi)\geq\gamma_1(h^p,\xi)\geq\cdots\geq\gamma_1(h^p,\xi)
\geq\cdots\geq 0
\end{equation}
for the decreasing sequence of eigenvalues
of $\phi_{\xi/2p}^*\,\cB_{h^p}$.
The following result is the analogue of \cite[Th.\,3.1]{IKPS19}
for the equivariant Berezin transform \cref{BertransT}, and
is the analytic basis of our proof of \cref{mainth}.

\begin{theorem}\label{Bpgap}
For every integer $k\in\N$, there exists a constant $C_k>0$
such that for any $p\in\N^*$,
\begin{equation}\label{eq-LB}
\big|1-\gamma_{k}(h^p,\xi)-p^{-1}\lambda_k(h,\xi)\big|\leq C_kp^{-2}\;.
\end{equation}
Moreover, there exists $l\in\N$ such that
the constant $C_k>0$ can be chosen uniformly for
$\xi\in\sqrt{-1}\Lie T$ in a compact set and
$h\in\Met^+(L)^K$ in a bounded subset in $\CC^l$-norm.
\end{theorem}
\begin{proof}
%This result follows from \cref{KS} and \cref{hjrefprop} in the same way
%as the proof of \cite[Th.\,3.1]{IKPS19} in \cite[\S\,3.5]{IKPS19},
%as it is manifestly uniform in
%$h\in\Met^+(L)$ in a bounded subset in $\CC^l$-norm
%for some $l\in\N$.
By \cref{KSxi} and by the Sobolev estimate \cref{Sobth},
there is $m\in 2\N,\, l\in\N$ and a constant $C>0$
%uniform for $\xi\in\sqrt{-1}\Lie T$ in a compact set
%and $h\in\Met^+(L)$ in a bounded subset
%in $\CC^l$-norm,
such that for any $f\in\cinf(X,\C)^K$,
\begin{equation}\label{KSSob}
\left\|p(1-\phi^*_{\xi/2p}\cB_{h^p})f-\Delta_h^{(\xi)}
f\right\|_{L_2(h,\xi)}
\leq Cp^{-1}\|f\|_{H^{m}}\;,
\end{equation}
and the estimate \cref{WpL2eq} shows that
equation \cref{KSSob} also holds
in the norm $\|\cdot\|_{L_2(h,\xi,p)}$.
Let now $j\in\N$ be fixed, and let $e_j\in\cinf(X,\C)$ satisfy
$\Delta_h^{(\xi)} e_j=\lambda_j(h,\xi) e_j$ and $\|e_j\|_{L_2}=1$.
We then get
$C>0$ such that for all $p\in\N^*$,
\begin{equation}\label{estevdelt}
\left\|p(1-\phi^*_{\xi/2p}\cB_p)e_j-\lambda_j(h,\xi)e_j\right\|
_{L_2(h,\xi,p)}\leq C p^{-1}\;.
\end{equation}
Let $m_j\in\N$ be the multiplicity of $\lambda_j(h,\xi)$ as an
eigenvalue of $\Delta_h^{(\xi)}$.
Then the the estimate \cref{estevdelt} for all
eigenfunctions of $\Delta_h^{(\xi)}$ associated with
$\lambda_j(h,\xi)$
gives a constant $C_j>0$ such that for all $p\in\N^*$,
\begin{equation}\label{spec>}
\#\left(\Spec\big(p(1-\phi^*_{\xi/2p}\cB_p)\big)\cap\left[\lambda_j(h,\xi)-C_jp^{-1},\lambda_j(h,\xi)+C_jp^{-1}\right]\right)
\geq m_j\;.
\end{equation}
Conversely, fix $L>0$ and let $\{f_p\}_{p\in\N^*}$ be
the sequence of normalized eigenfunctions considered
in \cref{hjrefprop}. Then by
\cref{KSSob}, we get $C>0$ such that
\begin{equation}\label{specinv>}
\left\|p(1-\mu_p)f_p-\Delta_h^{(\xi)}
f_p\right\|_{L_2(h,\xi)}\leq Cp^{-1}\;.
\end{equation}
In particular, we get that
\begin{equation}\label{eq-spec-est}
\textup{dist}\left(p(1-\mu_p),\Spec\,\Delta_h^{(\xi)}
\right)\leq Cp^{-1}\;,
\end{equation}
showing that all eigenvalues of
$p(1-\phi^*_{\xi/2p}\cB_p)$ bounded by some
$L>0$ have to be included in the left hand side of \cref{spec>}.

Let us finally show that \cref{spec>} is an equality
for $p\in\N^*$ big enough. Let $l\in\N$ with $l\geq m_j$
be such that for all $p\in\N^*$, there exists
an orthonormal family $\{f_{k,p}\}_{1\leq k\leq l}$
of eigenfunctions of $\phi^*_{\xi/2p}\cB_p$
for $\|\cdot\|_{l^2(h,\xi,p)}$ with associated
eigenvalues $\{\mu_{k,p}\in\R\}_{1\leq k\leq l}$
satisfying
\begin{equation}
p(1-\mu_{k,p})\in[\lambda_j(h,\xi)-Cp^{-1},\lambda_j(h,\xi)+Cp^{-1}]\;,~~
\text{for all}~~1\leq k\leq l\;.
\end{equation}
By \cref{hjrefprop} and \cref{WpL2eq},
the compact inclusion of the Sobolev space $H^4$ in $H^{2}$
gives
a subsequence of $\{f_{k,p}\}_{p\in\N^*}$ converging
to a function $f_k$ in $H^{2}$-norm,
for all $1\leq k\leq l$.
In particular, using \cref{WpL2eq} again,
the family $\{f_k\}_{1\leq k\leq l}$
is orthonormal in $L_2(h,\xi)$
and satisfies $\Delta_h^{(\xi)} f_k=\lambda_j(h,\xi) f_k$ for all $1\leq k\leq l$ by
\cref{specinv>}. By definition of the multiplicity $m_j\in\N$
of $\lambda_j(h,\xi)$, this forces $l=m_j$. We thus get
\begin{equation}\label{spec=}
\#\left(\Spec\big(p(1-\phi^*_{\xi/2p}\cB_p)\big)\cap
\left[\lambda_j(h,\xi)-Cp^{-1},\lambda_j(h,\xi)+Cp^{-1}\right]\right)
=m_j\;.
\end{equation}
By the uniformity of the constants in
\cref{boundexp}, \cref{hjrefprop}, \cref{Sobth}, \cref{WpL2eq},
and by the smooth dependance 
of the eigenfunctions of $\phi_{\xi/2p}^*\,\cB_{h^p}$ and
$\Delta_h^{(\xi)}$ with
respect to the initial data, we get the result.
\end{proof}

\section{Proof of the main Theorem}
\label{proofsec}

In this Section, we use the preliminary results of all previous Sections
to establish \cref{mainth}. In \cref{approxbalsec}, we will
show how to use the results of \cref{KRsec} on the differential operator
of Tian and Zhu and the results of \cref{quantFutsec} on the quantized
Futaki invariants \cref{quantFutdefintro} to construct approximately
balanced metrics from a given Kähler-Ricci soliton. The heart of the proof
is in \cref{existencesec}, where we use these
approximately balanced metrics to
establish existence and convergence in \cref{mainth},
applying the tools of \cref{BTsecT} on the moment map
picture of \cref{relmomsec}.
Finally, we establish uniqueness in \cref{uniquesec} using
\cref{Futpropintro} and an energy
functional for the relative moment map .

Throughout the whole section, we will assume given a
positive Hermitian metric $h_\infty\in\Met^+(L)$ such that
$\om_{h_\infty}\in\Om^2(X,\R)$ is a Kähler-Ricci soliton
with respect to $\xi_\infty\in\Lie\Aut(X)$ in the sense of \cref{KRdef}.
We write $K\subset\Aut_0(X)$ for the identity component of
the subgroup of isometries
of $(X,g_{h_\infty}^{TX})$, and $T\subset K$ for the identity component
of its center, so that $h_\infty\in\Met^+(L)^T$ and
$\xi_\infty\in\sqrt{-1}\Lie T$.

\subsection{Approximately balanced metrics}
\label{approxbalsec}

The following semi-classical
estimate on the Berezin symbol is inspired
by \cite[Lem.\,24,\,(35)-(35')]{Don01}. We provide a proof
that does not make use of any moment map construction.
Write $\|\cdot\|_{tr}$ for the trace norm
on $\End(H^0(X,L^p))$.

\begin{lem}\label{sigmaestlem}
For any any $m\in\N$ and $h\in\Met^+(L)^T$,
there exists a constant $C_m>0$ such that
for any $A\in\cL(\HH_p)^T$ and all $p\in\N^*$, we have
\begin{equation}
|\sigma_{h^p}(A)|_{\CC^m}\leq C_m\,p^{n+\frac{m}{2}}\,\|A\|_{tr}\,.
\end{equation}
Furthermore, there exists $l\in\N$ such that
the constant $C_m>0$ can be chosen uniformly for
$h\in\Met^+(L)^T$ in a bounded subset in $\CC^l$-norm.
\end{lem}
\begin{proof}
Using the Sobolev embedding
theorem as in \cite[Lem.\,2]{MM15}, we get for any $m\in\N$
and $h\in\Met^+(L^p)$ a constant
$C_m>0$ such that for all $p\in\N^*$ and any holomorphic section
$s\in H^0(X,L^p)$, we have
\begin{equation}\label{Sobunif}
|s|_{\CC^m(h^p)}\leq C_m\,p^{\frac{n+m}{2}}\|s\|_{L^2(h^p)}\,,
\end{equation}
where $|\cdot|_{\CC^m(h^p)}$ denotes the $\CC^m$-norm
with respect to the Chern connection of $(L^p,h^p)$.
Replacing $h^p$ by $e^f h^p$ with $\|f\|_{\CC^m}<C$ for some
fixed $C>0$ in \cref{Sobunif}, we readily see that $C_m>0$ can be chosen
uniformly for all
$h^p\in\Met^+(L^p)^T$ in a subset bounded in $\CC^m$-norm.

Let now $\{s_j\}_{j=1}^{n_p}$ be an orthonormal basis for $L^2(h^p)$, 
and for $A\in\Herm(\C^{n_p})^T$, write
$(A_{jk})_{j,\,k=1}^{n_p}\in\Herm(\C^{n_p})$ for its matrix in this
basis. Using formulas \cref{h=rhohs} and \cref{cohstatecoord},
we know that
\begin{equation}\label{Rsig-R}
\sigma_{h^p}(A)=\rho_{h_p}^{-1}
\sum_{j,\,k=1}^{n_p}A_{jk}\<s_k,s_j\>_{h^p}\,.
\end{equation}
On the other hand, using \cref{Bergdiagexp}, we get
$C>0$ such that
\begin{equation}
|\rho_{h^p}^{-1}|_{\CC^0}=\min_{x\in X}\left(p^n\,\frac{\om^n_h}{d\nu_h\,n!}
+O(p^{n-1})\right)^{-1}\leq C p^{-n}\,.
\end{equation}
Using the Leibniz rule on successive derivatives of
$\rho_{h_p}^{-1}$ and by \cref{Bergdiagexp} again, this implies the existence of $C_m'>0$,uniform in the $\CC^l$-norm of $h^p$ for some
$l\in\N$. such that
for all $p\in\N^*$, we have
\begin{equation}\label{rho-1est}
|\rho_{h^p}^{-1}|_{\CC^m}\leq C_m' p^{-n}\,.
\end{equation}
Then using the estimates \cref{Sobunif}, \cref{rho-1est} and the fact, following from \cref{Bergdiagexp} and formula \cref{intPirho=Id}, that
$n_p\leq C p^{n}$ as $p\to+\infty$ for some $C>0$,
we can use
Cauchy-Schwartz inequality on the trace norm
%recalling the expansion \cref{npexp} for the dimension
to get for all $m\in\N$
a uniform constant $C'>0$,
such that for all $A\in\Herm(\C^{n_p})^T$ and all $p\in\N^*$,
%up to a change of orthonormal basis
%of $H_{\s_k(p)}\in\Prod(H^0(X,L^p))$, we can always assume
%that $A\in\Herm(\C^{n_p})$ is diagonal.
%On the other hand,
\begin{equation}\label{sigmaG}
\begin{split}
&|\sigma_{h^p}(A)|_{\CC^m}
\leq\left|\rho_{h^p}^{-1}\right|_{\CC^m}
\sum_{j,\,k=1}^{n_p}\left|A_{jk}\<s_k,s_j\>_{h^p}
\right|_{\CC^m}\\
&\leq C_m' p^{-n}\|A\|_{tr}
\sqrt{\sum_{j,\,k=1}^{n_p}\sum_{r,l=1}^m\binom{m}{r}\binom{m}{l}
|s_k|_{\CC^{r}(h^p)}|s_j|_{\CC^{m-r}(h^p)}|s_k|_{\CC^{l}(h^p)}|s_j|_{\CC^{m-l}(h^p)}}\\
&\leq C'p^{-n} p^{n+\frac{m}{2}}\,
\|A\|_{tr}\,n_p\leq C'C p^{n+\frac{m}{2}}\|A\|_{tr}\,.
\end{split}
\end{equation}
This gives the result.
\end{proof}

The following result
is an extension of the analogous result of Donaldson
\cite[Th.\,26]{Don01} to the case of general $\Aut(X)$,
and is an anticanonical analogue of the result of Sano and
Tipler in \cite[Th.\,5.5]{ST17}.
The proof closely follows their strategy.

%All the local $\CC^m$-norms are taken with respect to the fixed Kähler-Ricci
%soliton.

\begin{prop}\label{approxbal}
There exist $K$-invariant
functions $f_r\in\cinf(X,\R)^T$ for all $r\in\N$, such that
for every $k,\,m\in\N$, there exists a constant $C_{k,m}>0$
such that all $p\in\N^*$ big enough,
the $T$-invariant positive Hermitian metric
\begin{equation}\label{approxbalmet}
h_k(p):=\exp\left(\sum_{r=1}^{k-1}\frac{1}{p^r}f_r\right)
h_\infty\in\Met^+(L^p)^K\,,
\end{equation}
have associated Rawnsley function $\rho_{h_k^p(p)}\in\cinf(X,\R)$
satisfying
\begin{equation}\label{approxbalfla}
\left|\,\rho_{h_k^p(p)}\,\sigma_{h_k^p(p)}(e^{L_{\xi_p}/p})-
\frac{\Tr[e^{L_{\xi_p}/p}]}{\Vol\left(d\nu_{h_k(p)}\right)}\,
\right|_{\CC^m}\leq C_{k,m} p^{n-k}\,,
\end{equation}
for the sequence $\{\xi_p\in\sqrt{-1}\Lie T\}_{p\in\N^*}$ of
\cref{xipmin}.
\end{prop}
\begin{proof}
In this proof, the notation $O(p^{-k})$ for some $k\in\N$
is taken to be in its usual sense, uniformly in $\CC^m$-norm for all
$m\in\N$ and uniform the $\CC^l$-norm of $h\in\Met^+(L)$ for some $l\in\N$.

Fix $h^p\in\Met^+(L^p)^K$, and consider the setting of \cref{BTsecT}.
%Recall from \cref{findimsec}
%that for any $\xi\in\sqrt{-1}\Lie T$, we have $e^{L_\xi}\in\cL(\HH_p)^T$.
Using \cref{Tuy} and the definition of the exponential of an operator,
we know that for any $\xi\in\Lie\Aut(X)$, we have
$e^{L_\xi/p}=\Id+|\xi|\,O(1)$, so that
\cref{xipexp} implies that for any
$k\in\N$, we have
$e^{L_{\xi_p}/p}=e^{L_{\xi_\infty}/p}\prod_{j=1}^{k}
e^{L_{p^{-j}\xi^{(j)}}/p}+O(p^{-k-1})$.
We can then use \cref{Tuycor} and \cref{Toepexp} to get functions
$\eta^{(j)}_{\xi}\in\cinf(X,\C)$
for all $j\in\N$, depending smoothly on $\xi\in\sqrt{-1}\Lie T$,
such that for any $k\in\N$, we have
\begin{equation}\label{eLxipexp}
e^{L_{\xi_p}/p}
=T_{h^p}\left(e^{\theta_h(\xi_p)}\right)+
\sum_{j=1}^{k}\,p^{-j}\,T_{h^p}(\eta^{(j)}_{\xi_p})+O(p^{-k-1})\,.
\end{equation}
Write $R_k(p)\in\cL(\HH_p)$ for the remainder in \cref{eLxipexp},
so that $\|R_k(p)\|_{op}=O(p^{-k-1})$.
Using the fact $n_p=O(p^n)$
by \cref{Bergdiagexp} and formula \cref{intPirho=Id}, Cauchy-Schwartz
inequality then implies that
$\|R_k(p)\|_{tr}=O(p^{\frac{n}{2}-k-1})$.
Using now \cref{KS},
we get functions
$g_j(h)\in\cinf(X,\R)$ for all $j\in\N$, depending smoothly in the
successive derivatives of $h\in\Met^+(L)$, such that
for any $k_0\in\N$, \cref{sigmaestlem} applied to the expansion
\cref{eLxipexp} for $k> k_0+(3n+m)/2$ gives
\begin{equation}\label{sigeLxiexp}
\begin{split}
\sigma_{h^p}(e^{L_{\xi_p}/p})&=\cB_{h^p}(e^{\theta_h(\xi_\infty)})
+\sum_{j=1}^{k}\,p^{-j}\,\cB_{h^p}(\eta^{(j)}_{\xi_p})
+\sigma_{h^p}(R_k(p))\\
&=e^{\theta_h(\xi_\infty)}+\sum_{j=1}^{k_0}
\,p^{-j}\,g_j(h)+O(p^{-k_0-1})\,.
\end{split}
\end{equation}
Comparing with \cref{Bergdiagexp}, and using
\cref{pullbackprop}, we get $K$-invariant functions
$f_j(h)\in\cinf(X,\R)^K$ for all $j\in\N$, depending smoothly in the
successive derivatives of $h\in\Met^+(L)$, such that
for any $k\in\N$, we have
\begin{equation}\label{rhosigexp}
p^{-n}\,\sigma_{h^p}(e^{L_{\xi_p}/p})\,\rho_{h^p}
=\frac{e^{\theta_h(\xi_\infty)}\om_h^n}{n!\,d\nu_h}+
\sum_{j=1}^{k}\,p^{-j}\,f_j(h)+O(p^{-k-1})\,.
\end{equation}
Via the characterization \cref{KRvolfla}, we see that
the first coefficient
of \cref{rhosigexp} is constant if and only if
$h\in\Met^+(L)$ is a Kähler-Ricci soliton with respect to $\xi_\infty$.
Integrating both sides against $d\nu_{h}$ and using formula
\cref{rhosig=tre}, this implies the result for $k=1$.
%\begin{equation}\label{sigmaepfh}
%\begin{split}
%p^{-n}\,\sigma_{(e^{p^{-k}f}h)^p}&(e^{L_{\xi_p}/p})=e^{\theta_h(\xi_\infty)}
%e^{-p^{-k}df.\xi^{1,0}_\infty}
%+\sum_{j=1}^{k}\,p^{-j}\,g_j(e^{p^{-k}f}h)+O(p^{-(k+1)})\\
%&=e^{\theta_h(\xi_\infty)}
%+\sum_{j=1}^{k-1}\,p^{-j}\,g_j(h)+p^{-k}\left(g_{k}(h)-
%df.\xi^{1,0}_\infty\right)
%+O(p^{-(k+1)})\,.
%\end{split}
%\end{equation}

Recall that $\Delta_h$ denotes the scalar Riemannian Laplacian
of $(X,g_h^{TX})$. Using the variation formula \cref{dnuef/dnu}
for the anticanonical volume form and
a classical formula in Kähler geometry,
for any $f\in\cinf(X,\R)$ we get
\begin{equation}\label{varinftyfla}
\dt\Big|_{t=0}\frac{\om_{e^{tf}h_\infty}^n}{\,d\nu_{e^{tf}h_\infty}}
=\left(\frac{1}{4\pi}\Delta_{h_\infty} f-f\right)\frac{\om^n_{h_\infty}}
{d\nu_{h_\infty}}\,.
\end{equation}
For any $\xi\in\sqrt{-1}\Lie T$ and
$h\in\Met^+(L)^K$, recall the operator $\Delta_h^{(\xi)}$
acting on $\cinf(X,\R)^K$
defined in \cref{TZopdef}.
Using \cref{thetaxiprop} on holomorphy
potentials and \cref{Bergdiagexp},
the expansion \cref{rhosigexp} implies that for all $k\in\N$
and $f\in\cinf(X,\R)^T$, we have
\begin{multline}\label{rhosigexp2}
p^{-n}\,\sigma_{(e^{p^{-k}f}h)^p}(e^{L_{\xi_p}/p})
\,\rho_{(e^{p^{-k}f}h)^p}\\
=\frac{e^{\theta_h(\xi_\infty)}\om_h^n}{n!\,d\nu_h}+
\sum_{j=1}^{k-1}\,p^{-j}\,f_j(h)+p^{-k}\left(f_k(h)+
\Delta_{h_\infty}^{(\xi_\infty)} f-f\right)
+O(p^{-k-1})\,.
\end{multline}
On the other hand, using \cref{dualprop,Tuy} and the definition
\cref{quantFutdefintro} of the
quantized Futaki invariant,
\cref{xipmin} implies that for any
$h\in\Met^+(L)$ and all $\eta\in\Lie\Aut(X)$, we have
\begin{equation}
\int_X\,\theta_h(\eta)\,\sigma_{h^p}(e^{L_{\xi_p}/p})\,\rho_{h^p}\,d\nu_h
=\frac{\Fut^{\xi_p}_p(\eta)}{p+1}=0\,,
\end{equation}
Using \cref{TZ},
this implies that the coefficients in the expansion \cref{rhosigexp}
%we then get
%for all $j\in\N$ and any
%$h\in\Met^+(L)$,
%\begin{equation}\label{intfj=0}
%\int_X\,\theta_{h}(\eta)\,f_j(h)\,d\nu_{h}=0\,.
%\end{equation}
%This implies in particular that
for $h_\infty$ satisfy
\begin{equation}\label{fjhinfperp}
f_j(h_\infty)\in
\left(\Ker \left(\Delta_{h_\infty}^{(\xi_\infty)}-\Id\right)\right)^\perp
\quad\text{for all}\quad j\in\N\,,
\end{equation}
for the $L^2$-scalar product $L^2(h,\xi)$ on
$\cinf(X,\R)^K$, defined by formula
\cref{L2theta} using the characterization
\cref{KRvolfla} of Kähler-Ricci solitons.
Thus for all $j\in\N$, there exists a function $f_j\in\cinf(X,\R)^K$
satisfying $f_j(h_\infty)=f_j-\Delta_{h_\infty}^{(\xi_\infty)}f_j$.
Taking $h_1(p):=e^{f_1/p}h_\infty\in\Met^+(L)^K$,
the second coefficient of the expansion \cref{rhosigexp2} with $k=2$
vanishes, and integrating both sides against $d\nu_{h_1(p)}$
gives the result for $k=2$ via formula \cref{rhosig=tre} as above.

%Note that there is the freedom of a multiplicative constant $c_p=1+O(p^{-1})$
%in the choice of the metrics \cref{approxbalmet}
%satisfying \cref{approxbalfla}.
%In this proof, we are using this freedom to get all coefficients of
%\cref{rhosigexp} to vanish.
Let us now assume that for some $k\in\N$,
we have positive Hermitian metrics $h_k(p)\in\Met^+(L)^K$ as in
\cref{approxbalmet} satisfying
\begin{equation}\label{rhosigexp3}
\sigma_{h_k(p)}(e^{L_{\xi_p}/p})\rho_{h_k^p(p)}
=\frac{e^{\theta_{h_\infty}(\xi_\infty)}\om_{h_\infty}^n}
{n!\,d\nu_{h_\infty}}+O(p^{-k})\,.
\end{equation}
As we have $h_k(p)\to h_\infty$ smoothly as $p\to+\infty$
by hypothesis, we can again
apply \cref{Bergdiagexp} to get expansion \cref{rhosigexp}
for $h_k^p(p)$, and taking the Taylor expansion as $p\to+\infty$
of the coefficients $f_r(h_k(p))$ for all $1\leq r\leq k+1$, we then get
for any $f\in\cinf(X,\R)^K$,
\begin{multline}\label{rhosigexp4}
p^{-n}\,\sigma_{e^{p^{-k}f}h}(e^{L_{\xi_p}/p})\rho_{(e^{p^{-k}f}h)^p}\\
=\frac{e^{\theta_{h_\infty}(\xi_\infty)}\om_{h_\infty}^n}
{n!\,d\nu_{h_\infty}}+p^{-k}\,
\left(f_k(h_\infty)+\Delta_{h_\infty}^{(\xi_\infty)}f-f\right)
+O(p^{-k-1})\,.
\end{multline}
Taking $h_{k+1}(p):=e^{f_k/p^k}h_k(p)\in\Met^+(L^p)^K$ for all $p\in\N^*$,
where $f_k\in\cinf(X,\R)^K$ satisfies
$f_k(h_\infty)=f_k-\Delta_{h_\infty}^{(\xi_\infty)}f_k$
thanks to \cref{fjhinfperp}, we get
the result for $k+1$ via formula \cref{rhosig=tre} as above.
This gives the result
for general $k\in\N$ by induction.
\end{proof}

For any $k\in\N$ and $p\in\N^*$ big enough,
consider the positive Hermitian metrics
$h_k(p)\in\Met^+(L)^K$ constructed in \cref{approxbal} and
let $\textbf{s}_k(p)\in\BB(H^0(X,L^p))^T$
be orthonormal with respect to $L^2(h_k^p(p))$.
The following Lemma shows that these metrics indeed
approximate the Kähler-Ricci soliton.

\begin{lem}\label{approxomkplem}
For any $k,\,k_0,\,m\in\N$ with $k\geq k_0>n+1+m/2$,
there exists $C>0$
such that for all $p\in\N^*$ and any
$B\in\Herm(\C^{n_p})^T$ with $\|B\|_{tr}\leq C^{-1}p^{-k_0}$,
we have
\begin{equation}
\label{approxhBp}
\begin{split}
\left|\om_{e^{B}e^{L_{\xi_p}/2p}\s_k(p)}-\om_{h_\infty}\right|_{\CC^{m}}
&\leq\frac{C}{p}\,,\\
\left|\frac{d\nu_{e^{B}e^{L_{\xi_p}/2p}\s_k(p)}}{d\nu_{h_k(p)}}-
\frac{\Vol(d\nu_{e^{B}\s_k(p)})}{\Vol(d\nu_{h_k(p)})}
\,\right|_{\CC^0}
&\leq Cp^{-k_0-1}\,,
\end{split}
\end{equation}
and $C^{-1}<\Vol(d\nu_{e^{B}\s_k(p)})<C$.
\end{lem}
\begin{proof}
Using \cref{FSvar,FSvar2}, we know that for all
$k\in\N,\,p\in\N^*$
and $B\in\Herm(\C^{n_p})^T$, we have
\begin{equation}\label{hkp=rhosighs}
\begin{split}
h_k^p(p)&=\rho_{h_k^p(p)}\,h_{\s_k(p)}
=\rho_{h_k^p(p)}\,
\sigma_{h^p_k(p)}(e^{L_{\xi_p}/p})\,h_{e^{L_{\xi_p}/2p}\s_k(p)}^p\\
&=\rho_{h_k^p(p)}\,
\sigma_{h^p_k(p)}(e^{L_{\xi_p}/p})\,\sigma_{e^{L_{\xi_p}/p}\s_k(p)}(e^{2B})
\,h_{e^Be^{L_{\xi_p}/2p}\s_k(p)}^p\,.
\end{split}
\end{equation}
%Using this formula, the result is then a
%consequence of \cref{sigmaestlem} and \cref{approxbal}, following
%for instance the proofs of \cite[Lem.\,3.6, 3.7]{Ioo20} and the
%fact that
By definition of the Kähler form
\cref{preq}, we then get
\begin{equation}\label{ddbar}
\begin{split}
\om_{e^Be^{L_{\xi_p}/p}\s_k(p)}=\om_{h_k(p)}
&-\frac{\sqrt{-1}}{2\pi p}\dbar\partial
\log\sigma_{e^{L_{\xi_p}/p}\s_k(p)}\left(e^{2B}\right)\\
&-\frac{\sqrt{-1}}{2\pi p}\dbar\partial\log\left(\rho_{h_k^p(p)}
\sigma_{h_k^p(p)}(e^{L_{\xi_p}/p})\right)\\
=\om_{h_k(p)}
-\frac{\sqrt{-1}}{2\pi p}\dbar\partial\log&
\left(1+\sigma_{e^{L_{\xi_p}/p}\s_k(p)}\left(e^{2B}-\Id\right)\right)\\
-\frac{\sqrt{-1}}{2\pi p}\dbar\partial
\log&\left(1+\left(\frac{\Vol(d\nu_{h_k(p)})}{\Tr\left[e^{L_{\xi_p}/p}\right]}\rho_{h_k^p(p)}\sigma_{h_k^p(p)}(e^{L_{\xi_p}/p})
-1\right)\right)\,.
\end{split}
\end{equation}
Recall from \cref{Tuycor,xipexp} that
$C^{-1}p^{n}<\Tr\left[e^{L_{\xi_p}/p}\right]<Cp^{n}$ for some $C>0$,
while $\Vol(d\nu_{h_k(p)})\to\Vol(d\nu_{h_\infty})$ as $p\to+\infty$
by definition \cref{approxbalmet} of $h_k(p)$.
Then by \cref{sigmaestlem} and \cref{approxbal}, we can take the Taylor expansion as $p\to+\infty$ of formula \cref{ddbar}
to get that for any $k,\,k_0,\,m\in\N$
with $k\geq k_0>n+1+m/2$, there exists $C>0$ such
for all
$B\in\Herm(\C^{n_p})$ with $\|B\|_{\xi_p}\leq C^{-1}p^{-k_0}$,
we have
\begin{equation}\label{approxhBp'}
\left|\om_{e^{B}e^{L_{\xi_p}/p}\s_k(p)}-\om_{h_k(p)}\right|_{\CC^{m-2}}
\leq C p^{-k_0-1}\,.
\end{equation}
By formula \cref{approxbalmet} for $h_k(p)$ and the
corresponding formula for $\om_{h_k(p)}$ as in \cref{ddbar},
this implies the first inequality of \cref{approxomkplem}.

%For the remaining statements, let us first recall
%that by \cref{pullbackprop}
%and formula \cref{etxi*dnu}, for all  we have
%$\Vol(d\nu_\s)=\Vol(d\nu_{e^{L_\xi}\s})$.
%As $B\in\Herm(\C^{n_p})^T$ commutes with $L_\xi$ for all
%$\xi\in\sqrt{-1}\Lie T$, we can thus assume without loss of
%generality that $\xi=0$.

Let us now establish the second inequality of \cref{approxomkplem}.
For any
$p\in\N^*$ and $B\in\Herm(\C^{n_p})$, formula \cref{dnuef/dnu} and
\cref{hkp=rhosighs} give
\begin{multline}\label{logvol}
\log\frac{d\nu_{e^Be^{L_{\xi_p/p}}\s_k(p)}}{d\nu_{h_k(p)}}
-\frac{1}{p}\log\frac{\Vol(d\nu_{h_k(p)})}{\Tr\left[e^{L_{\xi_p}/p}\right]}\\
=-\frac{1}{p}\log\left(\frac{\Vol(d\nu_{h_k(p)})}{\Tr\left[e^{L_{\xi_p}/p}\right]}\rho_{h_k^p(p)}\right)
-\frac{1}{p}\log\sigma_{e^{L_{\xi_p/p}}\s_k(p)}(e^{2B})\,.
\end{multline}
Taking the Taylor expansion as $p\to+\infty$ of the right hand side of
\cref{logvol} in the same way as we did to deduce \cref{approxhBp'} from
\cref{ddbar}, for any $k,\,k_0\in\N$ with $k\geq k_0>n+1+m/2$,
we get a constant $C>0$ such that
for all $p\in\N^*$ and all $B\in\Herm(\C^{n_p})$ with
$\|B\|_{tr}\leq C^{-1}p^{-k_0}$, we have
%\begin{equation}
%\begin{split}
%&\left|\log\frac{d\nu_{e^Be^{L_{\xi_p/p}}\s_k(p)}}{d\nu_{h_k(p)}}
%-\frac{1}{p}\log\frac{\Vol(d\nu_{h_k(p)})}{\Tr\left[e^{L_{\xi_p}/p}\right]}\right|_{\CC^0}\\
%&=\frac{1}{p}\left|\log\left(1+\left(\frac{\Vol(d\nu_{h_k(p)})}{\Tr\left[e^{L_{\xi_p}/p}\right]}\rho_{h_k^p(p)}-1\right)\right)+\log\left(1+\sigma_{\s_k(p)}\left(e^{2B}-\Id\right)\right)
%\right|_{\CC^0}\\
%&\leq Cp^{-k_0-1}\,.
%\end{split}
%\end{equation}
%Note that we used the fact from \cref{Bergdiagexp}
%that $n_p=O(p^{-n})$, which also shows that
%$\frac{1}{p}\log\frac{\Vol(d\nu_{h_k(p)})}{n_p}\to 0$
%as $p\to+\infty$.
%
%In other words, there exist
%constants $V_p>0$ satisfying $V_p\to 1$ as $p\to+\infty$
%such that for all $B\in\Herm(\C^{n_p})^T$ with
%$\|B\|_{tr}\leq C^{-1}p^{-k_0}$, we have
\begin{equation}\label{vol1}
\left|\frac{d\nu_{e^Be^{L_{\xi_p}/p}\s_k(p)}}
{d\nu_{h_k(p)}}-\frac{1}{p}\log\frac{\Vol(d\nu_{h_k(p)})}
{\Tr\left[e^{L_{\xi_p}/p}\right]}\,\right|_{\CC^0}
\leq Cp^{-k_0-1}\,.
\end{equation}
Taking the integral of both sides against the probability measure
$d\nu_{h_k(p)}/\Vol(d\nu_{h_k(p)})$,
we see that there is $C>0$ such that the constants $V_p>0$
for all $p\in\N^*$ satisfy
\begin{equation}\label{vol2}
\left|\frac{\Vol\left(d\nu_{e^{B}e^{L_{\xi_p}/p}\s_k(p)}\right)}{\Vol(d\nu_{h_k(p)})}-\frac{1}{p}\log\frac{\Vol(d\nu_{h_k(p)})}
{\Tr\left[e^{L_{\xi_p}/p}\right]}\right|<
Cp^{-k_0-1}\,.
\end{equation}
Using the fact from
\cref{pullbackprop} and formula \cref{etxi*dnu} that
$\Vol(d\nu_{\s})=\Vol(d\nu_{e^{L_{\xi}}\s})$ for all
$\s\in\BB(H^0(X,L^p))^T$ and $\xi\in\sqrt{-1}\Lie T$,
we get the second inequality of \cref{approxhBp}
by combining inequalities \cref{vol1} and \cref{vol2}. The last
statement then follows from the inequality \cref{vol2},
recalling that
$C^{-1}p^{n}<\Tr\left[e^{L_{\xi_p}/p}\right]<Cp^{n}$
and $\Vol(d\nu_{h_k(p)})\to\Vol(d\nu_{h_\infty})$ as $p\to+\infty$, so that the second term of the left hand side of \cref{vol2} goes to $0$ as $p\to+\infty$. This concludes the proof.
\end{proof}

\subsection{Moment map picture and existence}
\label{existencesec}

The goal of this Section is to give a proof of existence and convergence
in \cref{mainth}. In order to do so, it will be convenient to
introduce some extra notations.
Recall the setting of \cref{findimsec}, and
write $\Prod(H^0(X,L^p))^K$
for the space of $K$-invariant Hermitian inner products on $H^0(X,L^p)$.
Given a fixed basis $\s\in\BB(H^0(X,L^p))^T$
inducing $H_{\s}\in\Prod(H^0(X,L^p))^K$,
consider the induced identification
$H^0(X,L^p)\simeq\C^{n_p}$. Via this identification,
write $\GL(\C^{n^p})^K\subset\GL(\C^{n_p})^T$
for the group of invertible endomorphisms
commuting with the induced action of $K$ on $\C^{n_p}$, and
write $U(n_p)^K\subset\GL(\C^{n^p})^K$ for the subgroup
of unitary matrices commuting with the action of $K$.
The action of $G\in\GL(\C^{n^p})^K$ on $\s$ induces again a
$K$-invariant product $H_{G\s}\in\Prod(H^0(X,L^p))^K$,
and as for \cref{isomHermT}, we have
an identification
\begin{equation}\label{isomHermK}
\cL(H^0(X,L^p),H_{G\s})^K\simeq\Herm(\C^{n_p})^K\,,
\end{equation}
where $\Herm(\C^{n_p})^K\subset\Herm(\C^{n_p})^T$ is the space
of Hermitian matrices commuting with the induced action of $K$ on $\C^{n_p}$.
Note that the second statement of \cref{Trmu} precisely says that
$\mu_\xi(G\s)\in\Herm(\C^{n_p})^K$.
%In all the results stated in this
%Section, the fixed basis $\s\in\BB(H^0(X,L^p))^T$ considered above will
%always be clear from the context.

Our strategy for the proof of \cref{mainth}
is based on the following fundamental
link between the anticanonical moment map
of \cref{momentdef} and the Berezin-Toeplitz quantum channel
of \cref{quantchanrel}.
Following \cref{h=rhohsrmk},
let $h^p\in\Met^+(L^p)^K$ be a $K$-invariant positive Hermitian metric,
let $\s_p\in\BB(H^0(X,L^p))^T$ be orthonormal with respect to
$L^2(h^p)$, and consider the induced identification
\cref{isomHermK}.
By \cref{Trmu}, for any $\xi\in\sqrt{-1}\Lie T$ and
$A\in\Herm(\C^{n_p})^K$ we have
\begin{equation}
D_{\s_p}\mu_\xi(A):=\dt\Big|_{t=0}\,
\mu_\xi(e^{tA}\s_p)\in\Herm(\C^{n_p})^K\,.
\end{equation}
Let $\<\cdot,\cdot \>_{\xi}$ be the scalar product
\cref{Trxi} on $\Herm(\C^{n_p})^K$, and write $\|\cdot\|_{\xi}$
for the associated norm.

\begin{prop}\label{dmu}
Assume that $h^p\in\Met^+(L^p)^K$ is anticanonically balanced relative to
$\xi_p\in\sqrt{-1}\Lie T$, and let $\s_p\in\BB(H^0(X,L)^p)^T$ be
orthonormal with respect to $L^2(h^p)$. Then
for all $A\in\Herm(\C^{n_p})^K$ satisfying $\<\Id,A\>_{\xi_p}=0$,
we have
\begin{equation}\label{TrAdmuA}
\frac{\Tr[e^{L_{\xi_p}/p}]}{2\Vol(d\nu_h)}
\<A,D_{\s_p}\mu_{\xi_p}(A)\>_{\xi_p}=\|A\|_{\xi_p}^2-
\left(1+\frac{1}{p}\right)\<A,\EE_{h^p}^{\xi_p}(A)\>_{\xi_p}\,.
\end{equation}
\end{prop}
\begin{proof}
Let us first compute
$D_{\textbf{s}}\mu_\xi(A)\in\Herm(\C^{n_p})^T$,
for general $\xi\in\sqrt{-1}\Lie T$, $\textbf{s}\in\BB(H^0(X,L^p))^T$ and
$A\in\Herm(\C^{n_p})^T$.
Using \cref{FSvar}, \cref{pullbackprop} and
formula \cref{cohstatecoordxi}, in the identification \cref{isomHermT}
we get
\begin{multline}\label{dmucomput}
D_{\textbf{s}}\mu_\xi(A)
=\left(\int_X\,\dt\Big|_{t=0}
\<e^{tA}s_j,e^{tA}s_k\>_{h_{e^{tA}e^{L_\xi/2p}\s_p}^p}
\,d\nu_{e^{L_\xi/2p}\textbf{s}}
\right)_{j,\,k=1}^{n_p}\\
+\left(\int_X\,
\<s_j,s_k\>_{h_{e^{L_\xi/2p}\s}^p}
\,\dt\Big|_{t=0}d\nu_{e^{tA}e^{L_\xi/2p}\textbf{s}}\right)_{j,\,k=1}^{n_p}
-\left(\dt\Big|_{t=0}\frac{\Vol(d\nu_{e^{tA}\textbf{s}})}{\Tr[e^{L_\xi/p}]}\right)
\Id\\
=\int_X e^{-L_\xi/2p}(A\Pi_{e^{L_\xi/2p}\s}+\Pi_{e^{L_\xi/2p}\s} A-2\sigma_{e^{L_\xi/2p}\s}(A)\Pi_{e^{L_\xi/2p}\s})e^{-L_\xi/2p}\,d\nu_{e^{L_\xi/2p}\textbf{s}}\\
-\frac{2}{p}\int_X\,
\sigma_{e^{L_\xi/2p}\s}(A)
e^{-L_\xi/2p}\Pi_{e^{L_\xi/2p}\s}e^{-L_\xi/2p}
\,d\nu_{e^{L_\xi/2p}\textbf{s}}\\
+\frac{\Vol(d\nu_{\s})}{\Tr[e^{L_\xi/p}]}\,
\left(\frac{2}{p}\int_X\,\sigma_\s(A)\,d\nu_\s\right)\,\Id\,.
\end{multline}
Then by \cref{pullbackprop,FSvar2},
for any $A\in\Herm(\C^{n_p})^T$ with $\Tr[e^{L_\xi/p}\,A]=0$,
we get
\begin{multline}\label{dmucomputfinal}
\frac{1}{2}
\Tr[e^{L_\xi/p}\,A\,D_{\textbf{s}}\mu_\xi(A)]\\
=\int_X\sigma_{e^{L_\xi/2p}\s}(A^2)
\,d\nu_{e^{L_\xi/2p}\s}-
\left(1+\frac{1}{p}\right)\int_X\sigma_{e^{L_\xi/2p}\s}(A)^2\,
d\nu_{e^{L_\xi/2p}\s}
\end{multline}
On the other hand, from \cref{dualpropT,quantchanrel},
for any $h^p\in\Met^+(L^p)^T$ we get
\begin{equation}\label{quantchancomput}
\begin{split}
\Tr[e^{L_{\xi}/p}A^2]&=\int_X \sigma_{h^p}(A^2)\,
\sigma_{h^p}(e^{L_{\xi}/p})\,\rho_{h^p}
\,d\nu_{h}\,,\\
\Tr\left[e^{L_{\xi}/p}A\,\EE_{h^p}^{\xi}(A)\right]&=
\int_X\,\sigma_{h^p}(A)^2
\sigma_{h^p}(e^{L_{\xi}/p})\,\rho_{h^p}
\,d\nu_{h}\,.
\end{split}
\end{equation}
Using \cref{relbalrho} and comparing formulas
\cref{dmucomputfinal} and \cref{quantchancomput},
this gives the result.
\end{proof}

For any $k\in\N$ and $p\in\N^*$ big enough, 
consider the positive Hermitian metric
$h_k^p(p)\in\Met^+(L)^K$ constructed in \cref{approxbal}, and
let $\textbf{s}_k(p)\in\BB(H^0(X,L^p))^T$
be orthonormal with respect to $L^2(h_k^p(p))$.
Consider the induced identification \cref{isomHermK}, and let
$\xi_p\in\sqrt{-1}\Lie T$ be given by \cref{xipmin}.
The following result constitutes the heart of our strategy,
using the asymptotics of the spectral gap of the Berezin transform
given in \cref{Bpgap}
to give a crucial estimate from above on
the Berezin-Toeplitz quantum channel.

\begin{theorem}\label{quantchanbnd}
There exists $\epsilon>0$
such that for all $k\geq n+3$, all $p\in\N^*$ big enough and for
any $A\in\Herm(\C^{n_p})^K$ satisfying
$\<\Id,A\>_{\xi_p}=\<L_\eta,A\>_{\xi_p}=0$ for all
$\eta\in\sqrt{-1}\Lie T$,
we have
\begin{equation}\label{quantchanbndfla}
\<A,\EE^{\xi_p}_{h_k^p(p)}(A)\>_{\xi_p}
\leq\left(1-(1+\epsilon)p^{-1}\right)\|A\|_{\xi_p}^2\,.
\end{equation}
\end{theorem}
\begin{proof}
For any $\s\in\BB(H^0(X,L^p))^T$ and $\eta\in\sqrt{-1}\Lie T$,
we write $\theta_{\s}(\eta):=\theta_{h_\s}(\eta)$ to simplify
notations. By \cref{thetaxiprop,FSvar,pullbackprop}, we have
\begin{equation}\label{sig=theta}
p\,\theta_{\s}\,(\eta)h_{\s}^p=\dt\Big|_{t=0}\phi_{t\xi}^*h_{\s}^p=
\sigma_{\s}(L_\eta)\,h_{\s}^p\,.
\end{equation}
Given $h^p\in\Met^+(L^p)^K$ and $\s_p\in\BB(H^0(X,L^p))^T$
orthonormal with respect to $L^2(h^p)$, \cref{Tuy} and \cref{quantchanrel}
of the quantum channel
then imply that for any $\xi,\,\eta\in\sqrt{-1}\Lie T$, we have
\begin{equation}\label{quantchanbndfla1}
\EE^{\xi}_{h^p}(L_\eta)=p\,T_{h^p}(\theta_{e^{L_\xi/2p}\s_p}(\eta))\,.
\end{equation}
Using \cref{thetaxiprop,approxbal}, we get from formula \cref{hkp=rhosighs}
a constant $C>0$ such that
$|\theta_{e^{L_{\xi_p}/2p}\textbf{s}_k(p)}(\eta)
-\theta_{h_k(p)}(\eta)|_{\CC^0}\leq C|\eta|\,p^{n-k}$,
for all $\eta\in\sqrt{-1}\Lie T$ and $p\in\N^*$.
Hence by \cref{Bergdiagexp,BTquantdef,Tuy} and
as $\|\cdot\|_{tr}\leq Cp^{n/2}\|\cdot\|_{op}$ for some $C>0$, formula \cref{quantchanbndfla1}
implies
\begin{equation}\label{quantchanbndfla2}
\left\|\EE^{\xi_p}_{h_k^p(p)}(L_\eta)-\frac{p}{p+1}L_\eta\right\|_{tr}
\leq C p^{\frac{3n}{2}+1-k}|\eta|\,.
\end{equation}
%We see that $L_\eta$ is a quasi-eigenvector of $\EE^{\xi_p}_{h_k^p(p)}$
%of quasi-eigenvalue $\lambda=1+p^{-1}+O(p^{-2})$.
On the other hand, recall the notation \cref{lambdajhxi} for the
increasing sequence of eigenvalues of
the Tian-Zhu operator $\Delta_h^{\xi}$ of \cref{TZopdef},
for any $h\in\Met^+(L)$ and $\xi\in\sqrt{-1}\Lie T$.
\cref{xipexp} and formula \cref{approxbalmet} show that
for any $j\in\N$, there exists a constant $C_j>0$ such that
$|\lambda_j(h_k(p),\xi_p)-\lambda_j(h_\infty,\xi_\infty)|\leq C_jp^{-1}$,
for all $p\in\N^*$.
Using \cref{Bpgap}, we thus get that for
any $j\in\N$, there exists a constant $C_j>0$ such that for all $p\in\N^*$,
we have $\big|1-\gamma_{j}(h^p_k(p),\xi_p)-p^{-1}\lambda_k(h_\infty,\xi_\infty)\big|
\leq C_jp^{-2}$.
Using \cref{specidprop}, this shows that
there exists $\epsilon,\,C>0$ such that for all $p\in\N^*$,
\begin{multline}\label{specfar}
\Spec(\EE^{\xi_p}_{h_k^p(p)})\cap[1-(1+\epsilon)p^{-1},1-(1-\epsilon)p^{-1}]\\
\subset[1-p^{-1}-C p^{-2},1-p^{-1}+C p^{-2}]\,.
\end{multline}
Recall from \cref{thetaxiprop} that
$\theta_{h_k(p)}:\sqrt{-1}\Lie T\to\cinf(X,\R)$
is an embedding. For any
$\eta_1,\,\eta_2\in\sqrt{-1}\Lie T$,
\cref{Tuy,Tuycor,dualpropT,KSxi,approxbal}
imply that for all $p\in\N^*$, we have
\begin{equation}\label{<Leta1Leta2>}
\frac{\<L_{\eta_1},L_{\eta_2}\>_{\xi_p}}{p^{n+2}}=\<\theta_{h_k(p)}(\eta_1),
\theta_{h_k(p)}(\eta_2)\>_{L^2(h_k(p),\xi_p,p)}+|\eta_1|\,|\eta_2|
\,O(p^{-1})\,,
\end{equation}
so that $\|L_{\eta}\|_{\xi_p}\geq\epsilon |\eta| p^{\frac{n}{2}+1}$
for some $\epsilon>0$ not depending of $p\in\N^*$ big enough.
Note also from \cref{Tuy,xipexp} that there is $C>0$ such that
the norm induced by \cref{Trxi} satisfies
$C^{-1}\|\cdot\|_{tr}\leq\|\cdot\|_{\xi_p}\leq C\|\cdot\|_{tr}$
for all $p\in\N^*$.

Set now $k\geq n+3$. Using \cref{Tuy} together with an
elementary Lemma on quasi-modes (see for instance
\cite[Lem.\,2.1]{IKP20b}), formulas
\cref{quantchanbndfla2,specfar,<Leta1Leta2>}
imply the existence of a constant $C>0$ and an eigenvector
$\til{L_\eta}\in\Herm(\C^{n_p})^K$ of $\EE^{\xi_p}_{h_k^p(p)}$
with associated eigenvalue
$\lambda\in[1-p^{-1}-C p^{-2},1-p^{-1}+C p^{-2}]$ and such that
\begin{equation}\label{tilLeta}
\|\til{L_\eta}\|_{\xi_p}=\|L_{\eta}\|_{\xi_p}\geq\epsilon |\eta|
p^{\frac{n}{2}+1}\quad\text{and}\quad
\left\|\til{L_\eta}-L_\eta\right\|_{\xi_p}\leq C p^{\frac{n}{2}-1}|\eta|\,.
\end{equation}
Using \cref{TZ,Bpgap} again,
we know that the dimension of the sum of eigenspaces
associated with the right hand side of \cref{specfar} is equal
to $\dim T$, so that by formula \cref{<Leta1Leta2>}, the operators
$\til{L_\eta}\in\Herm(\C^{n_p})^K$ for all $\eta\in\sqrt{-1}\Lie T$
generate this subspace as soon as $p\in\N^*$
is big enough.
%As $\EE^{\xi_p}_{h_k^p(p)} then get from formula \cref{specfar}
%that formula \cref{quantchanbndfla} holds
%for any $A\in\cL(\HH_{k}(p))^K$ belonging to
%the orthogonal of
%\begin{equation}\label{IdoplustilLeta}
%\<\Id_{\HH_p}\>\oplus
%\<\til{L_\eta}~|~\eta\in\sqrt{-1}\Lie K\>\subset\cL(\HH_p)^K\,.
%\end{equation}
As we have $\EE^{\xi_p}_{h_k^p(p)}(\Id)=\Id$ by \cref{dualprop}, \cref{TZopdef}, \cref{specidprop} and formula \cref{specfar}
imply that \cref{quantchanbndfla} holds for
$A\in\Herm(\C^{n_p})^K$ belonging to the orthogonal of the subspace
generated by $\Id_{\HH_p}$ and $\til{L_\eta}$ for all
$\eta\in\sqrt{-1}\Lie K$.
Now for any $A\in\Herm(\C^{n_p})^K$ satisfying
$\<\Id,A\>_{\xi_p}=\<L_\eta,A\>_{\xi_p}=0$ for all $\eta\in\sqrt{-1}\Lie T$,
formula \cref{tilLeta} and Cauchy-Schwartz imply
the existence of $C>0$ such that
\begin{equation}
\<\til{L_\eta},A\>_{\xi_p}\leq C |\eta|\,p^{\frac{n}{2}-1}\|A\|_{\xi_p}
\leq\frac{C}{\epsilon}\,p^{-2}\|\til{L_\eta}\|_{\xi_p}\|A\|_{\xi_p}\,.
\end{equation}
It then suffices to consider the splitting of $A$ into the eigenspaces
of $\EE^{\xi_p}_{h_k^p(p)}$ to get the result.
%On the other hand, using Cauchy-Schwartz again, we know that
%$|\sigma_\s(B)|_{\CC^0}\leq \|B\|_{tr}$ for all $B\in\cL(\HH_k(p))$, and
%we get a constant $C>0$ such that
%$\|\EE^{\xi_p}_{h_k^p(p)}(B)\|_{\xi_p}\leq Cp^{\frac{n}{2}}\|B\|_{\xi_p}$
%by \cref{quantchanrel}.
%Considering the splitting of $A$ into the subspace
%\cref{IdoplustilLeta} and its orthogonal, this implies the result
%for any $k>2n+2$.
\end{proof}

Using the relation between the derivative of the moment map
and the relative quantum channel given in \cref{dmu},
we can now apply the estimate of \cref{quantchanbnd}
to give an estimate from below for the
derivative of the moment map at the approximately balanced
bases. This lower bound
constitutes the core of the proof of \cref{mainth},
%following
%Donaldson's method in \cite{Don01},
and this shows how Berezin-Toeplitz
quantization can be used to to bypass the delicate geometric
argument in the proofs of
Donaldson \cite{Don01} and Phong and Sturm \cite{PS04}
of the analogous result for the original notion of balanced metrics.

\begin{cor}\label{dmuapprox}
For any $k,\,k_0\in\N$ with $k\geq k_0\geq n+3$,
there exists $\epsilon>0$ such that for all
$p\in\N^*$ big enough, for all
$B\in\Herm(\C^{n_p})^K$ with
$\|B\|_{\xi_p}\leq \epsilon p^{-k_0}$
and all $A\in\Herm(\C^{n_p})^K$ satisfying
$\<\Id,A\>_{\xi_p}=\<L_{\eta},A\>_{\xi_p}=0$
for all $\eta\in\sqrt{-1}\Lie K$, we have
\begin{equation}\label{TrAdmuA>A}
\frac{\Tr[e^{L_{\xi_p}/p}]}{\Vol(d\nu_{e^B\s_k(p)})}
\<A,D_{e^B\s_k(p)}\mu_{\xi_p}(A)\>_{\xi_p}
\geq\frac{\epsilon}{p}\,\|A\|_{\xi_p}^2\,.
\end{equation}
\end{cor}
\begin{proof}
The proof uses the fact that \cref{dmu} is approximately satisfied
for approximately balanced metrics.
First note that for all $\s\in\BB(H^0(X,L^p))^T$ and all
$A\in\Herm(\C^{n_p})^K$,
Cauchy-Schwartz inequality and the fact that
$\Pi_\s$ is a rank-$1$ projector implies that
$|\sigma_\s(A)|_{\CC^0}\leq \|A\|_{tr}$ and
$|\sigma_\s(A^2)|_{\CC^0}\leq\|A\|^2_{tr}$.
%Using \cref{FSvar2}, the submultiplicativity of the operator
%norm and the fact that $\|B\|_{op}\leq\|B\|_{tr}$ for all
%$B\in\Herm(\C^{n_p})$, the inequality \cref{trid}
%also shows that
%that for any $\epsilon>0$,
%there is a constant $C>0$
%such that for all $B\in\Herm(\C^{n_p})$
%with $\|B\|_{tr}\leq\epsilon p^{-k_0}$ and all $p\in\N^*$,
%we have
%\begin{equation}\label{sigmaeBest}
%\begin{split}
%|\sigma_{e^{B}\s}(A)^2-\sigma_\s(A)^2|_{\CC^0}
%&\leq 2\|A\|_{tr}\left|
%\sigma_{\s}(e^{2B})^{-1}
%\sigma_{\s}(e^{B}Ae^{B})-\sigma_\s(A)\right|_{\CC^0}\\
%&\leq Cp^{-k_0}\|A\|_{tr}^2\,,
%\end{split}
%\end{equation}
%and in the same way,
%\begin{equation}
%\begin{split}
%|\sigma_{e^{B}\s}(A^2)-\sigma_\s(A^2)|_{\CC^0}
%&=\left|
%\sigma_{\s}(e^{2B})^{-1}
%\sigma_{\s}(e^{B}A^2e^{B})-\sigma_\s(A^2)\right|_{\CC^0}\\
%&\leq Cp^{-k_0}\|A\|_{tr}^2\,.
%\end{split}
%\end{equation}
%Recall that
%$\s_k(p)\in\BB(H^0(X,L^p))$ is an orthonormal
%basis for $L^2(h^p_k(p))$, for all $p\in\N^*$ big enough.
%let us first write
%\begin{equation}
%\til{\rho}_{\s_k(p)}=
%\sigma_{h_k(p)}(e^{L_\xi/p})
%\,\rho_{h_k(p)}\,.
%\end{equation}
Consider the operator $S_p$ acting on
$A\in\cL(H^0(X,L^p),H_{\s})^K$ by
\begin{equation}\label{Spdef}
S_p(A):=A-\left(1+\frac{1}{p}\right)\EE_{h_k^p(p)}^{\xi_p}(A)\,.
\end{equation}
Then
plugging $\s=e^{B}\s_k(p)$ into formula
\cref{dmucomputfinal} and comparing with
\cref{quantchancomput}, we can use
\cref{approxbal,approxomkplem}
%as in the proof of
%\cite[Prop.\,3.9]{Ioo20}
to get a constant $C>0$
such that for all $p\in\N^*$,
for all $B\in\Herm(\C^{n_p})^K$
with $\|B\|_{\xi_p}\leq C^{-1} p^{-k_0}$ and 
for all $A\in\Herm(\C^{n_p})^K$
with $\<\Id,A\>_{\xi_p}=0$,
we get that
\begin{multline}\label{dmuestN}
\left|\frac{\Tr[e^{L_{\xi_p}/p}]}{\Vol(d\nu_{e^B\s_k(p)})}
\<A,D_{e^B\s_k(p)}\,
\mu_{\xi_p}(A)\>_{\xi_p}-2\<A,S_p(A)\>_{\xi_p}\,\right|
%\leq\int_X\left|\sigma_{e^B\s_k(p)}(A^2)
%\frac{\Tr[e^{L_{\xi_p}/p}]}{\Vol(d\nu_{e^B\s_k(p)})}
%\frac{d\nu_{e^B\s_k(p)}}{d\nu_{h_k(p)}}-
%\sigma_{h_k^p(p)}(A^2)\sigma_{h_k^p(p)}(e^{L_{\xi_p}/p})
%\rho_{h_k^p(p)}\right|\,
%d\nu_{h_k(p)}\\
%+\int_X\left|\sigma_{e^B\s_k(p)}(A)^2\frac{\Tr[e^{L_{\xi_p/p}}]}
%{\Vol(d\nu_{e^B\s_k(p)})}\frac{d\nu_{e^B\s_k(p)}}{d\nu_{h_k(p)}}-
%\sigma_{h_k^p(p)}(A)^2\sigma_{h_k^p(p)}(e^{L_{\xi_p}/p})\rho_{h_k^p(p)}
%\right|d\nu_{h_k(p)}\\
\leq C\,p^{n-k_0}\|A\|_{\xi_p}^2\,.
\end{multline}
We then get the result from \cref{quantchanbnd} by taking $k_0\geq n+3$.
\end{proof}

Thanks to the lower bound of \cref{dmuapprox},
we can now follow the standard
strategy of Donaldson in \cite{Don01}, adapted to the case of
general $\Aut(X)$.
%, in the manner of \cite[\S\,5.2]{ST17}
%and \cite[\S\,6]{Sey17}.
The following result is inspired from the moment map
Lemma of Donaldson in \cite[Prop.\,17]{Don01}.
We provide a proof working in greater generality, as we
do not claim that \cref{momentdef}
defines a moment map of any kind.

\begin{prop}\label{Donlem}
Consider a $\CC^1$-map
\begin{equation}\label{muDonlemdef}
\mu:\BB(H^0(X,L^p))^T\to\Herm(\C^{n_p})^T
\end{equation}
satisfying $\mu(\textbf{s})\in\Herm(\C^{n_p})^K$
for all $\s\in\BB(H^0(X,L^p))^T$ inducing $H_{\s}\in\Prod(H^0(X,L^p))^K$,
and such that $\mu(U\s)=U\mu(\s)\,U^*$ for all $U\in U(n_p)^T$ and
$\<\Id,\mu(\textbf{s})\>_{\xi_p}=\<L_\eta,\mu(\textbf{s})\>_{\xi_p}=0$
for all $\eta\in\sqrt{-1}\Lie T$.

Assume that there exist $\s\in\BB(H^0(X,L^p))^T$ inducing
$H_{\s}\in\Prod(H^0(X,L^p))^K$
and $\lambda,\,\delta>0$ such that
\begin{itemize}
\item[$(1)$] $\lambda\,\|\mu_{\xi_p}(\s)\|_{\xi_p}<\delta\,;$
\item[$(2)$] $\lambda\<A,D_{e^{B}\s}\mu_{\xi_p}(A)\>_{\xi_p}
\geq \|A\|^2_{\xi_p},
~\text{for all}~B\in\Herm(\C^{n_p})^K~
\text{such that}\,
\|B\|_{\xi_p}\leq\delta~\text{and all}\\
\,A\in\Herm(\C^{n_p})^K~\,\text{such that}~
\<\Id,A\>_{\xi_p}=\<L_\eta,A\>_{\xi_p}=0
~\,\text{for all}~\,\eta\in\sqrt{-1}\Lie T$.
\end{itemize}
Then there exists $B\in\Herm(\C^{n_p})^K$
with $\|B\|_{\xi_p}\leq\delta$ and $\mu(e^{B}\s)=0$.
\end{prop}
\begin{proof}
%Restricting to the subspace of $\BB(H^0(X,L^p))^T$ of bases of the form
%$e^B\s$, for any $B\in\cL(H^0(X,L^p),H_\s)^K$ and
%$\s\in\BB(H^0(X,L^p))^T$ fixed as in the statement, the proof is
%strictly analogous to the proof of \cite[Prop.\,3.9]{Ioo20},
%using the fact from \cref{xipmin} and
%\cref{Trmu} that $\mu_{\xi_p}(\textbf{s})\in\cL(H^0(X,L^p),H_\s)^K$
%and $\<\Id,\mu_{\xi_p}(\textbf{s})\>=\<L_\eta,\mu_{\xi_p}(\textbf{s})\>=0$,
%for all $\eta\in\sqrt{-1}\Lie T$.
First note that for any $A,\,B\in\Herm(\C^{n_p})^T$, $U\in U(n_p)^T$
and $\s\in\BB(H^0(X,L^p))^T$, using that $\mu(U\s)=U\mu(\s)\,U^*$,
we get
\begin{equation}
\begin{split}
\Tr[A\,D_{U\s}\mu(A)\,e^{L_{\xi_p}/p}]
&=\dt\Big|_{t=0}\Tr[A\,\mu(e^{tA}U\s)\,e^{L_{\xi_p}/p}]\\
&=\Tr[U^*AU\,D_\s\mu(U^*AU)\,e^{L_{\xi_p}/p}]\,.
\end{split}
\end{equation}
Thus assumption (2) is
equivalent to
\begin{itemize}
\item[$(2')$] $\lambda\<A,D_{Ue^{B}\s}\mu(A)\>_{\xi_p}\geq \|A\|^2_{\xi_p},
~\text{for all}~B\in\Herm(\C^{n_p})^K~
\text{such that}\,
\|B\|_{\xi_p}\leq\delta\,,~\text{all}\\ U\in U(n_p)^K~
\text{and all}~
\,A\in\Herm(\C^{n_p})^K~\,\text{such that}~
\<\Id,A\>_{\xi_p}=\<L_\eta,A\>_{\xi_p}=0~
\text{for all}\\ \eta\in\sqrt{-1}\Lie T$.
\end{itemize}
Let now $\s\in\BB(H^0(X,L^p))^T$
inducing $H_{\s}\in\Prod(H^0(X,L^p))^K$
be such that assumptions $(1)$ and $(2)$ are satisfied,
and consider the induced identification \cref{isomHermK}.
Then the map
\begin{equation}\label{ProdsymK}
\begin{split}
\GL(\C^{n_p})^K&\longrightarrow
\Prod(H^0(X,L^p))^K\\
G&\longmapsto H_{G\s}\,,
\end{split}
\end{equation}
identifies $\Prod(H^0(X,L^p))^K$
with the quotient of $\GL(C^{n_p})^K$ by $U(n_p)^K$.
This realizes $\Prod(H^0(X,L^p))^K$ as a
symmetric space, whose tangent space at every point
is naturally identified with $\Herm(\C^{n_p})^K$.
The scalar product $\<\cdot,\cdot\>_{\xi_p}$
then makes this space into
a complete Riemannian manifold, whose 
geodesics are of the
form
\begin{equation}\label{geodK}
t\longmapsto H_{e^{tB}\s}\in
\Prod(H^0(X,L^p))^K,\,t\in\R\,,
\end{equation}
for all $B\in\Herm(\C^{n_p})^K$.

Note on the other hand that the tangent space of the
orbit $\GL(\C^{n_p})^K.\,\s\subset\BB(H^0(X,L^p))^T$
is naturally identified with the space of endomorphisms
commuting with the action of $K$ on $\C^{n_p}$.
Then by assumption, the restriction of
the map \cref{muDonlemdef} to this orbit
can be identified with a vector field along this orbit, and we
define $\s_t\in\GL(\C^{n_p})^K.\,\s$
for all $t>0$ as the solution of the ODE
\begin{equation}\label{stdef}
\left\{
\begin{array}{l}
  \dt\,\s_t=-\mu(\s_t)\quad\text{for all}\quad t\geq 0\,, \\
  \\
  \s_0=\s\,.
\end{array}
\right.
\end{equation}
If $\mu(\s)=0$, then the result is trivially satisfied,
so that we can assume $\mu(\s)\neq 0$, in which case
$\mu(\s_t)\neq 0$ for all $t\geq 0$.
Let $t_0\geq 0$ be such that
there exist $U_t\in U(n_p)^K$ and
$B_t\in\Herm(\C^{n_p})^K$
with $\|B_t\|_{\xi_p}\leq\delta$ such that $\s_t=U_te^{B_t}\s$
for all $t\in[0,t_0]$.
Using assumption
$(2')$ with $A:=\mu(\s_t)$, for all $t\in[0,t_0]$ we have
\begin{equation}
-\lambda\dt\|\mu(\s_t)\|^2_{\xi_p}=
2\lambda\<\mu(\s_t)\,D_{\s_t}\mu(\mu(\s_t))\>_{\xi_p}
\geq 2\|\mu(\s_t)\|^2_{\xi_p}\,.
\end{equation}
By derivation of the square, this implies
$\lambda\dt\|\mu(\s_t)\|_{\xi_p}
\leq-\|\mu(\s_t)\|_{\xi_p}$ for all
$t\in[0,t_0]$, so that using Grönwall's lemma with initial
condition $(1)$ and as
$\mu(\s_t)=U_t\mu(e^{B_t}\s)\,U^*_t$, we get
\begin{equation}\label{Gronwallappli}
\|\mu(e^{B_t}\s)\|_{\xi_p}=
\|\mu(\s_t)\|_{\xi_p}\leq e^{-t/\lambda}\,\|\mu(\s)\|_{\xi_p}<
\frac{\delta}{\lambda}\,e^{-t/\lambda}\,.
\end{equation}
%Recall the quotient space \cref{Prodsym}, endowed
%with the Riemannian metric induced by $\<\cdot,\cdot\>_{\xi_p}$
%and whose geodesics are geodesics are the image of the
%$1$-parameter
%groups of the action of $\GL(\C^{n_p})^T$ as in formula
%\cref{geod}.
%seen as a tangent vector descending to itself in the quotient.
%For any $\s\in\BB(H^0(X,L^p))$, write
%$[\s]\in\BB(H^0(X,L^p))/U(n_p)$ for its image,
%Note that the identification
%and note that as $\mu_\nu(\s)\in\Herm(\C^{n_p})$,
%it descends to itself in the quotient via \cref{mutgt}.
Then by equation \cref{stdef}, the
Riemannian length $L(t_0)\geq 0$
of the path
$\{t\mapsto H_{\s_t}\}_{t\in[0,t_0]}
\subset\Prod(H^0(X,L^p))^K$ satisfies
\begin{equation}\label{path<delta}
L(t_0)=\int_0^{t_0}\,\|\mu(\s_t)\|_{\xi_p}\,dt
<\frac{\delta}{\lambda}\int_0^{+\infty}
e^{-t/\lambda}\,dt=\delta\,.
\end{equation}
This means that there exists $\epsilon>0$
such that all points of
$\{t\mapsto H_{\s_t}\}_{t\in[0,t_0+\epsilon]}$
can be joined by a geodesic of
length strictly less than $\delta$, i.e., that for each
$t\in[0,t_0+\epsilon]$, there exists
$B_t\in\Herm(\C^{n_p})^K$
with $\|B_t\|_{\xi_p}\leq\delta$ such that
$H_{\s_t}=H_{e^{B_t}\s}$, so that there exists
$U_t\in U(n_p)^K$
such that $\s_t=U_te^{B_t}\s$.
Thus $I:=\{t_0\geq 0\,|\,L(t_0)<\delta\}$
is non-empty, open and closed in $[0,+\infty[$,
so that $I=[0,+\infty[$.
In particular, the path $\{t\mapsto H_{\s_t}\}_{t>0}$
has total Riemannian length strictly less than $\delta$,
so that it converges to a limit point 
$H_{e^{B_\infty}\s}\in\Prod(H^0(X,L^p))^K$ by completeness,
with
$B_\infty\in\Herm(\C^{n_p})^K$ satisfying
$\|B_\infty\|_{\xi_p}\leq\delta$.
Finally, inequality \cref{Gronwallappli}
for all $t>0$
implies
\begin{equation}
\|\mu(e^{B_\infty}\s)\|_{\xi_p}=
\lim_{t\fl+\infty}\|\mu(e^{B_t}\s)\|_{\xi_p}=0\,.
\end{equation}
This gives the result.
\end{proof}

{\noindent
\textbf{Proof of existence and convergence in \cref{mainth}.}}
%First note from \cref{cohstateprojdef,BTquantdef} that for any
%$h\in\Met^+(L)$, as $\Pi_{h^p}$ is a rank-$1$ projector
%for any $p\in\N^*$ large enough, for any $A\in\cL(\HH_p)$ we have
%\begin{equation}\label{APi=sigmaAPi}
%A\Pi_{h^p}=\sigma_{h^p}(A)\Pi_{h^p}\,.
%\end{equation}
%Consider the setting of \cref{BTsecT}
Let $h^p\in\Met^+(L^p)^K$,
let $\s_p\in\BB(H^0(X,L^p))^T$ be orthonormal with respect to $L^2(h^p)$,
and consider the identification \cref{isomHermK}.
For any $\xi\in\sqrt{-1}\Lie T$, using \cref{FSvar2} and
formulas \cref{intPirho=Id,cohstatecoordxi},
we get from \cref{momentdef} the following inequality,
for all $A\in\Herm(\C^{n_p})^K$,
\begin{equation}\label{muRawn}
\begin{split}
&\frac{\Tr[e^{L_\xi/p}]}{\Vol(d\nu_{\s_p})}\<A,\mu_\xi(\s_p)\>_{\xi}\\
&=\frac{\Tr[e^{L_\xi/p}]}{\Vol(d\nu_{\s_p})}\int_X\,\sigma_{e^{L_\xi/2p}\s_p}(A)\,
d\nu_{e^{L_\xi/2p}\s_p}-
\int_X\,\sigma_{h^p}(e^{L_\xi/p}A)\,\rho_{h^p}\,d\nu_{h}\\
&=\int_X\,\sigma_{e^{L_\xi/2p}\s_p}(A)
\left(\frac{\Tr[e^{L_\xi/p}]}{\Vol(d\nu_{\s_p})}\frac{d\nu_{e^{L_\xi/2p}\s_p}}{d\nu_h}
-\sigma_{h^p}(e^{L_\xi/p})\,\rho_{h^p}\right) d\nu_h\,.
\end{split}
\end{equation}
For any
$k\in\N$ and $p\in\N^*$ big enough, consider now the approximately balanced
metric $h_k(p)\in\Met^+(L)^K$ of \cref{approxbal},
let $\s_k(p)\in\BB(H^0(X,L^p))^T$ be orthonormal with respect to
$L^2(h_k^p(p))$, and let $\{\xi_p\in\sqrt{-1}\Lie T\}_{p\in\N^*}$ be the
sequence of \cref{xipmin}.
By \cref{momentdef} and \cref{Trmu}, the relative moment map
$\mu_{\xi_p}:\BB(H^0(X,L^p))^T\to\Herm(\C^{n_p})^T$
satisfies the basic
assumptions of \cref{Donlem}, so that if suffices to show
that $\s_k(p)\in\BB(H^0(X,L^p))^T$ satisfies the assumptions
$(1)$ and $(2)$ of \cref{Donlem}, for some appropriate
$\lambda,\,\delta>0$.
Using \cref{approxbal,approxomkplem}
and the fact that
$|\sigma_{e^{L_{\xi_p}/2p}\s_k(p)}(A)|_{\CC^0}\leq \|A\|_{tr}$ by
Cauchy-Schwartz inequality, we get from formula \cref{muRawn}
a constant $C>0$ such that for all
$p\in\N^*$ and all $A\in\Herm(\C^{n_p})^K$, we have
\begin{equation}
\frac{\Tr [e^{L_\xi/p}]}{\Vol(d\nu_{\s_k(p)})}
\<A,\mu_{\xi_p}(\s_k(p))\>_{\xi_p}\leq C\|A\|_{\xi_p}\,p^{n-k}\,,
\end{equation}
which implies $\frac{\Tr [e^{L_\xi/p}]}{\Vol(d\nu_{\s_k(p)})}\|\mu_\xi(\s_k(p))\|_{\xi_p}\leq C p^{n-k}$
for all $p\in\N^*$.
Taking $k_0\geq n+3$, we can then choose $k> k_0+n+1$,
and \cref{dmuapprox} together with \cref{Tuycor}
%together with the estimate \cref{npexp} for the
%dimension
shows that $\s_k(p)\in\BB(H^0(X,L^p))^T$ satisfies the assumptions
$(1)$ and $(2)$ of \cref{Donlem}
for $p\in\N^*$ big enough, with
\begin{equation}
\lambda:=\frac{p}{\epsilon}\frac{\Tr[e^{L_{\xi_p}/p}]}
{\Vol(d\nu_{\s_k(p)})}
~~~~~\text{and}~~~~~
\delta:=\frac{C}{\epsilon}p^{n+1-k}\,.
\end{equation}
This gives Hermitian endomorphisms
$B_p\in\Herm(\C^{n_p})^K$ with 
$\|B_p\|_{\xi_p}\leq \epsilon p^{-k_0}$ such that
$\mu_\xi(e^{B_p}\s_k(p))=0$ for all $p\in\N^*$ big enough.
By \cref{momentbal}, the Hermitian metrics
$h_p:=h_{e^{B_p}e^{L_{\xi_p}/p}\s_k(p)}^p\in\Met^+(L^p)^K$
are then anticanonically balanced relative to $\xi_p$
for all $p\in\N^*$ big enough.
%and the associated Kähler forms satisfy
%\begin{equation}
%\om_{h_p}=p\,\om_{e^Be^{L_{\xi_p}/p}\s_k(p)}\,.
%\end{equation}
If we also chose $k_0> n+1+m/2$ for some
$m\in\N$, \cref{approxomkplem}
shows
the $\CC^{m}$-convergence \cref{mainthfla} to the Kähler-Ricci soliton
$\om_{h_\infty}$. Together with \cref{xipexp}, this concludes the proof
of \cref{mainth}.

\qed

\subsection{Energy functional and uniqueness}
\label{uniquesec}

%Recall the anticanonical volume form \cref{dnucandef} the energy functional
%$E:\Met^+(L^p)\to\R$ defined for any $h^p\in\Met^+(L^p)$ by
%\begin{equation}
%E(h^p):=-\log\Vol(d\nu_h)\,.
%\end{equation}
%It has been considered in \cite[\S\,6.3]{BBGZ13}
%as a replacement of
%the \emph{Aubin-Yau functional} in the anticanonical setting.
%Its key property in our context
%is the following Lemma of Berman \cref[Lem.\, 2.6]{Ber13},
%for which we give a proof
%as it is quite elementary.
%Fix a compact torus $T\subset\Aut(X)$, let $p\in\N^*$ be such that
%the Kodaira map \cref{Kod} is well defined and an embedding,
%and consider the setting
%of \cref{findimsec}.

In this Section, we will establish the uniqueness statement in
\cref{mainth} as a quantization of the analogous argument of Tian and
Zhu in \cite[Th.\,3.2]{TZ02}, using our study in \cref{quantFutsec}
of quantized Futaki invariants as obstructions for relative
anticanonically balanced metrics,
and convexity results
due to Berndtsson \cite{Ber09a,Ber09b} applied to
the energy functional associated with the relative moment map
of \cref{momentdef}.

Recall the setting of \cref{findimsec}.
Via the natural action of $\GL(\C^{n_p})^T$ on the space
$\BB(H^0(X,L^p))^T$, the quotient map
\begin{equation}\label{Prodsym}
\begin{split}
\BB(H^0(X,L^p))^T&\longrightarrow\BB(H^0(X,L^p))^T/U(n_p)^T\\
\s&\longmapsto [\s]\,,
\end{split}
\end{equation}
identifies $\BB(H^0(X,L^p))^T/U(n_p)^T$ with the space of
$T$-invariant Hermitian inner products on $H^0(X,L^p)$.
The twisted trace product \cref{Trxi} makes this space into
a complete Riemannian manifold, whose 
geodesics are of the
form
\begin{equation}\label{geod}
t\longmapsto [e^{tA}\s]\in\BB(H^0(X,L^p))^T/U(n_p)^T,\,t\in\R\,,
\end{equation}
for all $A\in\Herm(\C^{n_p})^T$.
Fixing a base point $\s_0\in\BB(H^0(X,L^p))^T$,
the free and transitive action  of $\GL(\C^{n_p})^T$ on $\BB(H^0(X,L^p))^T$
induces an identification
\begin{equation}
\BB(H^0(X,L^p))^T\simeq\GL(\C^{n_p})^T\,,
\end{equation}
and this induces a determinant map
\begin{equation}
\text{det}_{\s_0}:\BB(H^0(X,L^p))^T/U(n_p)^T\longrightarrow\,]0,+\infty[\,.
\end{equation} 
%From now on, we fix
%a base point $H_0\in\Prod(H^0(X,L^p))$, and
%identify any $H\in\Prod(H^0(X,L^p))$ with
%a Hermitian endomorphism $H\in\cL(H^0(X,L^p),H_0)$ via
%the formula 
%\begin{equation}\label{basept}
%H=H_0(H\cdot,\cdot)\,.
%\end{equation}
%Recall that $\Prod(H^0(X,L^p))$ is endowed with a natural
%structure of a symmetric space via the quotient map
%\cref{symmap}, with geodesics given by formula \cref{geod}.
The following energy functional has been introduced in
\cite[\S\,4.2.2]{BW14}.

\begin{defi}\label{Psidef}
The \emph{energy functional} $\Psi_\xi:\BB(H^0(X,L^p))^T/U(n_p)^T\to\R$
relative to $\xi\in\sqrt{-1}\Lie T$
is defined for all $H\in\Prod(H^0(X,L^p))$ by
\begin{equation}\label{Psifla}
\Psi([\s])=-\log\Vol(d\nu_{\s})
-\frac{2}{p}\frac{\log\det_{\s_0}[e^{L_\xi/p}\s]}{\Tr[e^{L_\xi/p}]}\,.
\end{equation}
\end{defi}

For any $\xi\in\sqrt{-1}\Lie T$,
recall the induced scalar product \cref{Trxi} on $\Herm(\C^{n_p})^T$.
The role of the energy functional of \cref{Psidef} comes from the
following identity.

\begin{lem}
For all $\s\in\BB(H^0(X,L^p))^T$
and $A\in\Herm(\C^{n_p})^T$, we have
\begin{equation}\label{dpsi=mu}
\frac{d}{dt}\Big|_{t=0}\Psi_\xi([e^{tA}\s])=\frac{2}{p\,
\Vol(d\nu_\s)}\<\mu_\xi(\s),A\>_{\xi}\,.
\end{equation}
\end{lem}
\begin{proof}
Using \cref{FSvar} and formula \cref{dnuef/dnu},
from \cref{Psidef} we compute
\begin{equation}
\frac{d}{dt}\Big|_{t=0}\Psi([e^{tA}\s])=\frac{2}{p\,
\Vol(d\nu_{\s})}\int_X\,\sigma_{\s}(A)\,
d\nu_{\s}-\frac{2}{p}\frac{\Tr[e^{L_\xi/p}A]}{\Tr[e^{L_\xi/p}]}\,.
\end{equation}
On the other hand, using formula \cref{cohstatecoordxi},
\cref{momentdef} gives
\begin{equation}
\Tr[e^{L_\xi/p}\mu_\xi(\s)A]=
\int_X\,\sigma_{e^{L_{\xi/2p}}\s}(A)\,d\nu_{e^{L_{\xi/2p}}\s}
-\frac{\Vol(d\nu_\s)}{\Tr[e^{L_\xi/p}]}\Tr[e^{L_\xi/p}A]\,.
\end{equation}
From formula \cref{etxi*dnu} and \cref{pullbackprop}, a change
of variable with respect to $\phi_{\xi/2p}\in T_\C$ gives the result.
\end{proof}

By \cref{momentbal}, this implies in particular that
$[\s]\in\BB(H^0(X,L^p))^T/U(n_p)^T$ is a
critical point of $\Psi_\xi:\BB(H^0(X,L^p))^T/U(n_p)^T\to\R$
if and only if
there exists $h^p\in\Met^+(L^p)^T$ anticanonically balanced relative to
$\xi$ with $\s$ orthonormal with respect to $L^2(h^p)$.

The following result is a consequence of the
results of \cite{Ber09a,Ber09b} on positivity of direct images.

\begin{prop}\label{BBGZ}
For any $\xi\in\sqrt{-1}\Lie T$,
the energy functional of \cref{Psidef}
is convex along geodesics of $\BB(H^0(X,L^p))^T/U(n_p)^T$, and
strictly convex except along geodesics of the form
$t\longmapsto [e^{t(L_\eta+c)}\s]$, with $c\in\R$ and $\eta\in\Lie\Aut(X)$
such that $L_\eta\in\cL(H^0(X,L^p),H_\s)^T$.
\end{prop}
\begin{proof}
By definition \cref{geod} of the geodesics in $\BB(H^0(X,L^p))^T/U(n_p)^T$,
the second term of formula \cref{Psifla} for $\Psi_\xi$ is clearly
affine along geodesics, so that it suffices to establish the convexity
of the first term.

%For any $\xi\in\sqrt{-1}\Lie T$
%and $\s\in\BB(H^0(X,L^p))^T$, recall from \cref{thetaxiprop} that
%$\theta_{h_\s}(\xi)\in\cinf(X,\R)$, so that using
%\cref{pullbackprop} and formula
%\cref{dtetxi*h}, we compute
Now as explained in the proof of \cite[Lem.\,7.2]{BBGZ13},
via formula \cref{holpot} the results of \cite{Ber09b}
imply that the first term of formula \cref{Psifla}
is convex along geodesics, and strictly convex along geodesics
except those generated by $A\in\Herm(\C^{n_p})^T$ such that
there exists $c>0$ and $\eta\in\Lie\Aut(X)$ with
%$L_\eta\in\Herm(\C^{n_p})^T$
%satisfying
\begin{equation}
\dt\Big|_{t=0}\,h_{e^{tA}\s}=(\theta_{h_\s}(\eta)+c)\,h_\s\,.
\end{equation}
By \cref{thetaxiprop}, the fact that $\theta_{h_\s}(\eta)\in\cinf(X,\R)$
implies that $L_{J\eta}h_\s=0$, so that
$L_\eta\in\cL(H^0(X,L^p),H_\s)^T$. By \cref{dualprop} and formula
\cref{sig=theta}, we then have $\sigma_\s(A)=\sigma_\s(L_\eta+c)$.
On the other hand, the image of the
Kodaira map \cref{Kod} is not contained in any
proper projective subspace of $\mathbb{P}(H^0(X,L^p)^*)$ by definition,
hence following for instance \cite[Prop.\,4.8]{IKPS19}, we know that
the Berezin symbol
$\sigma_\s:\cL(H^0(X,L^p),H_\s)^T\to\cinf(X,\R)$ is injective
(see also \cite[\S 3]{Has18}).
This concludes the proof.
%Using Kodaira embedding and
%\todo{faire un lemme sur cette utilisation de Kod}
%formula \cref{APi=sigmaAPi}, this implies in turn
%\begin{equation}
%\begin{split}
%A-L_\xi &=\int_X\left(A-L_\xi\right)\,\Pi_{\s}(x)\,d\alpha_\s(x)\\
%&=\int_X\left(\sigma_\s(A)-\sigma_\s(L_\xi)\right)
%\,\Pi_{\s}(x)\,d\alpha_\s(x)=0\,.
%\end{split}
%\end{equation}
\end{proof}

\noindent
\textbf{Proof of uniqueness in \cref{mainth}.}
First note that, if $h_p\in\Met^+(L^p)^T$ is anticanonically balanced
relative to $\xi_p\in\Lie\Aut(X)$, then for any $\phi\in\Aut_0(X)$,
the pullback metric $\phi^*\,h_p$ is anticanonically balanced
relative to $\phi^*\,\xi_p$. On the other hand,
following the argument of Tian-Zhu in the proof of \cite[Th.\,3.2]{TZ02},
if $\xi,\,\til{\xi}\in\Lie\Aut(X)$ are such that $J\xi,\,J\til{\xi}$
are in the Lie algebra of maximal compact subgroups
$K,\,\til{K}\subset\Aut_0(X)$,
then there exists $\phi\in\Aut_0(X)$ such that $\phi^*\,J\til{\xi}\in K$.
Using \cref{Futpropintro},
we are then reduced to show uniqueness up to $\Aut_0(X)$
of anticanonically balanced metrics
relative to the vector field $\xi_p\in\sqrt{-1}\Lie K$ of \cref{xipmin},
for a fixed maximal compact subgroup $K\subset\Aut_0(X)$.

Let now $h_p,\,\til{h}_p\in\Met^+(L^p)$ be anticanonically balanced
metrics
relative to $\xi_p\in\sqrt{-1}\Lie K$, let $T\subset K$ be the closure of
the subgroup generated by $J\xi_p$
and let $\s_p,\,\til{\s}_p\in\BB(H^0(X,L^p))^T$ be orthonormal with respect
to $L^2(h^p),\,L^2(\til{h}^p)$.
\cref{momentbal} and \cref{Psidef} imply that
$[\s_p],\,[\til{\s}_p]\in\BB(H^0(X,L^p))^T/U(n_p)^T$ are both critical points
of $\Psi_{\xi_p}$, so that $\Psi_{\xi_p}$ cannot be strictly convex along
the geodesic joining them. \cref{BBGZ} then implies
that there exists $c>0$ and $\eta\in\Lie\Aut(X)$ with
$\eta\in\cL(H^0(X,L^p),H_{\s_p})^T$
such that $[\til{\s}_p]=[ce^{L_\eta}\s_p]$. Then
$\phi_{\eta}^*\,\xi_p=\xi_p$,
and using the characterization
\cref{relbaldef2}, \cref{pullbackprop} implies
\begin{equation}
\om_{\til{h}_p}=\om_{e^{L_{\xi_p/2p}}\til{\s}_p}
=\om_{ce^{L_\eta}e^{L_{\xi_p/2p}}\s_p}
=\phi^*_{\eta}\,\om_{e^{L_{\xi_p/2p}}\s_p}=\phi^*_{\eta}\,\om_{h_p}\,.
\end{equation}
This concludes the proof.

\qed

%\section{Conflict of interest statement}
%
%On behalf of all authors, the corresponding author states that there is no conflict of interest.

%%\section{fantome}
%\bibliographystyle{amsplain}
%%\nocite{*}
%\bibliography{bbstat}

\providecommand{\bysame}{\leavevmode\hbox to3em{\hrulefill}\thinspace}
\providecommand{\MR}{\relax\ifhmode\unskip\space\fi MR }
\providecommand{\MRhref}[2]{%
  \href{http://www.ams.org/mathscinet-getitem?mr=#1}{#2}
}
\providecommand{\href}[2]{#2}

\Addresses
\end{document}